%% file: templateArxiv.tex
\documentclass{article}

\usepackage{PRIMEarxiv}

\usepackage[utf8]{inputenc} 
\usepackage[T1]{fontenc}    
\usepackage{hyperref}       
\usepackage{url}            
\usepackage{booktabs}       
\usepackage{amsfonts}       
\usepackage{nicefrac}       
\usepackage{microtype}      
\usepackage{lipsum}
\usepackage{fancyhdr}       
\usepackage{graphicx}       
\graphicspath{{media/}}     

\usepackage{graphicx}%
\usepackage{multirow}%
\usepackage{amsmath,amssymb,amsfonts}%
\usepackage{amsthm}%
\usepackage{bbm}
\usepackage{mathrsfs}%
\usepackage[title]{appendix}%
\usepackage{xcolor}%
\usepackage{textcomp}%
\usepackage{manyfoot}%
\usepackage{booktabs}%
\usepackage{algorithm}%
\usepackage{algorithmic}
\usepackage{listings}%

\usepackage[utf8]{inputenc} 
\usepackage[T1]{fontenc}    
\usepackage{hyperref}       
\usepackage{url}            
\usepackage{booktabs}       
\usepackage{amsfonts}       
\usepackage{nicefrac}       
\usepackage{microtype}      
\usepackage{lipsum}		
\usepackage{natbib}
\usepackage{doi}
\usepackage{xcolor}         
\usepackage{array}
\usepackage{pifont}
\usepackage{bbding}
\usepackage{makecell}

\newtheorem{theorem}{Theorem}[section]
\newtheorem{lemma}{Lemma}[section]

\newtheorem{assumption}{Assumption}
\newtheorem{definition}{Definition}
\newtheorem{remark}{Remark}

\input{math_command}

\allowdisplaybreaks[4]

\title{Gradient Tracking for High Dimensional Federated Optimization}

\author{
  Jiadong Liang \\
  School of Mathematical Sciences \\
  Peking University \\
  Beijing\\
  \texttt{jdliang@pku.edu.cn} \\
   \And
  Yang Peng \\
  School of Mathematical Sciences \\
  Peking University \\
  Beijing\\
  \texttt{pengyang@pku.edu.cn} \\
  \AND
  Zhihua Zhang \\
  School of Mathematical Sciences \\
  Peking University\\
  Beijing \\
  \texttt{zhzhang@math.pku.edu.cn} \\
}

\begin{document}
\maketitle

\begin{abstract}
In this paper, we study the (decentralized) distributed optimization problem with high-dimensional sparse structure. Building upon the FedDA algorithm, we propose a (Decentralized) FedDA-GT algorithm, which combines the {\bf gradient tracking} technique. It is able to eliminate the heterogeneity among different clients' objective functions while ensuring a dimension-free convergence rate.
Compared to the vanilla FedDA approach, (D)FedDA-GT can significantly reduce the communication complexity, from $\gO(s^2\log d/\varepsilon^{3/2})$ to a more efficient $\gO(s^2\log d/\varepsilon)$.
In cases where strong convexity is applicable, we introduce a multistep mechanism resulting in the Multistep ReFedDA-GT algorithm, a minor modified version of FedDA-GT. This approach achieves an impressive communication complexity of $\gO\left(s\log d \log \frac{1}{\varepsilon}\right)$ through repeated calls to the ReFedDA-GT algorithm.
Finally, we conduct numerical experiments,  illustrating that our proposed algorithms enjoy the dual advantage of being dimension-free and heterogeneity-free. 
\end{abstract}


\section{Introduction}
\input{intro}

\section{Problem Formulation} We will introduce the basic problem formulation considered in this work. In Section~\ref{sec:asp}, we give the assumptions which will be used in our theoretical analysis.

Given a set of clients $\gM = \{1,2,\cdots, M\}$, they exchange information with each other according to an undirected and connected communication graph $G=(V,E)$, where the nodes $V = \gM$ and the edges $E \in \gM \times \gM$. Every client has their local objective function $f_m(\vw)$ with the form of $\EB_{\xi_m \sim \gD_m} F(\vw; \xi_m)$, where $\gD_m$ is the local data distribution of the $m$-th client. Our goal is to find the optima $\vw^*$ which satisfies that,
\[
\vw^* = \argmin\limits_{\vw \in \RB^d} f(\vw) \left(:= \frac{1}{M}\ssum{m}{1}{M}f_m(\vw)\right).
\]
Based on this, we are primarily interested in situations where the dimension of the problem itself is very large. Additionally, we assume that $\vw^*$ is $s$-sparse ($s \ll d$), i.e., $\norm{\vw^*}_0 \le s$ with $\norm{\cdot}_0$ representing the number of non-zero elements of the corresponding vector. This assumption is common in high-dimensional statistics and optimization problems \citep{agarwal2012stochastic,juditsky2014unified,juditsky2023sparse}.
It is worth noting that the sparsity of $\vw^*$ does not imply that the minimum points of each local objective function are also sparse. In many cases, due to the biases of individual entities, solving a single problem is often more complex than solving a global problem. We will illustrate this point through numerical experiments in Section~\ref{sec:experiment}.

For each client $m ~(= 1, \cdots, M)$ and any parameter $\vw^m$, it can interact with the environment to obtain a random oracle denoted as $\nabla F(\vw^m, \xi^m)$, which serves as an unbiased estimation of $\nabla f_m(\vw^m)$.

In the aforementioned problem setting, our focus is to design an algorithm that can approximate $\vw^*$ with as low sample complexity (i.e., the number of interactions each client makes with the environment to request random oracles) and communication complexity (i.e., the number of communications between different clients) as possible.


\textbf{Notation}

Throughout the remainder of this paper, we utilize the notation $[n] = \{1, \ldots, n\}$ for brevity and employ $\inner{\cdot}{\cdot}$ to denote the inner product between two vectors of the same dimension. Given a vector $\vx\in \RB^d$ and $i\in [d]$, $[\vx]_i$ denotes the $i$-th coordinate of $\vx$. 
We define $\mathbf{1}$ as the vector of all ones with an appropriate dimension, i.e., $\mathbf{1} := (1, \cdots, 1)^\top$. Additionally, for $M\in \ZB_+$, $\rmJ_M \in \RB^{M\times M}$ is defined as $\rmJ_M = \frac{1}{M}\mathbf{1}\mathbf{1}^\top$. 

For any $p \ge 1$ and any $\vx\in \RB^d$, we express $\norm{\vx}_p := \left(\ssum{i}{1}{d} |[\vx]_i|^p\right)^{1/p}$. In particular, $\norm{\vx}_\infty = \max\limits_{i\in [d]} |[\vx]_i|$. Given a matrix $\rmA = [\va_1, \ldots, \va_{d_2}] \in \RB^{d_1 \times d_2}$, we define $\norm{\rmA}_{p,q}$ as $\left(\ssum{i}{1}{d_2} \norm{\va_i}_p^q\right)^{1/q}$. 

For any two sequences of positive numbers $\{a_n\},\{b_n\}$, both notations $a_n = \gO(b_n)$ and $a_n \precsim b_n$ convey that there exists a universal constant $C>0$ such that $a_n \le Cb_n$ for all $n \in \ZB_+$. And the notation $a_n = \tilde{\gO}(b_n)$ represents that there is a universal constant $C>0$ such that $a_n \le Cb_n \log b_n$.

\subsection{Basic Assumptions}\label{sec:asp}

Next, we introduce and discuss the assumptions that underlie our analysis.
\begin{assumption}[Convexity]\label{asp:convex}
For each $m\in [M]$, the local objective function $f_m$ is convex, that is,
\[
f_m(\vw) + \inner{\nabla f_m(\vw)}{\vv - \vw} \le f_m(\vv), \quad \forall \vw,\vv \in \RB^d.
\]
\end{assumption}

\begin{assumption}[Smoothness]\label{asp:smooth}
For each $m\in [M]$, the local objective function $f_m$ satisfies
\[
\norm{\nabla f_m(\vw) - \nabla f_m(\vv)}_\infty \le L \norm{\vw - \vv}_1, \quad \forall \vw, \vv \in \RB^d.
\]
\end{assumption}

\begin{assumption}[Sub-Gaussian]\label{asp:grad_subG}
For each $m\in [M]$ and $i\in [d]$, and an arbitrary vector $\vw \in \RB^d$, we use $\zeta_i^{m}(\vw)$ to denote the $i$-th coordinate of $\nabla F(\vw, \xi_m) - \nabla f_m(\vw)$. And we assume that the following holds,
\[
\EB_{\xi_m \sim \gD_m} \exp \{\lambda \zeta_i^{m}(\vw)\} \le \exp \left\{\frac{\lambda^2 \sigma^2}{2}\right\}, \; \forall \lambda \in \RB.
\]
\end{assumption}

Assumption~\ref{asp:smooth} and Assumption~\ref{asp:grad_subG} are common in sparse optimization literatures \citep{agarwal2012stochastic,juditsky2014unified,juditsky2023sparse}. And \citet{yuan2021federated} proposed the same smooth assumption when the primal norm in their work is assigned as $\norm{\cdot}_1$. However, in \citep{bao2022fast}, the authors proposed the smoothness assumption in $\gL_2$-norm and the light-tailed assumption of the whole noise vector, both of which are relatively more restricted than the above assumptions.

In addition to the aforementioned assumptions, this paper examines distributed optimization problems associated with an objective function under the conditions specified in Assumption~\ref{asp:lsc}. In this context, we adopt the assumption of local strong convexity to encompass situations where the convexity of the objective function decreases as the examination location moves away from the origin. This is particularly relevant to objective functions such as logistic regression.

In some statistical literature, researchers also consider restricted strong convexity as an alternative to local strong convexity \citep{buhlmann2011statistics,raskutti2011minimax,agarwal2012stochastic}. 
This adaptation is specifically designed to accommodate scenarios where the empirical loss function, being a finite sum, is considered with limited datasets rather than making assumptions about the population risk, which is the focus of our consideration.

While it is possible to extend our algorithm by incorporating regularization terms to make it compatible with more general assumptions, such extensions might deviate from the central theme of this paper. Therefore, we do not provide a detailed discussion on this matter here.

\begin{assumption}[Locally Strong Convexity]\label{asp:lsc}
    For each $m\in [M]$, the local objective function $f_m$ satisfies a $Q$-local form of strong convexity, i.e., there is a constant $\mu(Q)>0$ such that
    \[
    f_m(\vv) \ge f_m(\vw) + \inner{\nabla f_m(\vw)}{\vv - \vw} + \frac{\mu(Q)}{2}\norm{\vv - \vw}_2^2,
    \]
    for any $\vw, \vv \in \RB^d$ with $\norm{\vw}_1^2 \le Q$ and $\norm{\vv}_1^2 \le Q$.
\end{assumption}






To facilitate a more convenient characterization of the entire decentralized communication system, we can equivalently represent the communication graph $G$ as a gossip matrix $\rmU \in \RB^{M\times M}$ \citep{nedic2009distributed}. The definition of the gossip matrix is presented as follows.

\begin{definition}[Gossip Matrix]
    We define $\rmU = [u_{ij}] \in \RB^{M\times M}$ as a gossip matrix if it satisfies the following properties:
    \begin{itemize}
        \item (Symmetry) $\rmU^\top = \rmU$,
        \item (Non-negative entries) $u_{ij} \ge 0,~\forall i\neq j$,
        \item (Doubly stochastic) $\rmU\mathbf{1} = \mathbf{1}$ and $\mathbf{1}^\top \rmU = \mathbf{1}^\top$.
    \end{itemize}
\end{definition}

Building upon the above definition and the Perron-Frobenius theorem, it is evident that $\lim_{n\to \infty}\pi^\top \rmU^n = \mathbf{1}^\top,~\forall \pi \ge 0 \text{ and }\pi^\top \mathbf{1} = 1$. The subsequent assumption provides a quantitative description of the limitations and plays a crucial role in analyzing the communication complexity.

\begin{assumption}[Mixing Gossip]\label{asp:mx_gossip}
    Consider the second-largest eigenvalue of a matrix $\rmU$ in terms of absolute value (denoted $\sigma_2(\rmU)$). For ease of exposition, we assume here that $0 \leq \sigma_2(\rmU) < 1$. We refer to $1-\sigma_2(\rmU)$ as the spectral gap of $\rmU$.
\end{assumption}

It is noteworthy that when the gossip matrix $\rmU = \rmJ_M$, the communication graph becomes a fully connected graph, and it can be regarded as the centralized scenario. In this case, the spectral gap $1 - \sigma_2(\rmU) = 1$. The following result is celebrated, and it can help us to deal with high-dimensional noise.

\subsection{Mirror Map}
To address stochastic high-dimensional noise, the mirror map serves as a frequently employed tool in optimization. It establishes a connection between two distinct metric spaces. Specifically, given a strictly convex differentiable function $h: \RB^d \to \RB$, if $\nabla h$ diverges at infinity and its range is $\RB^d$, then $\nabla h$ can be viewed as the mirror map from the primal space to the dual space. Correspondingly, the $\nabla^{-1} h = \nabla h^*$ represents the inverse map from the dual space to the primal space, where $h^*$ is the Fenchel conjugate of $h$. Shifting the process of gradient updates to the dual space gives rise to the well-known mirror descent algorithm.

Let $2 \leq p$, $1< q\leq 2$ and $1/p+1/q=1$. Throughout this paper, we define $h(\vw)=\frac{1}{2(q-1)}\|\vw\|_q^2$ with its conjugate $h^*(\vz)=\frac{1}{2(p-1)}\|\vz\|_p^2$. The mirror map is $\nabla h(\vw) = \frac{1}{q-1}\|\vw\|_q^{2-q} \vw^{q-1}$ with its inverse map $\nabla^{-1}h(\vz) = \frac{1}{p-1}\|\vz\|_p^{2-p}\vz^{p-1}$. For a vector $\vx\in \RB^d$ and a positive number $\alpha > 0$, $\vx^\alpha$ is defined as 
$\left(\sgn([\vx]_i)\abs{[\vx]_i}^\alpha\right)_{i=1}^d$.

\begin{lemma}
\label{lem:strong_cvx_mirror_map}
    Let $2 \leq p$, $1< q\leq 2$, and $1/p+1/q=1$. Then $h$ is 1-strongly convex with respect to $\|\cdot\|_q$, and $h^*$ is 1-smooth with respect to $\|\cdot\|_p$.
\end{lemma}
\begin{proof}[Proof of Lemma \ref{lem:strong_cvx_mirror_map}]
    The strong convexity of $h$ follows directly from Lemma 17 of \cite{shalev2007online}, and the smoothness of $h^*$ follows from Lemma 15 (3) of \cite{shalev2007online}.
\end{proof}





\section{Algorithm}
\label{sec:algorithm}

In this section, we present our algorithm in Algorithm~\ref{alg:dscaffoldda} that we call ``Decentralized Federated Dual Averaging with Gradient Tracking.''  We will also provide a brief discussion on how the novel algorithm utilizes gradient tracking techniques to achieve lower communication complexity. 

Some readers might have already noticed that in Algorithm~\ref{alg:dscaffoldda}, we have introduced a Boolean input variable called ``\texttt{gradient tracking}.'' When this parameter is set to \textbf{False}, Algorithm~\ref{alg:dscaffoldda} reduces to a standard decentralized federated dual-averaging algorithm. Intuitively, this implies that the effect of gradient tracking and other aspects of the algorithm (such as the update mechanism of dual averaging) on the final complexity optimization are independent of each other. Given this observation, we prioritize analyzing the decentralized federated averaging algorithm first and then explore how the addition of the gradient tracking term can further improve its performance.
\input{Algorithms/decentral_scaffolDA_alg}
\subsection{Decentralized Federated Dual Averaging}
DFedDA can be viewed as a direct extension of FedDA, with the main difference being that DFedDA eliminates the need for a central server to aggregate information from each node. Specifically, at each client node $m$, two sequences, $\vw^m$ and $\vz^m$, are maintained. The former resides in the original space and approximates $\vw^*$, while the latter exists in the dual space and can be considered as an accumulation of gradient information. The relationship between them is established through the provided mirror map $\nabla h: \text{primal} \to \text{dual}$. 

During the $r$-th round and the $k$-th local update process, $\vz_{r,k}^m$ is mapped back to the original space using $\nabla^{-1}h$, generating an unbiased estimation of the gradient $\nabla f(\vw_{r,k}^m)$ denoted as $\nabla F(\vw_{r,k}^m, \xi_{r,k}^m)$. This estimation is used for updating $\vz_{r,k}^m$. Following the completion of the local updates, each client receives parameter information from its neighboring clients based on the gossip matrix and computes the average to obtain the initial parameters $\vz_{r+1,0}^m$ for the next round. It is obvious that when the gossip matrix corresponds to a complete communication graph, DFedDA is equivalent to FedDA.

\subsection{Participation of Gradient Tracking} \label{sec:parti of GT}

Similar to methods like FedAvg \citep{karimireddy2019scaffold} or Local SGD \citep{li2019communication}, for DFedDA, the greatest hindrance to its communication efficiency improvement lies in the inherent heterogeneity of the distributed optimization problem itself. The essence of the gradient tracking approach is to add a corresponding correction term to each stochastic gradient to counteract the impact of heterogeneity. The specific procedure is depicted in the 10th line of Algorithm~\ref{alg:dscaffoldda}. The most crucial tracking term $\vc_{r}^m$ is defined recursively as follows.
\begin{equation}\label{eq:gradient tracking step}
\begin{aligned}
    &\Delta_r^m = \frac{1}{K\eta_c} (\vz_{r,K}^m - \vz_{r,0}^m)\\
    &\vc_{r+1}^m = \vc_r^m + \Delta_r^m - \ssum{j}{1}{M}u_{jm}\Delta_r^j.
\end{aligned}
\end{equation}

To better comprehend the role of gradient tracking, we contemplate the centralized and noiseless scenario. In this setting, we consistently obtain accurate gradient information on each occasion (there is no randomness involved) and $\rmU = \frac{1}{M} \mathbf{1}\mathbf{1}^\top$. Accordingly, $c_r^m$ can be displayed more precisely,
\begin{align*}
    &\Delta_r^m = \frac{1}{K\eta_c} (\vz_{r,K}^m - \vz_{r,0}^m)
    = - \frac{1}{K}\ssum{k}{0}{K-1}\left(\nabla f_m(\vw_{r,k}^m) + \vc_r^m\right) = - \frac{1}{K}\ssum{k}{0}{K-1}\nabla f_m(\vw_{r,k}^m) - \vc_r^m;\\
    &\vc_{r+1}^m = \vc_r^m + \Delta_r^m - \frac{1}{M}\ssum{j}{1}{M}\Delta_r^j
    \overset{(a)}{=} - \frac{1}{K}\ssum{k}{0}{K-1}\nabla f_m(\vw_{r,k}^m) + \frac{1}{MK}\ssum{m}{1}{M}\ssum{k}{0}{K-1}\nabla f_m(\vw_{r,k}^m).
\end{align*}
We leverage the fact $\ssum{m}{1}{M} \vc_r^m = 0$ to establish the validity of equation $(a)$.

At this juncture, it becomes apparent that $\vc_r^m$ shares the same form as the correction term for gradients in variance reduction methods used in finite-sum optimization problems \citep{johnson2013accelerating,defazio2014saga}. With this in view, we can perceive gradient tracking as a form of variance reduction employed to address heterogeneity.

\section{Rate of Convergence under Convexity}\label{sec: conv rate under convexity}
In the following section, we delve into the convergence analysis of the DFedDA and DFedDA-GT algorithms under the assumption of a convex objective function. We aim to showcase the convergence rates achieved by both algorithms in this scenario.

\subsection{Convergence Results on Decentralized FedDA (with Gradient Tracking)}
The following is an unofficial statement for the convergence rate of DFedDA(-GT).
\begin{theorem}\label{thm:cvrt_decentralized_alg}
    Suppose Assumption~\ref{asp:convex} - \ref{asp:grad_subG} and Assumption~\ref{asp:mx_gossip} hold. Let $p = 2\log d$, $\tau = \frac{\log 4M}{2\log (1/\sigma_2(\rmU))}\vee 1$, and $\hat{\vw}^m := \frac{1}{R}\ssum{r}{0}{R-1}\vw_{r,0}^m$, $\eta_s = 1$.
    
    \textbf{\textsc{DFedDA:}} If the input \texttt{gradient tracking} in Algorithm~\ref{alg:dscaffoldda} takes \textbf{False}, then
    there exists a choice of $\eta_c$ such that
    \[
    \frac{1}{M} \ssum{m}{1}{M} \EB\{f(\hat{\vw}^m) - f(\vw^*)\} \lesssim \frac{\tau L h(\vw^*)}{R} + \frac{(\tau^2 L \gE_M^* h(\vw^*)^2)^{1/3}}{R^{2/3}} + \frac{(\tau h(\vw^*))^{1/2}(\log d)^2 \sigma}{(RK)^{1/2}},
    \]
    where $\gE_M^*$ is defined by $\gE_M^* := \frac{1}{M}\ssum{m}{1}{M}\norm{\nabla f_m(\vw^*)}_p^2$;
    
    \textbf{\textsc{DFedDA-GT:}} If the input \texttt{gradient tracking} in Algorithm~\ref{alg:dscaffoldda} takes \textbf{True}, then
    there exists a choice of $\eta_c$ such that
    \[
    \frac{1}{M} \ssum{m}{1}{M} \EB\{f(\hat{\vw}^m) - f(\vw^*)\} \lesssim \frac{\tau^2 L h(\vw^*)}{R} + \frac{\tau h(\vw^*)^{1/2}(\log d)^2 \sigma}{(RK)^{1/2}}.
    \]
\end{theorem}

\begin{remark}
    Here, $\tau$ represents the mixing time of the gossip matrix, which reflects the single-round communication efficiency of the communication network (with the highest efficiency achieved in the complete graph). Interestingly, in high-dimensional optimization problems, the dependence on $\tau$ in the convergence rate we obtain is consistent with algorithms proposed by \citet{liu2023decentralized} in non-high-dimensional scenarios. This alignment is not trivial to prove, and we will provide further explanations and clarifications on this matter later on.
\end{remark}
\begin{remark}
    It is important to emphasize that our algorithm cannot guarantee that the estimate $\hat{\vw}$ is sparse. This limitation arises because, under the sole assumption of convexity, existing optimization methods struggle to accurately characterize the distance between $\hat{\vw}$ and $\vw^*$. Consequently, if we forcibly sparsify $\hat{\vw}$ into $\hat{\vw}_s$, the control over $f(\hat{\vw}_s) - f(\vw^*)$ becomes challenging. We recognize this issue as one of the directions for future research.
\end{remark}

\subsection{Comparison with Related Convergence Results}
Theorem~\ref{thm:cvrt_decentralized_alg} in our paper examines the convergence rates of two decentralized algorithms: DFedDA and DFedDA-GT. When disregarding the dependence on $\log d$, both algorithms exhibit convergence rates similar to those of non-high-dimensional distributed algorithms that do not utilize gradient tracking techniques, such as Theorem 1 in \citep{karimireddy2019scaffold} and Corollary 1 in \citep{li2019communication}.

Furthermore, in the context of high-dimensional optimization, we compare our work with two relevant papers. Firstly, we consider FedDA proposed by \citet{yuan2020federated}. Our DFedDA achieves convergence rates that entirely cover FedDA's convergence rates for large step sizes, as shown in Theorem 4.2 of \citep{yuan2020federated}. Additionally, when FedDA adopts small step sizes, its convergence rate, described in Theorem 4.1 of \citep{yuan2020federated}, coincides with that of our DFedDA-GT. However, we note that their proofs raise certain concerns, which we will address in Section~\ref{sec:related_work}.

Secondly, we compare our approach with Fast-FedDA introduced by \citet{bao2022fast}, which focuses on strong convexity in the loss functions at each client. 
Although Fast-FedDA employs a more sophisticated mirror update technique for strong convexity, its convergence rate (Theorem 2.1 in \citep{bao2022fast}) aligns with our DFedDA-GT's convergence rate, limited to the convex settings only.
\section{An Overview of the Proof Procedure}\label{sec:proof sketch}
To underscore the differences from previous research, in this section, we briefly outline the proof strategies employed and enumerate some of the challenging aspects encountered during the analysis.

\subsection{How Does Mirror Map Works}
In Algorithm~\ref{alg:dscaffoldda}, the mirror map $\nabla h$ has the primitive function $h(\vw)=\frac{1}{2(q-1)}\|\vw\|_q^2$ with $q = \frac{2\log d}{2\log d - 1}$. The selection of this mirror map was originally proposed by \citet{shalev2009stochastic} and later became one of the fundamental techniques in online high-dimensional optimization, continuously used and extended in subsequent works such as \citep{agarwal2012stochastic,juditsky2023sparse,juditsky2014unified}. The motivation behind this approach is as follows: when using vanilla SGD, the convergence analysis is typically conducted in Euclidean space. However, when analyzing the variance of the stochastic noise $\EB\norm{\nabla F(\vw, \xi) - \nabla f(\vw)}_2^2$, its magnitude scales linearly with the dimension $d$, which affects the final convergence rate. On the contrary, by applying the mirror map $\nabla^{-1} h$, we can control the variance of the gradient noise (denoted as $\zeta$) to be of the order of a constant, without losing the original gradient information. Specifically, we can control the size of the gradient noise as $\EB \norm{\zeta}_{2\log d}^2 = \EB \left(\sum_{i=1}^{d}\abs{\zeta_i}^{2\log d}\right)^{\frac{1}{\log d}} \precsim d^{\frac{1}{\log d}} = e$, which is of a constant magnitude.

\subsection{One-Step Recursion Analysis}


The following lemma can be considered as a convergence guarantee for algorithms similar to dual averaging. There has been extensive research analyzing this result \citep{duchi2011dual, liu2020accelerated}, and it has been widely applied to complexity analyses of related algorithms \citep{duchi2012randomized,agarwal2012stochastic,juditsky2023sparse}. Here, we provide a brief restatement of it.

\begin{lemma}[Lemma 6 of \citet{liu2020accelerated}]\label{lem:innr_to_breg-norm}
Given a convex set $\gB$, a sequence of variables $\{\zeta_r\}$ and a positive sequence $\{a_r\}$. The vector sequence $\{\vv_r\}$ is defined by $\vv_r := \argmin\limits_{\vv \in \gB}\left\{ \ssum{i}{0}{r-1}a_i\inner{\zeta_i}{\vv} + h(\vv) \right\}$. And suppose $\vw^* \in \gB$. Then
\begin{align*}
    \ssum{i}{0}{r-1}a_i\inner{\zeta_i}{\vv_{i+1} - \vw^*} \le h(\vw^*) - \frac{1}{2}\ssum{i}{0}{r-1}\norm{\vv_i - \vv_{i+1}}_q^2. 
\end{align*}
\end{lemma}

Our main intuition of proof is as follows:
Consider a surrogate sequence $\bar{\vz}_{r,k}:= \frac{1}{M} \rmZ_{r,k}\mathbf{1}$ with $\rmZ_{r,k} := [\vz_{r,k}^1, \cdots, \vz_{r,k}^M] \in \RB^{d\times M}$, and its corresponding image in the primal space $\bar{\vw}_{r,k} = \nabla^{-1} h(\bar{\vz}_{r,k})$.
By using Lemma~\ref{lem:innr_to_breg-norm}, the one-step convergence results can be obtained

\begin{align*}
    &\frac{1}{MR}\ssum{r}{1}{R}\ssum{m}{1}{M} \EB [f_m(\bar{\vw}_{r,0}) - f_m(\vw^*)] \le \underbrace{\frac{h(\vw^*)}{RK\eta}}_{\gT_1} - \frac{1}{8RK\eta}\ssum{r}{1}{R} \EB \norm{\bar{\vw}_{r-1,0} - \bar{\vw}_{r,0}}_q^2\\
    &+ \underbrace{\frac{K\eta}{R}\ssum{r}{0}{R-1} \EB \norm{\bar{\vg}_r^\xi - \bar{\vg}_r}_p^2}_{\gT_2}
    + \underbrace{\frac{4L}{RMK}\ssum{r}{1}{R}\ssum{m}{1}{M}\ssum{k}{0}{K-1}\EB \norm{\bar{\vw}_{r-1,0} - \vw_{r-1,k}^m}_q^2}_{\gT_3}.
\end{align*}

The symbol $\bar{\vg}_r^{\xi}$ and $\bar{\vg}_r$ in the above formula stand for $\frac{1}{MK}\ssum{m}{1}{M}\ssum{k}{0}{K-1}\nabla F(\vw_{r,k}^m, \xi_{r,k}^m)$ and $\frac{1}{MK}\ssum{m}{1}{M}\ssum{k}{0}{K-1} \nabla f_m(\vw_{r,k}^m)$ respectively. Here and hereafter, when $\eta_s = 1$, we will denote $\eta_c$ as $\eta$. Note that except for the second term of the above formula, all the rest three terms are positive. We denote them as $\gT_1, \gT_2$ and $\gT_3$ respectively. The rest of this section is focused on how to bound these three terms.

\subsection{Bounds for Bias ($\gT_1$) and Variance ($\gT_2$)}
Concerning the initial term $\gT_1$, it emerges due to the disparity between $\vw^*$ and the initialization $\vw_0$, which can be reduced by increasing the number of global rounds.

As for the subsequent term $\gT_2$, it originates from the accumulation of noise during a single local update round. While FedAvg's analysis can readily handle this term, credit to the independence of distinct terms at separate time points within the same martingale difference sequence, the situation is different in the case of DFedDA. Here, it becomes necessary to bound the $\gL_p$ norm of the stochastic noise which is not amenable to being directly represented as a sum of inner products.

To address this challenge, we leverage the Burkholder-Davis-Gundy inequality,

\begin{lemma}[Burkholder-Davis-Gundy inequality]\label{lem:bdg_ineq}
Let $\{\iota_n\}_{n=1}^\infty$ be an $\{\gF_n\}_{n=1}^\infty$-martingale difference sequence in $\RB$. Suppose $\sup\limits_{i\ge 1}\EB |\iota_i|^p < \infty$. Then there is a constant $C$ (not depending on $p$) such that for each $N$,
\[
\EB \sup\limits_{n\le N} \left| \ssum{i}{1}{n}\iota_i \right|^p \le \frac{C^p p^{5p/2}}{(p-1)^p} \EB \left( 
\ssum{n}{1}{N} \EB[|\iota_n|^2| \gF_{n-1}] \right)^{p/2},
\]
where $\gF_0 := \{ \emptyset, \RB \}$.
\end{lemma}

Next, we can leverage Lemma~\ref{lem:bdg_ineq} to formulate Lemma~\ref{lem:bd_of_p_norm_2nd_mmt}, which proves instrumental in deriving a bound for $\gT_2$.

\begin{lemma} \label{lem:bd_of_p_norm_2nd_mmt}
    Suppose $\{\vv_i\}$ is an $\{\gF_i\}$-martingale in $\RB^d$. Then for any $p \ge 2$, we have
     \[\EB {\sup\limits_{1\le m \le n}}\norm{\ssum{i}{1}{m}\vv_i}_p^2 \le C^2 d^{2/p}p^3 n \sup\limits_{i\in [n], j\in [d]} \left\{ \EB \left| [\vv_i]_j \right|^p \right\}. \]
\end{lemma}

It's worth mentioning that in Lemma~\ref{lem:bdg_ineq}, we only require the noise introduced by the random oracle to have a bounded $p$-th moment, without assuming sub-Gaussianity. This implies that we effectively establish the validity of the bound for Theorem~\ref{thm:cvrt_decentralized_alg} under a more relaxed set of conditions.

\subsection{Bound for Deviation ($\gT_3$)}
Lastly, to conduct a proper analysis of $\gT_3$, a viable approach is to initiate the examination from each individual term within the summation. To put it more concretely, for any $r = 0, \cdots, R-1$ and $k = 0, \cdots, K-1$, we can get the following decomposition
\[
\frac{1}{M}\ssum{m}{1}{M} \EB \norm{\bar{\vw}_{r,0} - \vw_{r,k}^m}_q^2 \le \underbrace{\frac{2}{M}\ssum{m}{1}{M} \EB \norm{\vw_{r,k}^m - \vw_{r,0}^m}_q^2}_{\gT_3^{(1)}(r,k)} + \underbrace{\frac{2}{M}\ssum{m}{1}{M} \EB \norm{\bar{\vw}_{r,0} - \vw_{r,0}^m}_q^2}_{\gT_3^{(2)}(r,k)}.
\]
To put it candidly, we break down $\mathcal{T}_3$ into two distinct components: $\frac{4L}{RK}\sum_{r=0}^{R-1}\sum_{k=0}^{K-1} \mathcal{T}_3^{(1)}(r,k)$ and $\frac{4L}{RK}\sum_{r=0}^{R-1}\sum_{k=0}^{K-1} \mathcal{T}_3^{(2)}(r,k)$. And we define two quantities that play key roles in our analysis.
\begin{equation}\label{eq:define Hr&Er}
    \gH_r=\sum_{m=1}^M \EB\norm{\vz_{r,0}^m - \Bar{\vz}_{r,0}}_p^2, \quad \gE_r=\sum_{m=1}^M\sum_{k=0}^{K-1} \EB \norm{\vz_{r,k}^m - \vz_{r,0}^m}_p^2.
\end{equation}

The initial component accounts for the deviations resulting from the heterogeneity among the local devices during the local update process. To grasp this, we can scrutinize each $\mathbb{E}\|\vw_{r,k}^m - \vw_{r,0}^m\|_q^2$. By employing Lemma~\ref{lem:strong_cvx_mirror_map}, we can establish an upper bound using $\mathbb{E}\|\vz_{r,k}^m - \vz_{r,0}^m\|_p^2 = \eta^2\mathbb{E}\|\sum_{\tau=0}^{k-1}\nabla F(\vw_{r,\tau}^m, \xi_{r,\tau}^m)\|_p^2$. 
This kind of dependency also frequently appears in other federated learning works, 
and we establish the following lemma, whose proof is deferred to Section~\ref{sec:prf_for_local_dev_dfed}.

\begin{lemma}[One-step deviation of DFedDA No.1]\label{lem:local_dev_dfedda_1}
Under the assumptions of Theorem~\ref{thm:conv_rt_defedda}, for any $k= 0, \cdots, K-1;~ m\in [M]$, it is true that
    \begin{align*}
\gE_r &\le 4eC^2d^{2/p}p^4 \sigma^2 MK^2\eta^2 + 4ed^{4/p}L^2 K^3 \eta^2 \gH_r\\
&+ 16ed^{2/p}LMK^3\eta^2 \EB \{f(\Bar{\vw}_{r,0}) - f(\vw^*)\} + 8eK^3\eta^2 \ssum{m}{1}{M}\norm{\nabla f_m(\vw^*)}_p^2.
\end{align*}
\end{lemma}


The second term involved with $\gT_3^{(2)}(r,k)$ reflects the average deviation between each local parameter $\vw_{r,0}^m$ after the $r$-th communication round and the virtual parameter $\bar{\vw}_{r,0}$.
In the context of general decentralization problem, it's worth noting that the bound on $\gT_3^{(2)}$ is intricately connected to the network's spectral gap, $1 - \sigma_2(\rmU)$, as discussed in previous works \citep{yuan2016convergence,duchi2011dual,li2019communication,liu2023decentralized}. To extend similar conclusions in high-dimensional settings, we establish the following recursive inequality,

\begin{lemma}[One-step deviation of DFedDA No.2]\label{lem:local_dev_dfedda_2}
Under the assumptions of Theorem~\ref{thm:conv_rt_defedda}, if we take $\tau = \frac{\log 4M}{2\log (1/\sigma_2(\rmU))}\vee 1$, then for any $ r = 0, \cdots, R-1$,
    \begin{align*}
        \gH_r &\le \frac{1}{2}\gH_{r-\tau} + 40d^{4/p}\tau L^2K^2 \eta^2 \ssum{i}{r-\tau}{r-1}\gH_i + 40d^{4/p}\tau L^2K\eta^2 \ssum{i}{r-\tau}{r-1} \gE_i\\
        &+ 80d^{2/p}\tau LM K^2\eta^2 \ssum{i}{r-\tau}{r-1} \EB \{f(\Bar{\vw}_{i,0}) - f(\vw^*)\}\\
        &+ 40C^4 p^4 d^{2/p}\tau \sigma^2 MK\eta^2 + 40\tau^2 K^2\eta^2 \ssum{m}{1}{M} \norm{\nabla f_m(\vw^*)}_p^2.
    \end{align*}
\end{lemma}
The proof of Lemma~\ref{lem:local_dev_dfedda_2} is delayed to Section~\ref{sec:prf_for_local_dev_dfed}.

Let's briefly elaborate on the proof outline of this lemma. By iterating on $\gH_r$, we eventually arrive at an expression involving \(\rmU^t - \rmJ_M\), and it becomes crucial to control the specific operator norm of this matrix. For larger values of \(t\), we leverage the mixing properties of \(\rmU\), while for smaller \(t\), the aforementioned bounding technique becomes less refined. In such cases, we rely on the doubly stochastic properties of \(\rmU\) to establish the non-expansiveness of the operator. The diverse treatment strategy for different scenarios leads to the conclusions stated above.

An interesting observation is that, from Lemma~\ref{lem:local_dev_dfedda_1} and Lemma~\ref{lem:local_dev_dfedda_2}, it becomes apparent that $\gT_3^{(1)}$ and $\gT_3^{(2)}$ mutually serve as partial upper bounds for each other, exhibiting a \textbf{"spiraling"} recursive relationship. This intricate connection enables us, under specific weighted linear combinations, to simultaneously establish bounds for both through recursion, 
and we can obtain an upper bound for $\gT_3$.

By this point, we have derived upper-bound conclusions specific to each $\gT_i,~ i \in [3]$. Note that in practice, the surrogate sequence $\bar{\vw}_{r,0}$ is intractable caused by a decentralized mode of communication. To complete the convergence analysis of DFedDA, the last stage involves managing $\frac{1}{M}\sum_{m=1}^{M}\left\{f_m(\vw_{r,0}^m) - f_m(\bar{\vw}_{r,0})\right\}$. Given the $\gL_1$-smoothness of $f_m$, it's crucial to control $\gT_3^{(2)}$, a task that has already been accomplished in our preceding analysis step.

\subsection{The Role Played by Gradient Tracking}
In Section~\ref{sec:parti of GT}, 
we had a preliminary understanding about how gradient tracking works through a simple example.
In fact, there exists a substantial body of literature that has conducted in-depth analyses on the theoretical role of gradient tracking \citep{nedic2017achieving,karimireddy2019scaffold,koloskova2021improved,liu2023decentralized}. In this section, our focus lies more in conjunction with our earlier convergence analysis of DFedDA. We aim to explore how gradient tracking, in the context of higher dimensions, can yield substantial improvements in the algorithm's convergence rate.

Continuing with the notation from the preceding subsections, in simple terms, the incorporation of gradient tracking primarily leads to a significant reduction in the order of all $\EB \norm{\vz_{r,k}^m - \vz_{r,0}^m}_p^2$. To discern this effect, we can draw a comparison between the following three lemmas based on DFedDA-GT and Lemma~\ref{lem:local_dev_dfedda_1}, \ref{lem:local_dev_dfedda_2}.

Recall that $\gH_r$ and $\gE_r$ is defined by \eqref{eq:define Hr&Er} separately. Further, we define
\[\gL_r = \ssum{m}{1}{M}\EB \norm{\nabla f_m(\Bar{\vw}_{r,0}) + \vc_r^m}_p^2.
\]
Then we introduce three recursive consequences that bound $\gE_r,~ \gH_r$ and $\gL_r$ respectively.

\begin{lemma}[One-step deviation of DFedDA-GT No.1]\label{lem:local_dev_dfedda-gt_1}
    Under the assumption of Theorem~\ref{thm:conv_rt_defedda_gt}, for any $k=0,\cdots, K-1$, it is true that
    \begin{align*}
        \gE_r \le 4eC^2 d^{2/p} p^4 MK^2 \eta_c^2 \sigma^2 + 4ed^{4/p}L^2 K K^2\eta_c^2 \gH_r + 4eK^3 \eta_c^2 \gL_r,
    \end{align*}
    as long as $2d^{2/p}LK\eta \le 1$.
\end{lemma}

\begin{lemma}[One-step deviation of DFedDA-GT No.2]\label{lem:local_dev_dfedda-gt_2}
    Under the assumptions of Theorem~\ref{thm:conv_rt_defedda_gt}, if we take $\tau = \frac{\log 8M}{2\log (1/\sigma_2(\rmU))}\vee 1$, then for any $r= 0, \cdots, R$,
    \begin{align*}
        \gH_r &\le \frac{1}{2} \gH_{r-\tau} + 4d^{4/p} \tau L^2 K\eta^2 \ssum{i}{r-\tau}{r-1} \gE_i 
        +32 d^{4/p}\tau L^2 K^2 \eta^2 \ssum{i}{r-\tau}{r-1}\gH_i\\
        &+ 32 \tau K^2 \eta^2 \ssum{i}{r-\tau}{r-1}\gL_i + 32C^4 d^{2/p} p^4 \tau \sigma^2 MK\eta^2.
    \end{align*}
\end{lemma}

\begin{lemma}[One-step deviation of DFedDA-GT No.3]\label{lem:local_dev_dfedda-gt_3}
    Under the assumptions of Theorem~\ref{thm:conv_rt_defedda_gt}, if we take $\tau = \frac{\log 8M}{2\log (1/\sigma_2(\rmU))}\vee 1$, then for any $r= 0, \cdots, R$, we have
    \begin{align*}
        \gL_r &\le \frac{1}{4}\gL_{r-\tau} + 8d^{2/p}\tau ML^2 \ssum{i}{r-\tau}{r-1} \EB \norm{\Bar{\vw}_{i+1,0} - \Bar{\vw}_{i,0}}_1^2 + \frac{32C^4p^4 d^{2/p}\tau \sigma^2 M}{K}\\
        &+ \frac{16d^{4/p}\tau L^2}{K} \ssum{i}{r-\tau}{r-1} \gE_i + 
        16d^{4/p}\tau L^2 \ssum{i}{r-\tau}{r-1} \gH_i.
    \end{align*}
\end{lemma}

One can refer to Section~\ref{sec:prf_for_local_dev_dfed-gt} to check the proof of the above three one-step deviation lemmas.

In contrast to Lemma~\ref{lem:local_dev_dfedda_1}, the final term of Lemma~\ref{lem:local_dev_dfedda-gt_1}—most indicative of the influence of heterogeneity—transitions from $\mathcal{O}\left(K^3 \eta^2 \ssum{m}{1}{M}\norm{\nabla f(\vw^*)}_p^2\right)$ to $\mathcal{O}\left(K^3 \eta^2 \gL_r\right)$. A similar transformation in outcomes is observed from Lemma~\ref{lem:local_dev_dfedda_2} to Lemma~\ref{lem:local_dev_dfedda-gt_2}. Furthermore, by means of Lemma~\ref{lem:local_dev_dfedda-gt_3}, we can establish that $\gL_r = o\left( \frac{1}{M}\ssum{m}{1}{M}\norm{\nabla f(\vw^*)}_p^2\right)$. This underscores that, even in high-dimensional scenarios, gradient tracking techniques remain effective in mitigating the impact of heterogeneity.

\section{Enhancing Algorithm Performance under Local Strong Convexity}

Until now, our algorithm design and theoretical analysis have centered around the assumption of a convex population objective function. In this section, we leverage the more stringent assumption of local strong convexity. By harnessing the ReFedDA-GT algorithm, a minor modified version of FedDA-GT, across multiple stages under this heightened assumption, we achieve several notable enhancements compared to algorithms operating under convexity assumptions:

\begin{enumerate}
    \item Reduced sample complexity while maintaining the same accuracy requirements.
    \item Substantial reduction in communication complexity with the same accuracy requirements.
    \item Ability to estimate $\vw^*$ with an estimate $\hat{\vw}$ that possesses sparsity, while still achieving the desired accuracy.
\end{enumerate}

In Section~\ref{sec:alg for str cvx}, we introduce the modifications to our algorithm, and in Section~\ref{sec:conv rt of str cvx}, we present our theoretical findings. We conclude in Section~\ref{sec:prf skt of str cvx} with a concise overview of our proof strategy. Notably, for the sake of clarity and brevity, our focus in this section primarily centers around the centralized scenario. This entails the assumption that $\rmU = \rmJ_M$. For the results in the decentralized scenario, we can draw upon the analogous algorithms discussed in Section~\ref{sec: conv rate under convexity} and Section~\ref{sec:proof sketch}.

\subsection{Improve the Convergence Performance by Restarting}\label{sec:alg for str cvx}

Our algorithm can be divided into two primary layers. The first layer encompasses the Restricted FedDA-GT (ReFedDA-GT) algorithm \ref{alg:recfdagtr}, where we introduce minor modifications to the FedDA-GT algorithm. The second layer consists of the Multistep ReFedDA-GT algorithm \ref{alg:multirecfda}, which involves iterative utilization of ReFedDA-GT, contributing to enhanced efficiency. In the following, we delve into the specifics of these algorithms.

\subsubsection{Restricted FedDA-GT}
Based on FedDA-GT, we leverage the assumption of local strong convexity to refine our algorithm.

\paragraph{Dual-to-Primal Mapping}
In comparison to FedDA-GT, we replace the mirror map $\nabla^{-1} h$ with the proximal function \eqref{eq:prox function}, which involves solving a sub-optimization problem within a norm ball centered at $\vw_0$ with radius $\sqrt{Q}$. 
\begin{equation}\label{eq:prox function}
\begin{aligned}
    \textbf{Prox}_{h}(\vw_0, \vg; Q, \norm{\cdot}) := \argmin\limits_{\norm{\vw - \vw_0}^2 \le Q} \left\{\inner{\vw - \vw_0}{\vg} + h(\vw - \vw_0)\right\}.
\end{aligned}
\end{equation}
This change is motivated by the local validity of the strong convexity assumption for the population objective function.

Moreover, we draw a connection to the mirror map: When the constraint radius $\sqrt{Q} = \infty$, employing the first-order conditions and considering that $\vg$ is frequently chosen as $\vz_0 - \vz$ in Algorithm~\ref{alg:recfdagtr}, we deduce $\vw  = \vw_0 + \nabla^{-1} h (\vz - \vz_0)$. This configuration resembles the form of the mirror map $\nabla^{-1} h $, with the distinction that it is shifted from $(0,0)$ to $(\vz_0, \vw_0)$.

It's worth noting the computational efficiency of the proximal function, as thoroughly demonstrated in Appendix A of \citep{agarwal2012stochastic}.

\paragraph{Gradient Tracking}
ReFedDA-GT employs $\vc_{r}^m,~ \vc_{r}$ as tracking terms, governed by the updating relationship in Equation~\eqref{eq:str gradient tracking}. 
\begin{equation}\label{eq:str gradient tracking}
\begin{aligned}
\vc_{r+1}^m &= \vc_r^m - \vc_r - \frac{1}{K\eta_{c}}(\vz_{r,K}^m - \vz_{r,0}^m);\\
\vc_{r+1} &= \vc_r + \frac{1}{|\gS_r|}\sum\limits_{m \in \gS_r}(\vc_{r+1}^m - \vc_r^m);\\
\Delta_r &= \frac{1}{|\gS_r|} \sum\limits_{m\in \gS_r} (\vc_{r+1}^m - \vc_r^m).
\end{aligned}
\end{equation}
If $\gS_r = [M]$ for all $r = 0, \cdots, R-1$, this equation is equivalent to Equation~\eqref{eq:gradient tracking step}.

\paragraph{Sparsification}
The ReFedDA-GT algorithm concludes with a sparsification operation. Here, $\textbf{Sparse}(\vx; s)$ selects the $s$ largest components in absolute value in vector $\vx \in \mathbb{R}^d$, setting the rest to zero. We will see in Section~\ref{sec:prf skt of str cvx} to figure out why we do this.

\input{Algorithms/restricted_central_fedda-gt}

\subsubsection{Multistep ReFedDA-GT}
Multistep ReFedDA-GT is a valuable tool for addressing high-dimensional stochastic optimization under strong convexity assumptions \citep{agarwal2012stochastic,juditsky2023sparse,bao2022fast}.

This algorithm follows a clear intuition: At the end of the $n$-th step, we obtain $\vw_{n+1}$. By appropriately selecting $R_n$ and $K_n$, we achieve high-probability bounds like $\norm{\vw_{n+1} - \vw^*}_q^2 \le \frac{1}{2}\norm{\vw_n - \vw^*}_q^2 \le \frac{1}{2}Q_n = Q_{n+1}$. These pairs $(\vw_n , Q_n)$ and $(\vw_{n+1}, Q_{n+1})$ correspond precisely to the initial point and retrieval radius in the $n$-th and $(n+1)$-th steps of ReFedDA-GT. The search for $\vw^*$ narrows from $\{\vw\colon \norm{\vw - \vw_n}_q^2 \le Q_n\}$ to a smaller range $\{\vw\colon \norm{\vw - \vw_{n+1}}_q^2 \le Q_{n+1}\}$. With accuracy $\varepsilon$, this approach yields $N = \mathcal{O}(\log_2 \frac{1}{\varepsilon})$ steps, resulting in $\norm{\vw_N - \vw^*}_1^2 \precsim \frac{1}{2^N}\norm{\vw_0 - \vw^*}_q^2 = \mathcal{O}(\varepsilon)$.

\input{Algorithms/multi-stage_recfda_alg}

\subsection{Convergence Rate} \label{sec:conv rt of str cvx}
Now, we are going to provide a concise version of the convergence rate of Multistep ReFedDA-GT algorithm, Theorem~\ref{thm:conv_rt_multi_recfda}, which is then followed by its more intricate form in Appendix~\ref{sec:conv_rt & prf when str cvx}.
\begin{theorem}[Simplified Version of Theorem~\ref{thm:conv_rt_multi_recfda}]\label{thm:simple conv_rt_multi_recfda}
    Suppose Assumption~\ref{asp:convex} - \ref{asp:lsc} hold. For any given accuracy $\varepsilon>0$, there exists a scenario of the setting of the Algorithm~\ref{alg:multirecfda}'s hyperparameter such that $N = \tilde{\gO}(\log_2 (1/\varepsilon)),~ R_n = \tilde{\gO}(s\kappa)$ and $K_n = \tilde{\gO}(\kappa^{-1}2^n)$ with $\kappa := \frac{L}{\mu(Q_0)}$ and $\norm{\vw_0 - \vw^*}_1^2 \le Q_0$. Then $\norm{\hat{\vw}_N - \vw^*}_1^2 \le \varepsilon$ holds with high probability. Here we omit the explicit dependency on $\log d$ within the notation $\tilde{\mathcal{O}}(\cdot)$ for brevity.
\end{theorem}
\begin{remark}
    It can be immediately obtained that the sample complexity and the communication complexity are $\tilde{\gO}\left( \frac{s^2 \sigma^2}{\varepsilon} \right)$ and $\tilde{\gO}\left(s\kappa \log \frac{Q_0}{\varepsilon}\right)$ respectively for the purpose of achieving $\norm{\hat{\vw} - \vw^*}_1^2 \le \varepsilon$. This is a huge improvement compared to the vanilla FedDA-GT algorithm in the convex situation.

    Comparing the comparable results \citep{karimireddy2019scaffold,liu2023decentralized} and communication complexity lower bounds of distributed optimization \citep{arjevani2015communication} in non-high-dimensional scenarios, we note an additional dependence on the optima's sparsity $s$ in communication complexity. This deviates from the expected behavior, which should primarily correlate with the function $f$'s condition number – the ratio of its smoothness coefficient to its strong convexity coefficient. This divergence arises from our assumption of $\gL_1$-norm smoothness for the objective function, in contrast to the usual $\gL_2$-norm smoothness assumed in non-high-dimensional cases. By focusing on the smoothness of $f$ within the set of s-sparse vectors, we establish the relation: $f(\vw) - f(\vw^*) \le L \norm{\vw - \vw^*}_1^2 \le sL \norm{\vw - \vw^*}_2^2$, resulting in a condition number of $s\kappa$ for $f$. This aligns with the observed dependency in our complexity analysis.
    
    An intriguing research direction involves exploring lower bounds on communication complexity for high-dimensional distributed optimization, as accurate outcomes in this realm remain uncharted. This prospect will be a key focus of our future work.
\end{remark}

\subsection{Proof Sketch} \label{sec:prf skt of str cvx}
Leveraging some concentration analysis for martingale difference sequence helps us to obtain the following high-probability version of the ReFedDA-GT algorithm's convergence result. Here we only display a concise version. The complete statement of the conclusion can be referred to as Theorem~\ref{thm:conv_rt_recfda-gt}.

\begin{theorem}[Simplified Version of Theorem~\ref{thm:conv_rt_recfda-gt}]\label{thm:simple conv_rt_recfda-gt}
    Suppose Assumption~\ref{asp:convex} - \ref{asp:grad_subG} hold. Let $p=2\log d$. For any given $R, K \in \ZB_+$, we can find an appropriate $(\eta_c, \eta_s)$ pair, such that
    \begin{align*}
    f(\bar{\vw}_R) - f(\vw^*) = \tilde{\gO}\left( \frac{1}{R} + \frac{1}{\sqrt{RK}} + \frac{1}{R\sqrt{MK}} \right)
    \end{align*}
    holds for $\bar{\vw}_R$ defined as $\frac{1}{R}\ssum{r}{0}{R-1} \vw_{r,0}$ with high probability.
\end{theorem}

Furthermore, it's important to highlight that when integrating the aforementioned results with the strong convexity assumption, our conclusions are limited to bounds regarding $\norm{\bar{\vw}_R - \vw^*}_2^2$. To extend these bounds to the $\gL_1$-norm context, we also rely on the following crucial property pertaining to the operation $\textbf{Sparse}(\cdot, s)$.

\begin{lemma}[Lemma A.1 in \citep{juditsky2023sparse}]\label{lem:l1<sl2}
    Let $\vw^*$ be $s$-sparse, and let $\vw_s = \mathbf{Sparse}(\vw)$, then we have
    \[
    \norm{\vw_s - \vw^*}_1^2 \le {2s}\norm{\vw_s - \vw^*}_2^2 \le 8s\norm{\vw - \vw^*}_2^2.
    \]
\end{lemma}

With these preparations in place, we are now poised to demonstrate how the techniques of multistep can be leveraged to accelerate the existing algorithm.

First, given the assignment of the hyperparameters, Theorem~\ref{thm:conv_rt_recfda-gt} states that, at the $n$-th stage, when $\norm{\hat{\vw}_n - \vw^*}_1^2 \le Q_n$, there is a probability of at least $1-\frac{6\delta}{\pi n^2}$ for $f(\bar{\vw}_{n+1}) - f(\vw^*)$ to be less than or equal to $\frac{\mu(Q_0)}{32s}Q_n$. Additionally, by combining Lemma~\ref{lem:l1<sl2} with the strong convexity of $f$, we obtain
\[
\frac{\mu(Q_0)}{16s}\norm{\hat{\vw}_{n+1} - \vw^*}_1^2 \le \frac{\mu(Q_0)}{2}\norm{\bar{\vw}_{n+1} - \vw^*}_2^2 \le f(\bar{\vw}_{n+1}) - f(\vw^*) \le \frac{\mu(Q_0)}{32s}Q_n,
\]
implying that $\norm{\hat{\vw}_{n+1} - \vw^*}_1^2 \le \frac{Q_n}{2} = Q_{n+1}$. Thus, we can recursively deduce $\norm{\hat{\vw}_N - \vw^*}_1^2 \le Q_N \le \varepsilon$. The corresponding probability of failure is less than $\frac{6\delta}{\pi^2}(1^{-2} + 2^{-2} + 3^{-2} + \cdots) = \delta$.

\section{Experiment}\label{sec:experiment}
In this section, we verify the empirical efficacy of the proposed algorithms with two numerical experiments: decentralized sparse linear and logistic regression models, in which Assumptions~\ref{asp:convex}, \ref{asp:smooth}, \ref{asp:grad_subG}, and \ref{asp:lsc} are all satisfied, see Appendix~\ref{appendix:assumptions} for details.
To make the heterogeneous gradient term $\gE_M^*$ and the mixing time term $\tau$ as large as possible, in both experiments, we design heterogeneous local objective functions $\left\{f_m\right\}_{m=1}^M$ and use a chain communication graph (and its corresponding chain gossip matrix). 
In a chain communication graph, except for the two nodes at the edges that have only one neighbor, all other nodes have exactly two neighbors.
As a result, the graph has a small spectral gap, leading to low communication efficiency between clients.
We use $\tilde{\vw}^m$ to denote the minimum point of $f_m$ in client $m$, i.e. $\tilde{\vw}^m=\argmin_{\vw \in {\cal X} \subseteq     \RB^{d}}f_m(\vw)$. 
We design $\left\{\tilde{\vw}^m\right\}_{m=1}^M$ to be distinct and non-sparse, indicating there are non-zero and heterogeneous client-specific effects. 
And the global optima $\vw^*$ is designed to be sparse, implying there exist a few non-zero overall effects, while other client-specific effects vanish at the population level.
This setting is common in real-world applications, especially in the personalized learning literatures \citep{t2020personalized,li2021ditto,collins2021exploiting,lin2022personalized}.

We compare the proposed methods DFedDA, DFedDA-GT, DReFedDA-GT (decentralized version of ReFedDA-GT) and MReFedDA-GT (Multistep
ReFedDA-GT), 
with three baselines,
DFedMiD (decentralized version of FedMid in \citep{yuan2021federated}), Fast-FedDA and C-FedDA \citep{bao2022fast}. 
We evaluate the algorithms on these models with $M=16$ ($\sigma_2(\rmU)\approx 0.987$), $d=1024$, $s=16$, $p=12$, $K=10$, and batch size of 10.
We record average optimality gap, $\gL_2$ error and $\gL_1$ error.
We do not record support recovery since each method can recover the support perfectly. 
We tune other hyperparameters (e.g. client and server step size) by selecting the values that yield the minimum $\gL_1$ error, which can best reflect the performance in high-dimensional sparse recovery problems.
See more details of experiments in Appendix~\ref{appendix:exp}.


\paragraph{Decentralized Sparse Linear Regression}
$\xi=(\vx^\top,y)^\top\in \RB^{d}\times \RB$ and $F(\vw; \xi)=\frac{1}{2}\left(y-\inner{\vx}{\vw}\right)^2$. 
In client $m$, the covariate vector $\vx$ (with the bias term) is sampled from a common distribution with mean $\mu=(1, \mathbf{0}_{d-1}^\top)^\top$, covariance matrix $\Sigma=\text{diag}\left\{0, \sigma_1^2 \mathbf{I}_{d-1}\right\}$ and bounded support $\norm{\vx}_\infty\leq C$, we assume each component is sampled from a truncated Gaussian distribution.
The variate $y=\inner{\vx}{\tilde{\vw}^m}+e$ depends on the specific client, where $e$ is an independent $N(0, \sigma_2^2)$ noise. 
It can be shown in Lemma~\ref{lem:linear_optimal_solution} that $\vw^*$ is exactly the average of $\left\{\tilde{\vw}^m\right\}_{m=1}^M$.
Therefore, we can design $\left\{\tilde{\vw}^m\right\}_{m=1}^M$ to be distinct and non-sparse, but their average $\vw^*$ is sparse.
The global optima and loss for evaluation can be obtained exactly based on theoretical formulas.
The results are reported in Figure~\ref{Fig:Linear}.


\begin{figure}[htbp]
\begin{minipage}[t]{0.32\linewidth}
    \centering
    \includegraphics[width=\linewidth]{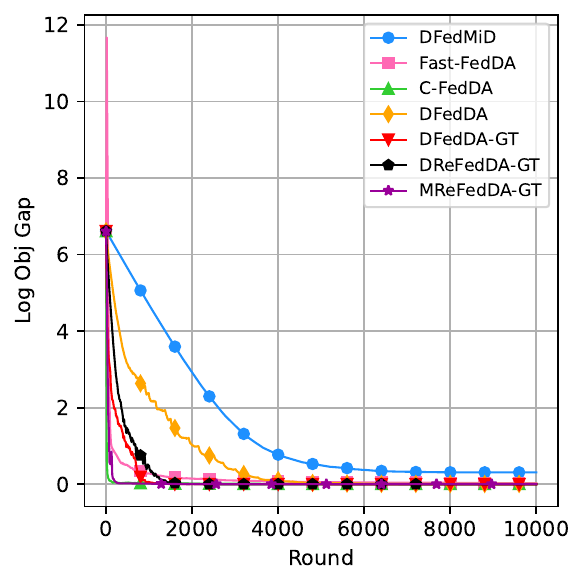}
\end{minipage}
\begin{minipage}[t]{0.32\linewidth}
    \centering
    \includegraphics[width=\linewidth]{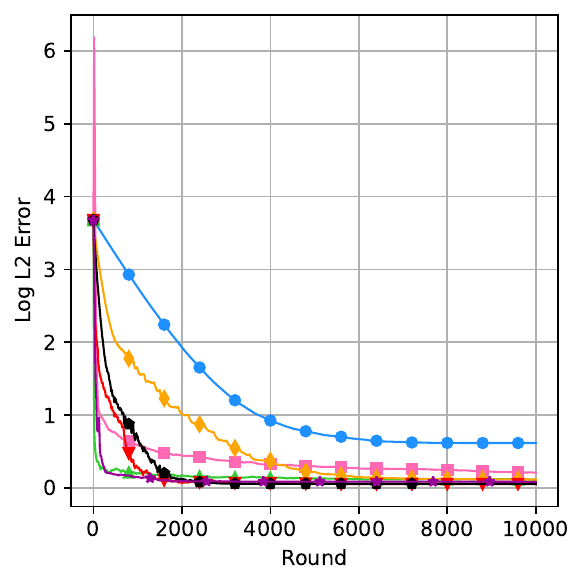}
\end{minipage}
\begin{minipage}[t]{0.32\linewidth}
    \centering
    \includegraphics[width=\linewidth]{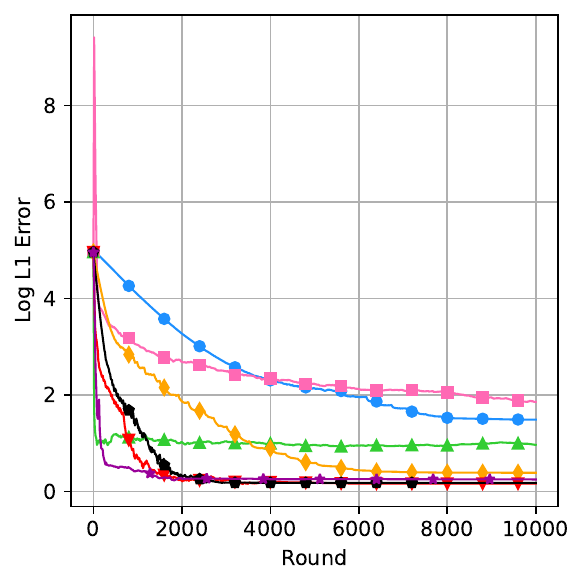}
\end{minipage}
    \caption{Results for decentralized sparse linear regression, the x-axis represents the communication rounds.
    Except for DFedMiD and DFedDA, all other methods exhibit rapid convergence in terms of optimality gap and $\gL_2$ error. 
    However, our methods DFedDA-GT, DReFedDA-GT and MReFedDA-GT demonstrate significant advantages in terms of $\gL_1$ error.}
    \label{Fig:Linear}    
\end{figure}

\paragraph{Decentralized Sparse Logistic Regression}
$\xi=(\vx^\top,y)^\top\in \RB^{d}\times \{0,1\}$ and $F(\vw; \xi)=-y\log\sigma(\inner{\vx}{\vw})-(1-y)\log\sigma(-\inner{\vx}{\vw})$, where $\sigma(z)=\frac{1}{1+\exp\{-z\}}$ is the sigmoid function. 
In client $m$, the covariate vector $\vx$ (without the bias term) is sampled from a common isotropic distribution with mean $\mu=\mathbf{0}_{d}$, covariance matrix $\Sigma=\sigma_1^2 \mathbf{I}_{d}$ and bounded support $\norm{\vx}_\infty\leq C$, we assume $\norm{\vx}_2$ and the absolute value of a truncated Gaussian are identically distributed.
The variate $y\sim\text{Ber}\left(\sigma(\inner{\vx}{\tilde{\vw}^m}) \right)$ depends on the specific client. 
Unlike decentralized linear regression, due to the nonlinearity of logistic regression, $\vw^*$ is no longer a simple average of $\left\{\tilde{\vw}^m\right\}_{m=1}^M$. 
However, for even client number $M=2I$, we can construct pairwise non-sparse and distinct local optimal solutions 
in Lemma~\ref{lem:logit_general_optimal_solution}, such that $\vw^*$ is sparse. 
Similarly, we can construct more general scenarios, while we will only consider this example with paired parameters for simplicity.
The global optimal solution used for evaluation is obtained by centralized SGD, and optimality gap is evaluated approximately using a fixed validation set.
The results are reported in Figure~\ref{Fig:Logistic}.
\begin{figure}[htbp]
\begin{minipage}[t]{0.32\linewidth}
    \centering
    \includegraphics[width=\linewidth]{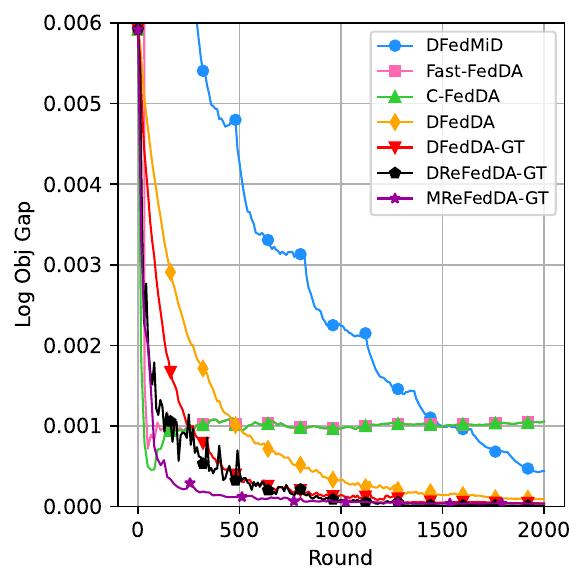}
\end{minipage}
\begin{minipage}[t]{0.32\linewidth}
    \centering
    \includegraphics[width=\linewidth]{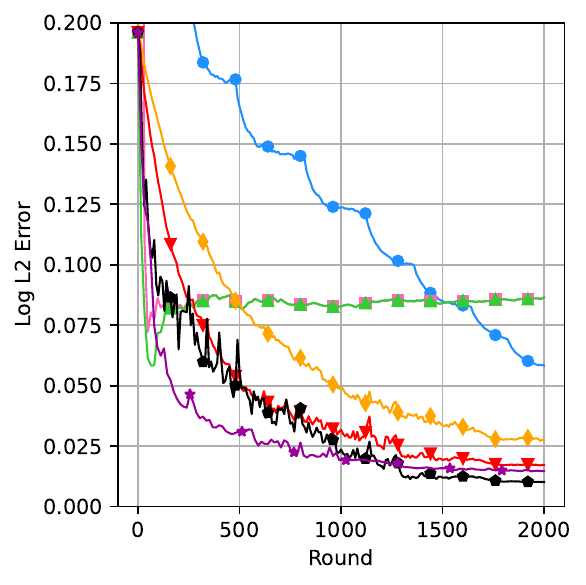}
\end{minipage}
\begin{minipage}[t]{0.32\linewidth}
    \centering
    \includegraphics[width=\linewidth]{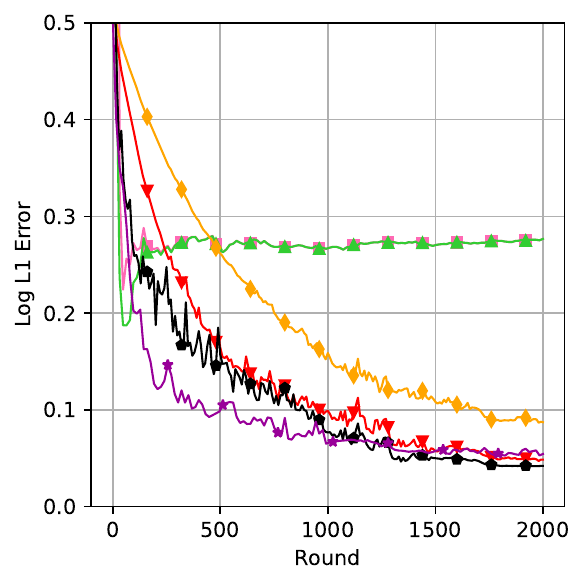}
\end{minipage}
    \caption{Results for decentralized sparse logistic regression, the x-axis represents the communication rounds. Our methods DFedDA-GT, DReFedDA-GT, and MReFedDA-GT outperform the baselines in all metrics, especially in terms of $\gL_1$ error.}
    \label{Fig:Logistic}    
\end{figure}

From Figure~\ref{Fig:Linear} and Figure~\ref{Fig:Logistic}, it can be observed that DFedDA-GT, DReFedDA-GT, and MReFedDA-GT consistently perform better than or comparably to baselines in all metrics, which reveals their dual advantage of being dimension-free and heterogeneity-free. 
It is worth noting that our methods demonstrate significant advantages in terms of $\gL_1$ error, highlighting their capability in addressing decentralized high-dimensional sparse recovery problems with low communication efficiency. 
These experimental findings are consistent with our theoretical results.

\section{Related Work}\label{sec:related_work}
The primary objective of this section is to conduct a comparative analysis between our work and two other papers that are most closely related to ours, namely, \citep{yuan2021federated} and \citep{bao2022fast}, to the best of our knowledge.

In \citep{yuan2021federated}, the authors introduced two novel algorithms, Federated Mirror Descent (FedMiD) and FedDA, and provided corresponding convergence rates based on different step size selections. Notably, under the scenario of larger step sizes, the convergence rate matches that of our DFedDA algorithm.

However, both of these step-size settings come with their respective issues. For the case of large step sizes, the authors employed the estimator $\hat{\vw}:= \frac{1}{RK}\sum_{r=0}^{R-1}\sum_{k=0}^{K-1} \nabla^{-1}h \left( \frac{1}{M}\sum_{m=1}^{M}\vz_{r,k}^m \right)$. When the mirror map $\nabla^{-1}h$ is nonlinear (as in the high-dimensional optimization scenario we consider), it necessitates knowledge of the average values of parameters $\vz_{r,k}^m$ across all clients at each iteration within every round. This is evidently impractical in a federated setting.

Conversely, for small step size settings, the issue arises in the proof itself. Frankly speaking, the proof in the literature becomes entangled in a form of circular reasoning, leading to an unattainable convergence order as claimed by the authors.

Nevertheless, regardless of these issues, our designed DFedDA-GT algorithm consistently achieves superior convergence performance compared to \citep{yuan2021federated}.

\citet{yuan2021federated} consider general convexity assumptions. While \citet{bao2022fast} derive the Fast-FedDA algorithm under the strong convexity assumption by appropriately adjusting the form of the proximal operator.

On the one hand, \citet{bao2022fast}'s approach fundamentally does not address the challenges of high-dimensional federated optimization. This is because their assumptions regarding noise and the smoothness of the loss function are still dependent on the full space dimensionality $d$.

On the other hand, from a communication complexity perspective, both Fast-FedDA and DFedDA-GT have the same complexity, which is $\gO(\varepsilon^{-\frac{1}{2}})$. However, the complexity result of DFedDA-GT does not require assuming strong convexity of the objective function. It's also worth noting that \citet{bao2022fast} has designed a multistep version of FedDA. Unlike multistep RFedDA-GT, which reduces communication complexity from $\gO(\varepsilon^{-\frac{1}{2}})$ to $\gO(\log \frac{1}{\varepsilon})$, \citet{bao2022fast}'s method does not achieve significant gains in communication efficiency through the use of multistep techniques. This is primarily because they overlook the key point that gradient tracking techniques can optimize communication efficiency by mitigating heterogeneity.
\section{Conclusion and Future Work}
In this paper, we investigate high-dimensional distributed optimization problems. Under the assumption of convexity, we apply gradient tracking techniques to the high-dimensional setting, proposing the decentralized FedDA-GT algorithm, which exhibits improved sample complexity and communication complexity. In the case of a strongly convex objective function, we introduce the multistep restricted FedDA-GT algorithm. Leveraging the doubling trick, we reduce the communication complexity to $\gO\left(\frac{sL}{\mu}\right)$, representing the best-known communication complexity achievable in high-dimensional scenarios to the best of our knowledge. Despite this achievement, there remains a gap compared to the lower bound of $\Omega\left(\sqrt{\frac{L}{\mu}}\right)$ established by \citet{shamir2014communication}. We attribute this gap to two main factors. Firstly, our algorithm does not incorporate acceleration techniques, indicating the need for the design of a novel algorithm. Secondly, the discrepancy between our results and the assumptions used in the lower bound suggests the necessity of re-establishing a communication complexity lower bound specific to high-dimensional scenarios. We leave these considerations for future work.

\newpage

\appendix
\input{supplement}

\bibliographystyle{plainnat}  
\bibliography{bib/distributed,bib/federated,bib/optimization,bib/references}

\end{document}

%% file: math_command.tex
\usepackage{amsmath,amsfonts,bm}
\usepackage{bbm}

\def\rmA{{\mathbf{A}}}

\def\rmC{{\mathbf{C}}}

\def\rmF{{\mathbf{F}}}
\def\rmG{{\mathbf{G}}}

\def\rmI{{\mathbf{I}}}
\def\rmJ{{\mathbf{J}}}

\def\rmU{{\mathbf{U}}}

\def\rmW{{\mathbf{W}}}
\def\rmX{{\mathbf{X}}}

\def\rmZ{{\mathbf{Z}}}

\def\va{{\bm{a}}}

\def\vc{{\bm{c}}}

\def\ve{{\bm{e}}}

\def\vg{{\bm{g}}}

\def\vq{{\bm{q}}}

\def\vs{{\bm{s}}}

\def\vu{{\bm{u}}}
\def\vv{{\bm{v}}}
\def\vw{{\bm{w}}}
\def\vx{{\bm{x}}}

\def\vz{{\bm{z}}}

\DeclareMathAlphabet{\mathsfit}{\encodingdefault}{\sfdefault}{m}{sl}
\SetMathAlphabet{\mathsfit}{bold}{\encodingdefault}{\sfdefault}{bx}{n}

\def\gB{{\mathcal{B}}}

\def\gD{{\mathcal{D}}}
\def\gE{{\mathcal{E}}}
\def\gF{{\mathcal{F}}}

\def\gH{{\mathcal{H}}}

\def\gJ{{\mathcal{J}}}
\def\gK{{\mathcal{K}}}
\def\gL{{\mathcal{L}}}
\def\gM{{\mathcal{M}}}

\def\gO{{\mathcal{O}}}

\def\gS{{\mathcal{S}}}
\def\gT{{\mathcal{T}}}

\def\gX{{\mathcal{X}}}

\def\0{{\bf 0}}
\def\1{{\bf 1}}

\def\EB{{\mathbb E}}

\def\PB{{\mathbb P}}

\def\RB{{\mathbb R}}

\def\ZB{{\mathbb Z}}

\def\argmin{\mathop{\rm argmin}}

\def\sgn{\mathrm{sgn}}

\newcommand{\E}{\mathbb{E}}

\newcommand{\ssum}[3]{\sum\limits_{{#1}={#2}}^{#3}}

\def\inner#1#2{\left\langle #1, #2 \right\rangle}
\newcommand{\norm}[1]{\left\|{#1}\right\|}
\newcommand{\abs}[1]{\left|{#1}\right|}

%% file: intro.tex
In this era of highly developed communication and anyone has mobile electronic devices, federated learning has rapidly become a popular and important branch of machine learning in recent years. It represents a philosophy of decentralizing the training process locally, with the server only performing the communication and integration of information between all local devices \citep{mcdonald2010distributed, mcmahan2017communication, stich2019local}.
In distributed learning, we face two critical challenges: \textbf{communication complexity}, which deals with how often devices exchange information, and \textbf{data heterogeneity}, which relates to variations in data distribution among devices.

Unlike traditional stochastic algorithms like Parallel SGD, Federated Learning offers a distinct advantage by achieving the same accuracy with significantly reduced communication costs. One of the most widely adopted federated learning algorithms is Federated Averaging (FedAvg). It works by independently performing SGD on different clients in parallel, aggregating their local updates on a central server, and then utilizing the averaged information to update the global parameters.
There are plenty of works that analyzed the convergence efficiency, the communication complexity, and some statistical nature of FedAvg \citep{li2019convergence, bayoumi2020tighter, koloskova2020unified, woodworth2020local, li2021statistical}. 

However, these studies have revealed that data heterogeneity poses a barrier to FedAvg's ability to make breakthroughs in communication complexity. We must find a way to mitigate the impact of this heterogeneity originating from each local device. This is where gradient tracking comes into play. The concept of gradient tracking is not a recent innovation; it has its roots in the field of distributed optimization \citep{di2016next,qu2017harnessing,nedic2017achieving}. Following its introduction, a series of subsequent studies harnessed the power of gradient tracking, proposing diverse algorithms tailored to address various challenges in distributed optimization scenarios \citep{pu2021distributed,karimireddy2019scaffold,zhang2019decentralized,koloskova2021improved,liu2023decentralized}.

As hardware capabilities continue to advance, machine learning models' parameter sizes grow correspondingly. However, a high parameter dimension often demands an extensive dataset to learn effectively, which can be impractical \citep{bronstein2021geometric}. In many scenarios, we encounter high-dimensional and sparse problems where the ground truth $\vw^*$ resides in $\RB^d$ with a large $d$, yet the number of non-zero entries in $\vw^*$ is significantly smaller. This formulation is commonly found in genetic data analysis \citep{cawley2006gene} and portfolio selection \citep{chen2013sparse}. Nevertheless, our conventional methods face challenges due to limited sample sizes when addressing such optimization problems.

To tackle this kind of difficulty, a series of evolved manuscripts have been proposed in the past decade~\citep{daubechies2004iterative, beck2009fast, langford2009sparse, salim2019performance, shalev2009stochastic, duchi2010composite, agarwal2012stochastic, juditsky2023sparse}. It is worth emphasizing that \citet{shalev2009stochastic} introduced a mirror descent approach into high dimensional optimization problems and substantially alleviated the sample complexity. And \citet{agarwal2012stochastic} combined the dual average algorithm (which is known as `lazy' mirror descent) with an annealing technique and refined the sample complexity from $\gO\left(\frac{s\sigma^2\log d}{\epsilon^2}\right)$ to $\gO\left(\frac{s\sigma^2 \log d}{\epsilon}\right)$.

Recently, there's been a growing trend in solving high-dimensional stochastic optimization problems through federated learning due to the increasing number of model parameters and data distribution. Several recent studies have explored distributed optimization methods tailored for high-dimensional scenarios \citep{duchi2011dual, yuan2020federated, tran2021feddr,bao2022fast, zhu2023federated}. 
Notably, \citep{duchi2011dual} was among the first to address high-dimensional optimization in a distributed context, but they didn't consider communication complexity. More recently, \citet{yuan2020federated} analyzed two algorithms: federated mirror descent (FedMiD) and federated dual averaging (FedDA). Their analysis revealed that FedDA outperforms FedMiD in terms of convergence performance due to its `primal averaging free' approach. 
Additionally, \citet{bao2022fast} introduced the Fast-FedDA method, an improvement over vanilla FedDA. Fast-FedDA improves sample and communication complexity, reducing them from $\gO(\epsilon^{2})$ and $\gO(\epsilon^{3/2})$ to $\gO(\epsilon)$ and $\gO(\epsilon^{1/2})$, respectively, assuming that each device's loss function is strongly convex.

Despite the vitality of this research direction, in most of these works, two key aspects were largely overlooked. Firstly, they neglected the impact of data heterogeneity on algorithm efficiency, resulting in suboptimal communication complexity for these algorithms. Secondly, their focus remained primarily on centralized scenarios, whereas in reality, decentralized situations are more commonly encountered. This leads us to the following questions, which constitute the core issues addressed in this paper:

\textit{\textbf{For high-dimensional (decentralized) distributed optimization problems, what is the best achievable communication complexity?}}

\subsection{Contribution}
This paper makes two main contributions, addressing the questions raised earlier under the conditions of convex and strongly convex objective functions.

\begin{enumerate}
    \item When each device deals with a generally convex objective function, we introduce the Decentralized Federated Dual Averaging with Gradient Tracking algorithm (DFedDA-GT). By incorporating gradient tracking techniques for the first time in high-dimensional scenarios, DFedDA-GT outperforms DFedDA in terms of communication complexity, improving it from $\gO\left( \frac{\gE_M^*}{\varepsilon^{3/2}} \right)$ to $\gO\left(\frac{s^2}{\varepsilon}\right)$, respectively. Notably, before our work, there has been no discussion on high-dimensional federated optimization in decentralized settings.

    \item Furthermore, when the global objective function is locally strongly convex with parameter $\mu$, we propose the Multistep ReFedDA-GT algorithm. During its execution, this algorithm iteratively calls the ReFedDA-GT algorithm, leveraging the terminal state of the previous step as the starting point for the current step. This enables us to further enhance efficiency, achieving a sample complexity of $\gO(\frac{s^2(\log d)^4 \sigma^2}{\mu^2 \varepsilon^2})$ and a communication complexity of $\gO\left(s\kappa \log d \log \frac{1}{\varepsilon}\right)$.
\end{enumerate}

\begin{table}[ht]
    \centering
    \fontsize{7}{9}\selectfont
    \begin{tabular}
    {>{\centering\arraybackslash}p{1.7cm}|>{\centering\arraybackslash}p{1.6cm}|>{\centering\arraybackslash}p{1.3cm}|>{\centering\arraybackslash}p{2.2cm}|>{\centering\arraybackslash}p{2.6cm}|>{\centering\arraybackslash}p{1.7cm}}
    \toprule
    \textbf{Algorithm} & \textbf{Condition} & \textbf{Gradient tracking} & \textbf{Sample complexity} & \textbf{Communication complexity} & \textbf{Decentralized} \\
    \hline
    FedDA\citep{yuan2021federated} & Convex, $\gL_1$-smooth & \vspace{.4pt}\makecell{\centering\XSolidBrush} & $\gO\left( \frac{Ls^4\sigma^2}{\varepsilon^{3/2} } \vee \frac{s^{2}\sigma^2}{M\varepsilon^2}\right)$ & $\gO\left( \frac{L^{\frac{1}{2}}s^2 \gE_M^*}{\varepsilon^{3/2}} \right)$ &\vspace{.4pt} \makecell{\centering\XSolidBrush} \\
    \hline
    DFedDA & Convex, $\gL_1$-smooth &\vspace{.4pt} \makecell{\centering\XSolidBrush} & $\gO\left(\frac{\tau s^2\sigma^2 (\log d)^4}{\varepsilon^2}\right)$ & $\gO\left(\frac{\tau s^2 L^{\frac{1}{2}}{\gE_M^*}^{\frac{1}{2}}}{\varepsilon^{3/2}}\right)$ & \vspace{.4pt} \makecell{\centering\Checkmark} \\
    \hline
    Fast-FedDA\citep{bao2022fast} & Strongly Convex, $\gL_2$-smooth &\vspace{.4pt} \makecell{\centering\XSolidBrush} & $\gO \left( \frac{Ls^2 + \sigma^2/(\mu M)}{\varepsilon} \right)$ &  $\gO\left( \frac{sL^{\frac{3}{2}} + (L\gE_M^*)^{\frac{1}{2}}}{\mu \varepsilon^{1/2}} \right) $ &\vspace{.4pt} \makecell{\centering\XSolidBrush} \\
    \hline
    Multistep C-FedDA\citep{bao2022fast} & Strongly convex, $\gL_2$-smooth & \vspace{.4pt} \makecell{\centering\XSolidBrush} & $\gO\left(\frac{s\sigma^2}{M\mu^2 \varepsilon}\right)$ & $\gO\left(\frac{(sL\gE_{M}^*)^{\frac{1}{2}} + (sL)^{\frac{3}{2}})}{\mu^{3/2}\varepsilon^{1/2}}\right)$ &\vspace{.4pt} \makecell{\centering\XSolidBrush} \\
    \hline
    DFedDA-GT & Convex, $\gL_1$-smooth &\vspace{.4pt} \makecell{\centering\Checkmark} & $\gO\left(\frac{\tau^2s^2\sigma^2\log^2 d}{\varepsilon^2}\right)$ & $\gO\left(\frac{\tau^2Ls^2}{\varepsilon}\right)$ &\vspace{.4pt} \makecell{\centering\Checkmark} \\
    \hline
    Multistep ReFedDA-GT & Strongly convex, $\gL_1$-smooth &\vspace{.4pt} \makecell{\centering\Checkmark} & $\gO\left(\frac{s^2\sigma^2}{\varepsilon}\right)$ & $\gO\left(\frac{sL}{\mu}\log \frac{s}{\varepsilon}\right)$ &\vspace{.4pt} \makecell{\centering\XSolidBrush} \\
    \hline
    DASM\citep{duchi2011dual} & Convex, $\gL_1$-Lipschitz &\vspace{.4pt} \makecell{\centering\XSolidBrush} & $\gO\left( \frac{\tau s^2\Tilde{L}^2 \log^2\frac{M}{\varepsilon}}{\varepsilon^2} \right)$ & $\gO\left( \frac{\tau s^2\Tilde{L}^2 \log^2\frac{M}{\varepsilon}}{\varepsilon^2} \right)$ &\vspace{.4pt} \makecell{\centering\Checkmark} \\
    \bottomrule
    \end{tabular}
    \caption{
    Comparison of the sample complexity and communication complexity between our result and prior works. Here, $s$ denotes the sparsity of the optima, while $\Tilde{L}$, $L$, and $\mu$ represent the Lipschitz continuity, smoothness, and strong convexity of the objective function, respectively. Additionally, $\gE_M^*$ can be interpreted as the heterogeneity of this optimization problem. In the context of decentralized optimization scenarios, $\tau$ refers to the mixing time of a communication graph.
    }
    \label{tab:my_label}
\end{table}

%% file: Algorithms/decentral_scaffolDA_alg.tex
\begin{algorithm}
  \caption{\textsc{Decentralized FedDA with Gradient Tracking}}
  \label{alg:dscaffoldda}
  \begin{algorithmic}[1]
  \STATE {\textbf{procedure}} \textsc{DFedDA-gt} ($\vw_0,~ \eta_{c},~ \eta_{s},~ \rmU, K, R$, \texttt{gradient tracking})
  \STATE \textbf{Input:} Initial point $\vw_0$, client step size $\eta_c$, server step size $\eta_s$, gossip matrix $\rmU$, local iteration number $K$, global round $R$, $\texttt{gradient tracking}$ (a boolean variable, default: \textbf{True})
  \STATE \textbf{Initialization:}
  \STATE $\vz_{0} \gets \nabla h(\vw_0)$ \hfill \textit{// parameter initialization}
  \STATE $\vc_{0}^m \gets -\nabla F(\vw_0, \xi_{0}^m) + \ssum{j}{1}{M}u_{jm} \nabla F(\vw_0, \xi_{0}^j)$
  \hfill \textit{// tracker initialization}
  \FOR {$r=0, \ldots, R-1$}
      \FORALL { $m \in [M]$ {\bf in parallel}}
        \FOR {$k = 0, \ldots, K-1$}
          \STATE $\vw_{r,k}^m \gets \nabla^{-1}h(\vz_{r,k}^m)$
            \hfill \textit{// retrieve primal}
          \STATE $\vg_{r,k}^m \gets \nabla F (\vw_{r,k}^m; \xi_{r,k}^m) + \vc_{r}^m \mathbbm{1}_{\{\texttt{gradient tracking}\}}$  \hfill \textit{// query gradient}  
          \STATE $\vz_{r,k+1}^m \gets \vz_{r,k}^m - \eta_{c} \vg_{r,k}^m$
            \hfill \textit{// client \emph{dual} update}
        \ENDFOR
    \STATE Compute $\Delta_r^m,~\vc_{r+1}^m$ according to \eqref{eq:gradient tracking step}
    \STATE $\vz_{r+1,0}^m \gets \ssum{j}{1}{M}u_{jm}(\vz_{r,0}^j + K\eta_s\eta_c \Delta_r^j)$
    \hfill \textit{// communicate and enter the next round}
    \ENDFOR
    \ENDFOR
    \end{algorithmic}
\end{algorithm}

%% file: Algorithms/restricted_central_fedda-gt.tex
\begin{algorithm}
  \caption{\textsc{Restricted Federated Dual Averaging with Gradient Tracking}}
  \label{alg:recfdagtr}
  \begin{algorithmic}[1]
  \STATE {\textbf{procedure}} \textsc{ReFedDA-GT} ($\vw_0, \eta_{c}, \eta_{s}, Q, R, K, c_0$)
  \STATE $\vz_{0,0} \gets 0$ \hfill \textit{/// server dual initialization}
  \FOR {$r=0, \ldots, R-1$}
    \STATE sample a subset of clients $\mathcal{S}_r \subseteq [M]$
      \FORALL {$m \in \gS_r$ {\bf in parallel}}
        \FOR {$k = 0, \ldots, K-1$}
          \STATE $\vw_{r,k}^m \gets \textbf{Prox}_h (\vw_0, \vz_{r,k}^m; Q, \norm{\cdot}_q)$
          \STATE $\vg_{r,k}^m \gets \nabla F(\vw_{r,k}^m , \xi_{r,k}^m) - \vc_r^m + \vc_r$
          \STATE $\vz_{r,k+1}^m \gets \vz_{r,k}^m - \eta_{c} \vg_{r,k}^m$
        \ENDFOR
      \STATE Compute $\vc_{r+1}^m, \vc_{r+1}$ and $\Delta_r$ according to \eqref{eq:str gradient tracking}  
      \STATE $\vz_{r+1,0} \gets \vz_{r,0} - K\eta_{s} \eta_{c}\Delta_r$
      \STATE $\vw_{r+1,0} \gets \textbf{Prox}_h (\vw_0, \vz_{r+1,0}; Q, \norm{\cdot}_q)$
      \ENDFOR
  \ENDFOR
  \STATE $\hat{\vw} \gets \textbf{Sparse}\left( \frac{1}{R}\ssum{r}{0}{R-1} \vw_{r,0};s \right)$
\end{algorithmic}
\end{algorithm}

%% file: Algorithms/multi-stage_recfda_alg.tex
\begin{algorithm}
  \caption{\textsc{Multistep Restricted FedDA-GT}}
  \label{alg:multirecfda}
  \begin{algorithmic}[1]
  \STATE {\textbf{procedure}} \textsc{Multi-ReFedDA-GT} ($\vw_0, \{\eta_{c}^n, \eta_{s}^n, Q_0, R_n,K_n\}_{n=0}^\infty$)
  \FOR{$n= 0, \ldots, N-1$}
  \STATE ${\vw}_{n+1} \gets$ \textsc{ReFedDA-GT} ($\vw_n, \eta_{c}^n, \eta_{s}^n, Q_n, R_n, K_n, 0$)
  \STATE $Q_{n+1} = Q_n / 2$
  \ENDFOR
  
  \RETURN ${\vw}_N$
\end{algorithmic}
\end{algorithm}

%% file: supplement.tex
\input{Proofs/auxiliary_lemmas}

\section{Convergence Rate for Decentralized Algorithms}
\input{theory/decentral_convrate}

\section{Convergence Rate under Strong Convexity}

\input{theory/str_convrate}

\section{Experiments}\label{appendix:exp}
\input{apendix_experiments}

%% file: Proofs/auxiliary_lemmas.tex
\section{Auxiliary Lemmas}
In this subsection, we present some existing results and auxiliary lemmas useful for our later
analysis.

\begin{lemma}[Theorem 2.1.5 of \citep{nesterov2018lectures}]\label{lem:conj_sc}
Under Assumptions~\ref{asp:convex}, \ref{asp:smooth}, the following result holds for all $m\in [M]$ and $\vw,\vv \in \RB^d$,
\[
f_m(\vw) + \inner{\nabla f_m(\vw)}{\vv - \vw} + \frac{1}{2L}\norm{\nabla f_m(\vw) - \nabla f_m (\vv)}_{\infty}^2
\le f_m(\vv).
\]
\end{lemma}

\begin{lemma}\label{lem:multi_subG}
Let $M$ independent zero-mean random variables $\rmX_1, \cdots, \rmX_M$ satisfy $\EB \exp \{\lambda \rmX_i \} \le \exp \{\frac{\lambda^2 \sigma^2}{2}\},~ \forall \lambda \in \RB, i\in [M]$. Then,
\[
\EB \exp \left\{ \lambda \ssum{i}{1}{M} \rmX_i\right\} \le \exp \left\{\frac{\lambda^2 \sigma_M^2}{2} \right\}, ~\forall \lambda \in \RB.
\]
Here $\sigma_M = \sqrt{M}\sigma$.
\end{lemma}

\begin{lemma}\label{lem:3_pnt_smooth}
Under Assumptions~\ref{asp:convex}, \ref{asp:smooth}, the following result holds for all $m\in [M]$ and $\vw,\vv,\vu \in \RB^d$,
\[
\inner{\nabla f_m(\vw)}{\vv - \vu} \le f_m(\vv) - f_m(\vu) + \frac{L}{2}\norm{\vu - \vw}_1^2.
\]
\end{lemma}

\begin{lemma}[Theorem 2.6 in \citet{wainwright2019high}] \label{lem:mmt_bd_subG}
Given any zero-mean random variable $\rmX$, suppose that $\EB e^{\lambda \rmX} \le e^{\frac{\lambda^2\sigma^2}{2}}$ for all $\lambda \in \RB$. Then there exists a constant $C$ such that
\[
\EB \rmX^{2k} \le \frac{(2k)!}{2^k k!}C^{2k} \sigma^{2k}, \quad \forall k = 1,2, \cdots.
\]
\end{lemma}

\begin{lemma}[H\"older's inequality] \label{lem:holder_ineq}
For any vector $\vx\in \RB^d$, $p\geq 1$ and $1/p+1/q=1$, we have
\[
\|\vx\|_p^2 \leq d^{2/p} \|\vx\|_\infty^2 \text{ and } \|\vx\|_1^2 \leq d^{2/p} \|\vx\|_q^2.
\]
\end{lemma}

\begin{proof}[Proof of Lemma~\ref{lem:bd_of_p_norm_2nd_mmt}]
    By leveraging Lemma~\ref{lem:bdg_ineq} and Jensen's inequality,
    \begin{align}
        &\EB \norm{\ssum{i}{1}{n}\vv_i}_p^2 = \EB \left( \ssum{j}{1}{d}\left| \ssum{i}{1}{n} [\vv_i]_j \right|^p \right)^{2/p}
        \le \left(
        \ssum{j}{1}{d}\EB \left| \ssum{i}{1}{n} [\vv_i]_j \right|^p
        \right)^{2/p}\\
        \le&
        \left\{
        \ssum{j}{1}{d}C^p p^{3p/2}\EB \left(
        \ssum{i}{1}{n}\EB \left[ [\vv_i]_j^2 | \gF_{i-1} \right]
        \right)^{p/2}
        \right\}^{2/p}\\
        \le&
        C^2 d^{2/p}p^{3} \sup\limits_{j\in [d]} \left\{
        \EB \left(
        \ssum{i}{1}{n}\EB \left[ [\vv_i]_j^2 | \gF_{i-1} \right]
        \right)^{p/2}
        \right\}^{2/p}.
    \end{align}
    For any $j\in [d]$, using Jensen's inequality twice yields,
    \begin{align}
        &\EB \left(
        \ssum{i}{1}{n}\EB \left[ [\vv_i]_j^2 | \gF_{i-1} \right]
        \right)^{p/2} 
        = n^{p/2} \EB \left(
        \frac{1}{n}\ssum{i}{1}{n}\EB \left[ [\vv_i]_j^2 | \gF_{i-1} \right]
        \right)^{p/2}\\
        \le&
        n^{p/2-1} \EB \ssum{i}{1}{n}\left(
        \EB \left[ [\vv_i]_j^2 | \gF_{i-1} \right]
        \right)^{p/2}
        \le
        n^{p/2 - 1} \ssum{i}{1}{n}\EB \left| [\vv_i]_j \right|^p\\
        \le&
        n^{p/2} \sup\limits_{i\in [n]} \EB \left| [\vv_i]_j \right|^p.
    \end{align}
    So finally,
    \begin{align}
        &\EB \norm{\ssum{i}{1}{n}\vv_i}_p^2 \le C^2d^{2/p}p^3 \sup\limits_{j\in [d]}\left\{
        n^{p/2}\sup\limits_{i\in [n]} \EB \left| [\vv_i]_j \right|^p
        \right\}^{2/p}\\
        \le&
        C^2d^{2/p}p^3n \sup\limits_{i\in[n], j\in [d]} \EB \left| [\vv_i]_j \right|^p.
    \end{align}
\end{proof}

\begin{lemma}\label{lem:mixing_rt_under_l_p2}
    Suppose $\rmA \in \RB^{d\times M}$, then for any $t \ge 0$,
    \[
    \norm{\rmA(\rmU^{t} - \rmJ_M)}_{p,2}^2 \le M(\sigma_2(\rmU))^{2t} \norm{\rmA}_{p,2}^2.
    \]
\end{lemma}

\begin{proof}[Proof of Lemma~\ref{lem:mixing_rt_under_l_p2}]

We use $\va_i$ to denote the $i$-th column of $\rmA$, then
\begin{align*}
    &\norm{\rmA(\rmU^t - \rmJ_M)}_{p,2}^2 = \ssum{j}{1}{M}\norm{\ssum{i}{1}{M}(u^{(t)}_{ij} - M^{-1}) \va_i}_p^2 \overset{(a)}{\le}
    \ssum{j}{1}{M}\left( 
    \ssum{i}{1}{M}|u_{ij}^{(t)} - M^{-1}| \cdot \norm{\va_i}_p
    \right)^2\\
    \le&
    \ssum{j}{1}{M}\left( 
    \ssum{i}{1}{M}(u_{ij}^{(t)} - M^{-1})^2 
    \right)
    \left( 
    \ssum{i}{1}{M} \norm{\va_i}_p^2
    \right) \le
    \left( \ssum{j}{1}{M}\norm{\rmU^{t}\ve_j - \frac{\mathbf{1}}{M}}^2 \right) \norm{\rmA}_{p,2}^2\\
    \le&
    M\sigma_2(\rmU)^{2t} \norm{\rmA}_{p,2}^2.
\end{align*}
Here $(a)$ holds by the triangular inequality of $\norm{\cdot}_p$.
\end{proof}

\begin{lemma}\label{lem:db_stoc_under_l_p2}
    Suppose $\rmA \in \RB^{d\times M}$, and $\rmU \in \RB^{M\times M}$ be a double stochastic matrix. Then
    \begin{align*}
        \norm{\rmA \rmU}_{p,2}^2 \le \norm{\rmA}_{p,2}^2.
    \end{align*}
\end{lemma}

\begin{proof}[Proof of Lemma~\ref{lem:db_stoc_under_l_p2}]
We use $\va_i$ to denote the $i$-th column of $\rmA$. Then
\begin{align*}
    &\norm{\rmA \rmU}_{p,2}^2 = \ssum{j}{1}{M}\norm{\ssum{i}{1}{M}u_{ij}\va_i}_p^2
    \le \ssum{j}{1}{M}\left(\ssum{i}{1}{M}u_{ij}\norm{\va_i}_p\right)^2\\
    \overset{(a)}{\le}& \ssum{j}{1}{M}\left(\ssum{i}{1}{M}u_{ij}\right)\left(\ssum{i}{1}{M}u_{ij}\norm{\va_i}_p^2\right) = \ssum{j}{1}{M}\ssum{i}{1}{M}u_{ij}\norm{\va_i}_p^2
    \le \ssum{i}{1}{M}\left( \ssum{j}{1}{M} u_{ij} \right)\norm{\va_i}_p^2 
    = \norm{\rmA}_{p,2}^2.
\end{align*}
\end{proof}

\begin{lemma}\label{lem:rcs_to_etr_fedda}
Suppose the positive sequence $\{\vs_r\}_{r=0}^{K-1}$ satisfies $\vs_0 = 0$,
\[
\vs_{r+1} \le (1+ ar\eta^2)\vs_r + b,\quad r=0,\cdots K-2
\]
with $a,b > 0$. Then we have $\vs_{K-1} \le e K b$ as long as $\sqrt{a}K\eta \le 1$.
\end{lemma}

\begin{proof}[Proof of Lemma~\ref{lem:rcs_to_etr_fedda}]
By simple computation, we can get
\begin{align*}
\vs_{r+1} &\le (1+ar\eta^2) \vs_r + b \le (1 + (aK^2\eta^2)/K)\vs_r + b\\
&\le \left(1 + K^{-1} \right)\vs_r + b.
\end{align*}
Using the above inequality recursively results in
\begin{align*}
\vs_{K-1} &\le \left(1 + K^{-1}\right)^K \vs_0 + b\ssum{r}{0}{K-1}(1+ K^{-1})^r \le b\ssum{r}{0}{K-1}e = eKb.
\end{align*}
\end{proof}

%% file: theory/decentral_convrate.tex
\subsection{Convergence Rate for DFedDA}
\begin{theorem}[Convergence rate for DFedDA]\label{thm:conv_rt_defedda}
Suppose Assumptions~\ref{asp:convex} - \ref{asp:mx_gossip} hold. Let $p=2\log d$, $\tau = \frac{\log 4M}{2\log(1/\sigma_2(\rmU))}\vee 1$, $\eta \lesssim \frac{1}{\tau LK}$, and $\hat{\vw}^m := \frac{1}{R}\ssum{r}{0}{R-1}\vw_{r,0}^m$. We have
\begin{align}
    \frac{1}{M} \ssum{m}{1}{M} \EB\{f(\hat{\vw}^m) - f(\vw^*)\} &\lesssim \frac{h(\vw^*)}{RK{\eta}} + \tau (\log d)^4 \sigma^2\eta + \tau^2 LK^2\gE_M^* \eta^2.
\end{align}
Furthermore, if we take $\eta = \min\left\{ \frac{1}{\tau LK}, \frac{h(\vw^*)^{1/3}}{(\tau^2 LR \gE_M^*)^{1/3}K}, \frac{h(\vw^*)^{1/2}}{(\tau R K)^{1/2}(\log d)^2 \sigma} \right\}$, then
\begin{align}
    \frac{1}{M} \ssum{m}{1}{M} \EB\{f(\hat{\vw}^m) - f(\vw^*)\} &\lesssim \frac{\tau L h(\vw^*)}{R} + \frac{(\tau^2 L \gE_M^* h(\vw^*)^2)^{1/3}}{R^{2/3}} + \frac{(\tau h(\vw^*))^{1/2}(\log d)^2 \sigma}{(RK)^{1/2}}.
\end{align}
\end{theorem}

\subsection{Convergence Rate for DFedDA-GT}
\begin{theorem}[Convergence rate for DFedDA-GT]\label{thm:conv_rt_defedda_gt}
Suppose Assumptions~\ref{asp:convex} - \ref{asp:mx_gossip} hold. Let $p=2\log d$, $\tau = \frac{\log 8M}{2\log(1/\sigma_2(\rmU))}\vee 1$, $\eta \lesssim \frac{1}{\tau^2 LK}$, and $\hat{\vw}^m := \frac{1}{R}\ssum{r}{0}{R-1}\vw_{r,0}^m$. We have
\begin{align}
    \frac{1}{M} \ssum{m}{1}{M} \EB\{f(\hat{\vw}^m) - f(\vw^*)\} &\lesssim \frac{h(\vw^*)}{RK{\eta}} + \tau^2 (\log d)^4 \sigma^2\eta.
\end{align}
Furthermore, if we take $\eta = \min\left\{ \frac{1}{\tau^2 LK}, \frac{h(\vw^*)^{1/2}}{\tau (R K)^{1/2}(\log d)^2 \sigma} \right\}$, then
\begin{align}
    \frac{1}{M} \ssum{m}{1}{M} \EB\{f(\hat{\vw}^m) - f(\vw^*)\} &\lesssim \frac{\tau^2 L h(\vw^*)}{R} + \frac{\tau h(\vw^*)^{1/2}(\log d)^2 \sigma}{(RK)^{1/2}}.
\end{align}
\end{theorem}

\subsection{Proof for Decentralized Federated Dual Averaging}

\input{Proofs/proof_for_defedda_conv}

\subsection{Proof for Decentralized FedDA with Gradient Tracking}
\input{Proofs/proof_for_defedda-rt_conv}

%% file: Proofs/proof_for_defedda_conv.tex
{

Starting from this section, we establish certain notations that will recurrently appear in the subsequent proofs.
\begin{align*}
    &\zeta_{r,k}^m = \nabla F(\vw_{r,k}^m, \xi_{r,k}^m) - \nabla f_m(\vw_{r,k}^m);\\
    &\vartheta_{r,k}^m = \nabla f_m(\vw_{r,k}^m) - \nabla f_m(\vw_{r,0}^m).
\end{align*}
Put simply, $\zeta_{r,k}^m$ signifies the noise introduced by the stochastic gradient, while $\vartheta_{r,k}^m$ captures the gradient's deviation from its initial state during the local update process.

\begin{proof}[Proof of Lemma~\ref{lem:innr_to_breg-norm}]
Define $m_r(\vv):= \ssum{i}{0}{r-1}a_i\inner{\zeta_i}{\vv} + h(\vv)$ with $m_0(\vv):=h(\vv)$. The strong convexity of $h$ implies that each $m_r$ is also 1-strongly convex (with respect to $\|\cdot\|_q$), so we have
\begin{align*}
    m_{r-1}(\vv) - m_{r-1}(\vv_{r-1}) \geq \frac{1}{2} \|\vv - \vv_{r-1}\|_q^2,\quad \forall \vv \in \RB^d.
\end{align*}
Taking $\vv=\vv_r$ yields
\begin{align*}
    0 &\leq m_{r-1}(\vv_r) - m_{r-1}(\vv_{r-1}) - \frac{1}{2}\|\vv_r - \vv_{r-1}\|_q^2 \\
    &= m_{r}(\vv_r) - a_{r-1}\inner{\zeta_{r-1}}{\vv_r} - m_{r-1}(\vv_{r-1}) - \frac{1}{2}\|\vv_r - \vv_{r-1}\|_q^2.
\end{align*}
Summing up the above inequality yields
\begin{align*}
    \sum_{i=0}^{r-1} a_i \inner{\zeta_i}{\vv_{i+1}} \leq m_r(\vv_r) - \frac{1}{2} \sum_{i=0}^{r-1} \norm{\vv_{i+1} - \vv_i}_q^2.
\end{align*}
On the other hand, we have
\begin{align*}
    \sum_{i=0}^{r-1} a_i \inner{\zeta_i}{-\vw^*} &\leq \max_\vw \left\{ \sum_{i=0}^{r-1} a_i \inner{\zeta_i}{-\vw} - h(\vw) \right\} + h(\vw^*) \\
    &= - \min_{\vw} \left\{ \sum_{i=0}^{r-1} a_i \inner{\zeta_i}{\vw} + h(\vw) \right\} + h(\vw^*) = -m_r(\vv_r) + h(\vw^*).
\end{align*}
Combining the above two inequalities concludes the proof.
\end{proof}

\begin{proof}[Proof of Theorem~\ref{thm:conv_rt_defedda}]
We denote $\rmZ_{r,k} = [\vz_{r,k}^1, \cdots, \vz_{r,k}^M]$ and $\rmW_{r,k} = [\vw_{r,k}^1, \cdots, \vw_{r,k}^M]$, and we denote the gossip matrix as $\rmU \in 
\RB^{M\times M}$.
Then the update rule of Algorithm~\ref{alg:dscaffoldda} is equivalent to:
\begin{enumerate}
    \item In the $r$-th round, for $k = 0, \cdots K-1$,
    \begin{align*}
        \rmW_{r,k} &= \left[ \nabla^{-1}h(\vz_{r,k}^1), \cdots, \nabla^{-1}h(\vz_{r,k}^M) \right], \\
        \rmZ_{r,k+1} &= \rmZ_{r,k} - \eta \rmG_{r,k}^{\xi} \text{ with } \rmG_{r,k}^{\xi} := \left[ \nabla F(\vw_{r,k}^1,\xi_{r,k}^1), \cdots, \nabla F(\vw_{r,k}^M, \xi_{r,k}^M) \right].
    \end{align*}
    \item From $r$ to $r+1$,
    \[\rmZ_{r+1,0} =  \rmZ_{r,K}\rmU. \]
\end{enumerate}
Then we can derive the relationship between $\rmZ_{r+1,0}$ and $\rmZ_{r,0}$,
\begin{align}
    \rmZ_{r+1,0} = \rmZ_{r,0}\rmU - \eta\ssum{k}{0}{K-1}\rmG_{r,k}^{\xi}\rmU. \label{eq:decentral_update_1}
\end{align}
Recall the double stochastic property of $\rmU$. Let $\Bar{\rmZ}_{r,k} = \rmZ_{r,k}\rmJ_M$ and let $\Tilde{\rmZ}_{r,k} = \rmZ_{r,k} - \Bar{\rmZ}_{r,k}$, then we have
\begin{align}
    \Bar{\rmZ}_{r+1, 0} &= \Bar{\rmZ}_{r,0} - \eta \left( \ssum{k}{0}{K-1} \rmG_{r,k}^{\xi} \right)\rmJ_M \quad\text{and}\quad \Bar{\rmZ}_{r,k}\rmJ_M = \Bar{\rmZ}_{r,k} = \Bar{\rmZ}_{r,k}\rmU. \label{eq:decentral_update_2}
\end{align}

We additionally define $\Bar{\rmW}_{r} = \nabla^{-1}h(\Bar{\rmZ}_{r})$ with $\nabla^{-1}h$ acting on every column of $\Bar{\rmZ}_r$.

We can do the following algebraic computation,
\begin{align*}
    &\ssum{m}{1}{M}(f_m(\Bar{\vw}_{r,0}) - f_m(\vw^*)) = \ssum{m}{1}{M}\left[ 
    f_m(\Bar{\vw}_{r,0}) - \frac{1}{K}\ssum{k}{0}{K-1} f_m(\vw_{r-1,k}^m) + \frac{1}{K}\ssum{k}{0}{K-1} f_m(\vw_{r-1,k}^m) - f_m(\vw^*)
    \right]\\
    &\le  \frac{1}{K}\ssum{m}{1}{M}\ssum{k}{0}{K-1}\left\{
    \inner{\nabla f_m(\vw_{r-1,k}^m)}{\Bar{\vw}_{r,0} - \vw_{r-1,k}^m}
    + \frac{L}{2}\norm{\Bar{\vw}_{r,0} - \vw_{r-1,k}^m}_1^2
    \right\}\\
    &+ \frac{1}{K}\ssum{m}{1}{M}\ssum{k}{0}{K-1}(f_m({\vw}_{r-1,k}^m) - f_m(\vw^*))\\
    &\le \frac{1}{K}\inner{\ssum{m}{1}{M}\ssum{k}{0}{K-1}\nabla f_m(\vw_{r-1,k}^m)}{\Bar{\vw}_{r,0} - \vw^*} + \frac{2L}{K}\ssum{m}{1}{M}\ssum{k}{0}{K-1}\norm{\Bar{\vw}_{r,0} - \vw_{r-1,k}^m}_1^2.
\end{align*}
Define $\Bar{\vg}_{r} = \frac{1}{MK}\ssum{m}{1}{M}\ssum{k}{0}{K-1}\nabla f_m(\vw_{r,k}^m)$ and $\Bar{\vg}_r^{\xi} = \frac{1}{MK}\ssum{m}{1}{M}\ssum{k}{0}{K-1}\nabla F(\vw_{r,k}^m, \xi_{r,k}^m)$. For simplicity, we also let $\Tilde{\eta} = K\eta$.
Since $\Bar{\vz}_{r,0} = \Bar{\vz}_{r-1,0} - \Tilde{\eta}\vg_{r-1}^\xi$, we have $\Bar{\vw}_{r,0} = \arg\min\left\{\inner{\ssum{i}{0}{r-1}\Tilde{\eta}\Bar{\vg}_i^\xi}{\vw} + h(\vw)\right\}$. Thus by Lemma~\ref{lem:innr_to_breg-norm}, we have
\[
\ssum{r}{0}{R-1}\Tilde{\eta}\inner{\Bar{\vg}_{r}^\xi}{\Bar{\vw}_{r+1,0} - \vw^*} \le h(\vw^*) - \frac{1}{2}\ssum{r}{0}{R-1}\norm{\Bar{\vw}_{r,0} - \Bar{\vw}_{r+1,0}}_q^2.
\]
Combine the above results yields
\begin{align}
    &\frac{1}{MR}\ssum{r}{1}{R}\ssum{m}{1}{M}(f_m(\Bar{\vw}_{r,0}) - f_m(\vw^*)) \le \frac{1}{R\Tilde{\eta}}\left\{ 
    h(\vw^*) - \ssum{r}{1}{R}\frac{1}{2}\norm{\Bar{\vw}_{r-1,0} - \Bar{\vw}_{r,0}}_q^2\right\}\\
    &+ \frac{2L}{RMK}\ssum{r}{1}{R}\ssum{m}{1}{M}\ssum{k}{0}{K-1}\norm{\Bar{\vw}_{r,0} - {\vw}_{r-1,k}^m}_1^2 + \frac{1}{R}\ssum{r}{1}{R}\inner{\Bar{\vg}_{r-1} - \Bar{\vg}_{r-1}^\xi}{\Bar{\vw}_{r,0} - \vw^*}. \label{eq:dfedda_1}
\end{align}
Note by Fenchel-Young's inequality and the definition of conditional expectation,
\begin{align*}
    &\inner{\Bar{\vg}_{r-1} - \Bar{\vg}_{r-1}^\xi}{\Bar{\vw}_{r,0} - \Bar{\vw}_{r-1,0}}
    \le \Tilde{\eta}\norm{\Bar{\vg}_{r-1}^\xi - \Bar{\vg}_{r-1}}_p^2 + \frac{1}{4\Tilde{\eta}}\norm{\Bar{\vw}_{r,0} - \Bar{\vw}_{r-1,0}}_q^2, \\
    &\E\left[ \left. \inner{\Bar{\vg}_{r-1} - \Bar{\vg}_{r-1}^\xi}{\Bar{\vw}_{r-1,0}-\vw^*} \right| \gF_{r-1} \right] = 0.
\end{align*}
Therefore, taking expectation of \eqref{eq:dfedda_1} yields
\begin{equation}\label{eq:dfedda_expand}
\begin{aligned}
    &\frac{1}{MR}\ssum{r}{1}{R}\ssum{m}{1}{M}\EB [f_m(\Bar{\vw}_{r,0}) - f_m(\vw^*)] 
    \le \frac{h(\vw^*)}{R\Tilde{\eta}} - \frac{1}{2R\Tilde{\eta}}\ssum{r}{1}{R}\EB\norm{\Bar{\vw}_{r-1,0} - \Bar{\vw}_{r,0}}_q^2\\
    &{+} \frac{2L}{RMK}\ssum{r}{1}{R}\ssum{m}{1}{M}\ssum{k}{0}{K-1} \EB \norm{\Bar{\vw}_{r,0} {-} {\vw}_{r-1,k}^m}_1^2 {+} \frac{1}{R}\ssum{r}{1}{R} \left( \Tilde{\eta}\EB \norm{\Bar{\vg}_{r-1}^\xi {-} \Bar{\vg}_{r-1}}_p^2 {+} \frac{1}{4\Tilde{\eta}}\E\|\Bar{\vw}_{r,0} {-} \Bar{\vw}_{r-1,0}\|_q^2\right)\\
    &= \frac{h(\vw^*)}{R\Tilde{\eta}} - \frac{1}{4R\Tilde{\eta}}\ssum{r}{1}{R}\EB\norm{\Bar{\vw}_{r-1,0} - \Bar{\vw}_{r,0}}_q^2\\
    &+ \frac{2L}{RMK}\ssum{r}{1}{R}\ssum{m}{1}{M}\ssum{k}{0}{K-1} \EB \norm{\Bar{\vw}_{r,0} - {\vw}_{r-1,k}^m}_1^2 + \frac{\Tilde{\eta}}{R}\ssum{r}{0}{R-1}\EB \norm{\Bar{\vg}_{r}^\xi - \Bar{\vg}_{r}}_p^2\\
    &\le \frac{h(\vw^*)}{R\Tilde{\eta}} - \frac{1}{4R\Tilde{\eta}}\ssum{r}{1}{R}\EB\norm{\Bar{\vw}_{r-1,0} - \Bar{\vw}_{r,0}}_q^2 + \frac{\Tilde{\eta}}{R}\ssum{r}{0}{R-1}\EB\norm{\Bar{\vg}_{r}^\xi - \Bar{\vg}_{r}}_p^2\\
    & + \frac{2L}{RMK}\ssum{r}{1}{R}\ssum{m}{1}{M}\ssum{k}{0}{K-1}\left( 2\EB \norm{\Bar{\vw}_{r-1,0} - {\vw}_{r-1,k}^m}_1^2 + 2\EB \norm{\Bar{\vw}_{r,0} - \Bar{\vw}_{r-1,0}}_1^2 \right)\\
    &\le \frac{h(\vw^*)}{R\Tilde{\eta}} - \frac{1}{4R\Tilde{\eta}}\ssum{r}{1}{R}\EB\norm{\Bar{\vw}_{r-1,0} - \Bar{\vw}_{r,0}}_q^2 + \frac{\Tilde{\eta}}{R}\ssum{r}{0}{R-1}\EB \norm{\Bar{\vg}_{r}^\xi - \Bar{\vg}_{r}}_p^2\\
    & + \frac{4L}{RMK}\ssum{r}{1}{R}\ssum{m}{1}{M}\ssum{k}{0}{K-1}\EB \norm{\Bar{\vw}_{r-1,0} - {\vw}_{r-1,k}^m}_1^2 + \frac{4d^{2/p}L}{R}\ssum{r}{1}{R} \EB \norm{\Bar{\vw}_{r,0} - \Bar{\vw}_{r-1,0}}_q^2\\
    &\le \frac{h(\vw^*)}{R\Tilde{\eta}} - \frac{1}{8R\Tilde{\eta}}\ssum{r}{1}{R}\EB\norm{\Bar{\vw}_{r-1,0} - \Bar{\vw}_{r,0}}_q^2 + \frac{\Tilde{\eta}}{R}\ssum{r}{0}{R-1} \EB \norm{\Bar{\vg}_{r}^\xi - \Bar{\vg}_{r}}_p^2\\
    & + \frac{4L}{RMK}\ssum{r}{1}{R}\ssum{m}{1}{M}\ssum{k}{0}{K-1} \EB \norm{\Bar{\vw}_{r-1,0} - {\vw}_{r-1,k}^m}_1^2, 
\end{aligned}
\end{equation}
where the last inequality holds when $d^{2/p}L\Tilde{\eta} \le 1/32$. The gradient noise term $\EB\norm{\Bar{\vg}_r^\xi - \Bar{\vg}_r}_p^2$ can be bounded as the following,
\begin{align}
    \EB\norm{\Bar{\vg}_r^\xi - \Bar{\vg}_r}_p^2 &= \EB \norm{\frac{1}{MK}\ssum{m}{1}{M}\ssum{k}{0}{K-1} (\nabla F(\vw_{r,k}^m, \xi_{r,k}^m) - \nabla f_m(\vw_{r,k}^m)) }_p^2 = \frac{1}{M^2K^2} \EB \norm{\ssum{m}{1}{M}\ssum{k}{0}{K-1} \zeta_{r,k}^m }_p^2 \\
    &\leq \frac{1}{MK^2} \ssum{m}{1}{M} \E\norm{\ssum{k}{0}{K-1} \zeta_{r,k}^m}_p^2 \leq \frac{C^2 d^{2/p} p^3 K}{MK^2} \ssum{m}{1}{M} \sup\limits_{i\in [n], j\in [d]} \left\{\EB \left| [\zeta_{r,i}^m]_j \right|^p \right\}^{2/p} \\
    &\le \frac{C^2d^{2/p}p^4 \sigma^2}{K}. \label{eq:dfedda_grad_noise}
\end{align}

Next, we focus on the last term of \eqref{eq:dfedda_expand}, i.e., the deviation between $\Bar{\vw}_{r-1,0}$ and $\vw_{r-1,k}^m$.
By the Cauchy-Schwartz inequality,
\begin{align}
    \EB \norm{\Bar{\vw}_{r,0} - \vw_{r,k}^m}_1^2 \le d^{2/p}\EB \norm{\vz_{r,k}^m - \Bar{\vz}_{r,0}}_p^2 \le 2d^{2/p}\EB \norm{\vz_{r,k}^m - \vz_{r,0}^m}_p^2 +
2d^{2/p}\EB \norm{\vz_{r,0}^m - \Bar{\vz}_{r,0}}_p^2. \label{eq:dfedda_last_term}
\end{align}

Using Lemma~\ref{lem:local_dev_dfedda_1}, we have
\begin{align}
\gE_r = \gE_{r,K} &\le 4ed^{4/p}L^2 K^3 \eta^2 \gH_r + 4eC^2d^{2/p}p^4 \sigma^2 MK^2\eta^2\\
&+ 16ed^{2/p}LMK^3\eta^2 \EB \{f(\Bar{\vw}_{r,0}) - f(\vw^*)\} + 8eK^3\eta^2 \ssum{m}{1}{M}\norm{\nabla f_m(\vw^*)}_p^2,\label{eq:bd_of_ge_r_to_gh}
\end{align}
as long as $2d^{2/p}LK\eta < 1$.

Now, combining \eqref{eq:bd_of_ge_r_to_gh} with Lemma~\ref{lem:local_dev_dfedda_2} yields
\begin{align*}
    &\gH_r \le \frac{1}{2}\gH_{r-\tau} + 80d^{2/p}\tau LM\Tilde{\eta}^2 (1 + 8ed^{4/p}L^2 \Tilde{\eta}^2)\ssum{i}{r-\tau}{r-1} \EB \left\{ 
    f(\Bar{\vw}_{i,0}) - f(\vw^*)
    \right\}\\
    &+ 40d^{4/p} \tau L^2\Tilde{\eta}^2 (1+ 4ed^{4/p}L^2 \Tilde{\eta}^2)\ssum{i}{r-\tau}{r-1}\gH_i
    + \frac{40C^4p^4 d^{2/p}\tau^2 \sigma^2 M\Tilde{\eta}^2}{K}(1 + 4ed^{4/p}L^2 \Tilde{\eta}^2)\\
    &+ 40\tau^2 \Tilde{\eta}^2 (1+8ed^{4/p}L^2\Tilde{\eta}^2) \ssum{m}{1}{M} \norm{\nabla f_m(\vw^*)}_p^2\\
    &\overset{(a)}{\le} \frac{1}{2}\gH_{r-\tau} + 80d^{4/p} \tau L^2\Tilde{\eta}^2 \ssum{i}{r-\tau}{r-1}\gH_i
    + 160d^{2/p}\tau ML\Tilde{\eta}^2 \ssum{i}{r-\tau}{r-1} \EB \{ f(\Bar{\vw}_{i,0}) - f(\vw^*) \}\\
    &+ \frac{80C^4p^4d^{2/p}\tau^2 \sigma^2 M\Tilde{\eta}^2}{K} + 80\tau^2 \Tilde{\eta}^2 \ssum{m}{1}{M}\norm{\nabla f_m(\vw^*)}_p^2,
\end{align*}
where $(a)$ holds as long as $8ed^{4/p}L^2 K^2 \eta^2 \le 1$. By summing both sides of the above inequality with respect to $r$, we can obtain,

\begin{align*}
    &\ssum{i}{1}{\tau} \gH_{(s+1)\tau + i} \le \frac{1}{2} \ssum{i}{1}{\tau} \gH_{s\tau + i} + 80d^{4/p}\tau L^2\Tilde{\eta}^2 \ssum{i}{1}{\tau}\ssum{j}{s\tau + i}{(s+1)\tau + i -1}\gH_j  + 80\tau^3 \Tilde{\eta}^2 \ssum{m}{1}{M}\norm{\nabla f_m(\vw^*)}_p^2\\
    &+ 160 d^{2/p}\tau ML\Tilde{\eta}^2 \ssum{i}{1}{\tau}\ssum{j}{s\tau + i}{(s+1)\tau + i -1}\EB \{
    f(\Bar{\vw}_{i,0}) - f(\vw^*)
    \} + \frac{80C^4p^4d^{2/p}\tau^3 \sigma^2 M\Tilde{\eta}^2}{K}\\
    &\le \frac{1}{2}\ssum{i}{1}{\tau} \gH_{s\tau + i} + 80d^{4/p}\tau^2 L^2 \Tilde{\eta}^2 \left\{\ssum{i}{1}{\tau -1}\gH_{(s+1)\tau + i} + \ssum{i}{1}{\tau}\gH_{s\tau + i}\right\} + \frac{80C^4p^4d^{2/p}\tau^3 \sigma^2 M\Tilde{\eta}^2}{K}\\
    &+ 160d^{2/p}\tau^2 ML\Tilde{\eta}^2 \left\{\ssum{i}{1}{\tau - 1}\EB \{f(\Bar{\vw}_{(s+1)\tau + i,0}) - f(\vw^*)\} + \ssum{i}{1}{\tau} \EB \{ f(\Bar{\vw}_{s\tau + i,0}) - f(\vw^*)\}\right\}\\
    &+ 80\tau^3 \Tilde{\eta}^2 \ssum{m}{1}{M}\norm{\nabla f_m(\vw^*)}_p^2.
\end{align*}
Since $9d^{2/p}\tau L\Tilde{\eta} \le \frac{1}{4}$, we have $\frac{1}{2} + 80d^{4/p}\tau^2 L^2 \Tilde{\eta}^2 \le \frac{1}{2} + \frac{1}{16} = \frac{9}{16}$ and $ 1 - 80d^{4/p}\tau^2 L^2 \Tilde{\eta}^2 \ge \frac{15}{16}$. Thus,
\begin{align*}
    &\frac{15}{16}\ssum{i}{1}{\tau}\gH_{(s+1)\tau + i} \le \frac{9}{16}\ssum{i}{1}{\tau}\gH_{s\tau + i} + \frac{80C^4p^4d^{2/p}\tau^3 \sigma^2 M\Tilde{\eta}^2}{K} + 80\tau^3 \Tilde{\eta}^2 \ssum{m}{1}{M}\norm{\nabla f_m(\vw^*)}_p^2\\
    &+ 160d^{2/p}\tau^2 ML\Tilde{\eta}^2 \left\{\ssum{i}{1}{\tau - 1}\EB \{f(\Bar{\vw}_{(s+1)\tau + i,0}) - f(\vw^*)\} + \ssum{i}{1}{\tau} \EB \{ f(\Bar{\vw}_{s\tau + i,0}) - f(\vw^*)\}\right\}.
\end{align*}

Denote $\ssum{i}{1}{\tau} \gH_{s\tau + i}$ as $\gH_s^{1:\tau}$, $\ssum{i}{1}{\tau}\EB \{f(\Bar{\vw}_{s\tau +i,0}) - f(\vw^*)\}$ as $\rmF_s$ and $\frac{1}{M}\ssum{m}{1}{M} \norm{\nabla f_m(\vw^*)}_p^2$ as $\gE_M^*$. Rearranging the above inequality yields
\begin{align*}
    \gH_{s+1}^{1:\tau} &\le \frac{3}{5}\gH_s^{1:\tau} + 320d^{2/p}\tau^2 ML\Tilde{\eta}^2 (\rmF_{s+1} + \rmF_s) + \frac{160C^4p^4d^{2/p}\tau^3 \sigma^2 M\Tilde{\eta}^2}{K} + 160\tau^3 M \Tilde{\eta}^2 \gE_M^*.
\end{align*}
By induction we obtain
\begin{equation}\label{eq:H_s_1:tau}
\begin{aligned}
    \gH_s^{1:\tau} &\leq \left(\frac{3}{5}\right)^{s} \gH_0^{1:\tau} + 320d^{2/p}\tau^2 ML\Tilde{\eta}^2 \cdot \frac{8}{5} \ssum{i}{0}{s} \left(\frac{3}{5}\right)^{s-1-i} \rmF_i \\
    &\quad + \left(\frac{160C^4p^4d^{2/p}\tau^3 \sigma^2 M\Tilde{\eta}^2}{K} + 160\tau^3 M \Tilde{\eta}^2 \gE_M^*\right)\ssum{i}{0}{s-1}\left(\frac{3}{5}\right)^{i} \\
    &\leq \left(\frac{3}{5}\right)^{s} \gH_0^{1:\tau} + 900d^{2/p}\tau^2 ML\Tilde{\eta}^2 \ssum{i}{0}{s} \left(\frac{3}{5}\right)^{s-i} \rmF_i + \left(\frac{400C^4p^4d^{2/p}\tau^3 \sigma^2 M\Tilde{\eta}^2}{K} + 400\tau^3 M \Tilde{\eta}^2 \gE_M^*\right). 
\end{aligned}
\end{equation}

Now, combining \eqref{eq:dfedda_grad_noise}, \eqref{eq:dfedda_last_term}, \eqref{eq:bd_of_ge_r_to_gh} with \eqref{eq:dfedda_expand}, if $R=b\tau$ then we have
\begin{align*}
    \frac{1}{R}\ssum{s}{0}{b-1} \rmF_s &= \frac{1}{MR}\ssum{r}{1}{R}\ssum{m}{1}{M}\EB [f_m(\Bar{\vw}_{r,0}) - f_m(\vw^*)] \\
    &\overset{\eqref{eq:dfedda_expand}}{\leq} \frac{h(\vw^*)}{R\Tilde{\eta}} + \frac{\Tilde{\eta}}{R}\ssum{r}{0}{R-1} \EB \norm{\Bar{\vg}_{r}^\xi - \Bar{\vg}_{r}}_p^2 + \frac{4L}{RMK}\ssum{r}{1}{R}\ssum{m}{1}{M}\ssum{k}{0}{K-1} \EB \norm{\Bar{\vw}_{r-1,0} - {\vw}_{r-1,k}^m}_1^2 \\&\overset{\eqref{eq:dfedda_grad_noise},\eqref{eq:dfedda_last_term}}{\leq} \frac{h(\vw^*)}{R\Tilde{\eta}} + \frac{C^2d^{2/p}p^4 \sigma^2 \Tilde{\eta}}{K} + \frac{8d^{2/p}L}{RMK}\ssum{r}{0}{R-1}\ssum{m}{1}{M}\ssum{k}{0}{K-1} \left( \EB \norm{\vz_{r,k}^m - \vz_{r,0}^m}_p^2 + \EB \norm{\vz_{r,0}^m - \Bar{\vz}_{r,0}}_p^2 \right) \\
    &= \frac{h(\vw^*)}{R\Tilde{\eta}} + \frac{C^2d^{2/p}p^4 \sigma^2 \Tilde{\eta}}{K} + \frac{8d^{2/p}L}{RMK}\ssum{r}{0}{R-1}(\gE_r + K\gH_r) \\
    &\overset{\eqref{eq:bd_of_ge_r_to_gh}}{\leq} \frac{h(\vw^*)}{R\Tilde{\eta}} + \frac{C^2d^{2/p}p^4 \sigma^2 \Tilde{\eta}}{K} + \frac{8d^{2/p}L}{RM}(1+4ed^{4/p}L^2 \Tilde{\eta}^2)\ssum{r}{0}{R-1}\gH_r \\
    &\quad + \frac{128ed^{4/p}L^2\Tilde{\eta}^2}{R} \ssum{r}{0}{R-1} \EB\{f(\Bar{\vw}_{r,0}) - f(\vw^*)\} + \frac{32eC^2d^{4/p}p^4\sigma^2L\Tilde{\eta}^2}{K} + 64ed^{2/p} L\Tilde{\eta}^2 \gE_M^*.
\end{align*}

Through further integration, we arrive at
\begin{equation}\label{eq:dfedda_Fs}
\begin{aligned}
    &\frac{1}{R}\ssum{s}{0}{b-1} \rmF_s 
    \le \frac{h(\vw^*)}{R\Tilde{\eta}} + \frac{C^2d^{2/p}p^4 \sigma^2 \Tilde{\eta}}{K}(1 + 32ed^{2/p}L\Tilde{\eta}) + 64ed^{2/p} L\Tilde{\eta}^2 \gE_M^* \\
    &\quad + \frac{8d^{2/p}L}{RM}(1+4ed^{4/p}L^2 \Tilde{\eta}^2)\ssum{r}{0}{R-1}\gH_r + \frac{128ed^{4/p}L^2\Tilde{\eta}^2}{R} \left( \{f(\vw_0) - f(\vw^*)\} + \ssum{s}{0}{b-1} \rmF_s \right) \\
    &\leq \frac{h(\vw^*)}{R\Tilde{\eta}} + \frac{2C^2d^{2/p}p^4 \sigma^2 \Tilde{\eta}}{K} + 64ed^{2/p} L\Tilde{\eta}^2 \gE_M^* + \frac{16d^{2/p}L}{RM}\ssum{r}{0}{R-1}\gH_r \\
    &\quad + \frac{1}{2R} \left( \{f(\vw_0) - f(\vw^*)\} + \ssum{s}{0}{b-1} \rmF_s \right). 
\end{aligned}
\end{equation}
Here the last inequality holds as long as $32ed^{2/p}L\Tilde{\eta} \leq 1$. In addition, \eqref{eq:H_s_1:tau} implies
\begin{align*}
    \ssum{r}{0}{R-1} \gH_r &\leq \gH_0 + \ssum{s}{0}{b-1} \gH_s^{1:\tau} \\
    &\leq \gH_0 + \ssum{s}{0}{b-1} \left(\frac{3}{5}\right)^{s} \gH_0^{1:\tau} + 900d^{2/p}\tau^2 ML\Tilde{\eta}^2 \ssum{s}{0}{b-1} \ssum{i}{0}{s} \left(\frac{3}{5}\right)^{s-i} \rmF_i \\
    &\quad +  \ssum{s}{0}{b-1} \left(\frac{400C^4p^4d^{2/p}\tau^3 \sigma^2 M\Tilde{\eta}^2}{K} + 400\tau^3 M \Tilde{\eta}^2 \gE_M^*\right) \\
    &\leq 4\gH_0^{1:\tau} + 2250 d^{2/p}\tau^2 ML\Tilde{\eta}^2 \ssum{s}{0}{b-1} \rmF_s + \left(\frac{400C^4p^4d^{2/p}\tau^2 \sigma^2 MR\Tilde{\eta}^2}{K} + 400\tau^2 MR \Tilde{\eta}^2 \gE_M^*\right).
\end{align*}
When $3200C^2\tau d^{2/p} L\Tilde{\eta} \leq 1$, plugging the above inequality into \eqref{eq:dfedda_Fs} yields
\begin{align*}
    \frac{1}{R}\ssum{s}{0}{b-1} \rmF_s &\leq \frac{h(\vw^*)}{R\Tilde{\eta}} + \frac{2C^2d^{2/p}p^4 \sigma^2 \Tilde{\eta}}{K} + 64ed^{2/p} L\Tilde{\eta}^2 \gE_M^* + \frac{1}{2R} \left( \{f(\vw_0) - f(\vw^*)\} + \ssum{s}{0}{b-1} \rmF_s \right) \\
    &\quad + \frac{64d^{2/p}L\gH_0^{1:\tau}}{RM} + \frac{36000 d^{4/p}\tau^2 L^2\Tilde{\eta}^2}{R} \ssum{s}{0}{b-1} \rmF_s + \frac{6400C^4p^4d^{4/p}\tau^2 \sigma^2 L\Tilde{\eta}^2}{K} + 6400 d^{2/p} \tau^2 L \Tilde{\eta}^2 \gE_M^* \\
    &\leq \frac{h(\vw^*)}{R\Tilde{\eta}} + \frac{\{f(\vw_0) - f(\vw^*)\}}{2R} + \frac{64d^{2/p}L\gH_0^{1:\tau}}{RM} + \frac{3}{4R}\ssum{s}{0}{b-1} \rmF_s \\
    &\quad + \frac{4C^2d^{2/p}\tau p^4 \sigma^2 \Tilde{\eta}}{K} + 6600d^{2/p}\tau^2 L\Tilde{\eta}^2 \gE_M^*.
\end{align*}
Further simplification yields
\begin{equation}
\begin{aligned}
    &\frac{1}{R}\ssum{r}{1}{R} \EB \{f(\Bar{\vw}_{r,0}) - f(\vw^*)\} = \frac{1}{R}\ssum{s}{0}{b-1} \rmF_s \\
    &\leq \frac{4h(\vw^*)}{R\Tilde{\eta}} + \frac{2\{f(\vw_0) - f(\vw^*)\}}{R} + \frac{256d^{2/p}L\gH_0^{1:\tau}}{RM} + \frac{16 C^2d^{2/p}\tau p^4 \sigma^2 \Tilde{\eta}}{K} + 26400d^{2/p}\tau^2 L\Tilde{\eta}^2 \gE_M^* \\
    &\lesssim \frac{1}{R\Tilde{\eta}} + \frac{d^{2/p}\tau p^4 \sigma^2 \Tilde{\eta}}{K} + d^{2/p}\tau^2 L \gE_M^* \Tilde{\eta}^2. \label{eq:dfedda_f-f1}
\end{aligned}
\end{equation}

Note, however, that $\Bar{\vw}_{r,0}$ is inaccessible since computing it needs aggregation of $\vz_{r,0}^m$ on each client. Therefore, we prefer to obtain a convergence rate with respect to accessible parameters $\vw_{r,0}^m$. We have
\begin{align*}
    f(\vw_{r,0}^m) - f(\Bar{\vw}_{r,0}) &\leq \inner{\nabla f(\Bar{\vw}_{r,0})}{\vw_{r,0}^m - \Bar{\vw}_{r,0}} + \frac{L}{2} \norm{\vw_{r,0}^m - \Bar{\vw}_{r,0}}_{1}^2 \\
    &= \inner{\nabla f(\Bar{\vw}_{r,0}) - \nabla f(\vw^*)}{\vw_{r,0}^m - \Bar{\vw}_{r,0}} + \frac{d^{2/p}L}{2} \norm{\vw_{r,0}^m - \Bar{\vw}_{r,0}}_{q}^2 \\
    &\leq \frac{1}{2d^{2/p}L} \norm{\nabla f(\Bar{\vw}_{r,0}) - \nabla f(\vw^*)}_p^2 + \frac{d^{2/p}L}{2} \norm{\vw_{r,0}^m - \Bar{\vw}_{r,0}}_q^2 + \frac{d^{2/p}L}{2} \norm{\vw_{r,0}^m - \Bar{\vw}_{r,0}}_q^2 \\
    &\leq f(\Bar{\vw}_{r,0}) - f(\vw^*) + d^{2/p} L \norm{\vw_{r,0}^m - \Bar{\vw}_{r,0}}_q^2.
\end{align*}
Taking expectation and summing over $m$ and $r$, we get
\begin{equation}\label{eq:dfedda_f-f2}
\begin{aligned}
    &\frac{1}{MR} \ssum{m}{1}{M} \ssum{r}{1}{R} \EB \{ f(\vw_{r,0}^m) - f(\Bar{\vw}_{r,0}) \} \\
    &\leq \frac{1}{MR} \ssum{m}{1}{M} \ssum{r}{1}{R} \left\{ \EB \{f(\Bar{\vw}_{r,0}) - f(\vw^*)\} + d^{2/p} L \EB \norm{\vw_{r,0}^m - \Bar{\vw}_{r,0}}_q^2 \right\} \\
    &\leq \frac{1}{R} \ssum{r}{1}{R} \EB \{f(\Bar{\vw}_{r,0}) - f(\vw^*)\} + \frac{d^{2/p}L}{MR} \ssum{r}{1}{R} \gH_r \lesssim \frac{1}{R} \ssum{r}{1}{R} \EB \{f(\Bar{\vw}_{r,0}) - f(\vw^*)\} \\
    &\quad + \frac{d^{2/p}L}{MR} \left[1 + d^{2/p}\tau^2 ML\Tilde{\eta}^2 \ssum{r}{1}{R} \EB \{f(\Bar{\vw}_{r,0}) - f(\vw^*)\} + \frac{p^4 d^{2/p} \tau^2 \sigma^2 MR \Tilde{\eta}^2}{K} + \tau^2 MR \Tilde{\eta}^2 \gE_M^*\right] \\
    &\lesssim \frac{1}{R}\ssum{r}{1}{R} \EB \{f(\Bar{\vw}_{r,0}) - f(\vw^*)\} + \frac{d^{2/p}L}{MR} + \frac{p^4 d^{4/p} \tau^2 \sigma^2 L \Tilde{\eta}^2}{K} + d^{2/p}\tau^2 L\gE_M^*\Tilde{\eta}^2. 
\end{aligned}
\end{equation}
Define $\hat{\vw}^m:= \frac{1}{R} \ssum{r}{1}{R} \vw_{r,0}^m$. By the convexity of $f$ and combination of \eqref{eq:dfedda_f-f1} and \eqref{eq:dfedda_f-f2}, we finally get
\begin{align*}
    \frac{1}{M} \ssum{m}{1}{M} \EB\{f(\hat{\vw}^m) - f(\vw^*)\} &\leq  \frac{1}{MR} \ssum{m}{1}{M} \ssum{r}{1}{R} \EB \{ f(\vw_{r,0}^m) - f(\Bar{\vw}_{r,0}) + f(\Bar{\vw}_{r,0}) - f(\vw^*)\} \\
    &\lesssim \frac{1}{R\Tilde{\eta}}\left(1 + \frac{d^{2/p}L\Tilde{\eta}}{M}\right) + \frac{d^{2/p}\tau p^4 \sigma^2 \Tilde{\eta}}{K}(1 + \tau d^{2/p} L \Tilde{\eta}) + d^{2/p}\tau^2 L \gE_M^* \Tilde{\eta}^2 \\
    &\lesssim \frac{1}{R\Tilde{\eta}} + \frac{d^{2/p}\tau p^4 \sigma^2 \Tilde{\eta}}{K} + d^{2/p}\tau^2 L \gE_M^* \Tilde{\eta}^2 \\
    &\lesssim \frac{1}{R\Tilde{\eta}} + \frac{(\log d)^4 \tau \sigma^2 \Tilde{\eta}}{K} + \tau^2 L\gE_M^* \Tilde{\eta}^2,
\end{align*}
as long as $\tau d^{2/p} L\Tilde{\eta} \lesssim 1$ and $p=2 \log d$. If we further choose $\Tilde{\eta} = \min\left\{ \frac{1}{\tau L}, \frac{1}{(\tau^2 LR \gE_M^*)^{1/3}}, \frac{K^{1/2}}{(\tau R)^{1/2}(\log d)^2 \sigma} \right\}$, then
\begin{align*}
    \frac{1}{M} \ssum{m}{1}{M} \EB\{f(\hat{\vw}^m) - f(\vw^*)\} &\lesssim \frac{\tau L}{R} + \frac{(\tau^2 L \gE_M^*)^{1/3}}{R^{2/3}} + \frac{\tau^{1/2}(\log d)^2 \sigma}{(RK)^{1/2}},
\end{align*}
which concludes the proof.
\end{proof}

\subsubsection{Proof for One-Step Deviation Lemmas}\label{sec:prf_for_local_dev_dfed}

\begin{proof}[Proof of Lemma~\ref{lem:local_dev_dfedda_1}]
Similar to the former analysis, we can easily get
\begin{align*}
    &\EB \norm{\vz_{r,k}^m - \vz_{r,0}^m}_p^2 = \eta^2 \EB \norm{\ssum{i}{0}{k-1}\nabla F(\vw_{r,i}^m , \xi_{r,i}^m)}_p^2\\
    =&
    \eta^2 \EB \norm{
    \ssum{i}{0}{k-1}\zeta_{r,i}^m + \ssum{i}{0}{k-1}\vartheta_{r,i}^m
    + k(\nabla f_m(\vw_{r,0}^m) - \nabla f_m(\Bar{\vw}_{r,0})) + k \nabla f_m(\Bar{\vw}_{r,0})
    }_p^2\\
    \le&
    4\eta^2 \EB \norm{\ssum{i}{0}{k-1}\zeta_{r,i}^m}_p^2 + 4\eta^2 \EB \norm{\ssum{i}{0}{k-1}\vartheta_{r,i}^m}_p^2 + 4k^2 \eta^2 \EB \norm{\nabla f_m(\Bar{\vw}_{r,0})}_p^2\\
    &+ 4k^2 \eta^2 \EB \norm{\nabla f_m(\vw_{r,0}^m) - \nabla f_m(\Bar{\vw}_{r,0})}_p^2\\
    \le&
     4\eta^2 \EB \norm{\ssum{i}{0}{k-1}\zeta_{r,i}^m}_p^2 + 4d^{4/p}L^2k\eta^2\ssum{i}{0}{k-1}\EB \norm{\vz_{r,i}^m - \vz_{r,0}^m}_p^2 + 4d^{4/p}L^2k^2 \eta^2 \EB \norm{\vz_{r,0}^m - \Bar{\vz}_{r,0}}_p^2\\
    & + 8k^2 \eta^2 \norm{\nabla f_m(\vw^*)}_p^2 + 8k^2 \eta^2 \EB \norm{\nabla f_m(\Bar{\vw}_{r,0}) - \nabla f_m(\vw^*)}_p^2.
\end{align*}
Owing to Lemma~\ref{lem:bd_of_p_norm_2nd_mmt} and Lemma~\ref{lem:mmt_bd_subG}, we have
\[
\EB \norm{\ssum{i}{0}{k-1} \zeta_{r,i}^m}_p^2 \le C^2 d^{2/p} p^3 k \sup\limits_{i\in [n], j\in [d]} \left\{\EB \left| [\zeta_{r,i}^m]_j \right|^p \right\}^{2/p} \le 
C^2d^{2/p}p^4 k \sigma^2.
\]
Furthermore, using Lemma~\ref{lem:conj_sc}, we can get
\[\EB \norm{\nabla f_m(\Bar{\vw}_{r,0}) - \nabla f_m(\vw^*)}_p^2 \le d^{2/p} \gD_{r}.
\]
Combining the above inequalities yields
\begin{align*}
        &\EB \norm{\vz_{r,k}^m - \vz_{r,0}^m}_p^2 \le 4d^{4/p}L^2k\eta^2\ssum{i}{0}{k-1}\EB \norm{\vz_{r,i}^m - \vz_{r,0}^m}_p^2 + 4d^{4/p}L^2 k^2 \eta^2 \EB \norm{\vz_{r,0}^m - \Bar{\vz}_{r,0}}_p^2\\
        &+ 4C^2 d^{2/p}p^4 \sigma^2 k\eta^2 + 8d^{2/p}k^2 \eta^2 \gD_r^m + 8k^2 \eta^2 \norm{\nabla f_m(\vw^*)}_p^2.
\end{align*}

Given Lemma~\ref{lem:local_dev_dfedda_1}, one can check the following inequality by summing up on the index $m$ from $1$ to $M$,
\begin{align*}
    &\ssum{m}{1}{M}\EB \norm{\vz_{r,k}^m - \vz_{r,0}^m}_p^2 \le 4d^{4/p}L^2k\eta^2 \ssum{i}{0}{k-1} \ssum{m}{1}{M} \EB \norm{\vz_{r,i}^m - \vz_{r,0}^m}_p^2 + 4d^{4/p}L^2k^2\eta^2 \gH_r\\
    &+ 4C^2d^{2/p}p^4 \sigma^2 Mk\eta^2 + 16d^{2/p}LMk^2\eta^2 \EB \{f(\Bar{\vw}_{r,0}) - f(\vw^*)\} + 8k^2\eta^2 \ssum{m}{1}{M}\norm{\nabla f_m(\vw^*)}_p^2.
\end{align*}
Let $\gE_{r,k} = \ssum{i}{0}{k-1}\ssum{m}{1}{M}\EB \norm{\vz_{r,i}^m - \vz_{r,0}^m}_p^2$, then
\begin{align*}
    &\gE_{r,k+1} \le (1 + 4d^{4/p}L^2k\eta^2) \gE_{r,k} + 4d^{4/p}L^2k^2\eta^2 \gH_r
    + 4C^2d^{2/p}p^4 \sigma^2 Mk\eta^2\\
    &+ 16d^{2/p}LMk^2\eta^2 \EB \{f(\Bar{\vw}_{r,0}) - f(\vw^*)\} + 8k^2\eta^2 \ssum{m}{1}{M}\norm{\nabla f_m(\vw^*)}_p^2.
\end{align*}
And the proof is concluded by making use of Lemma~\ref{lem:rcs_to_etr_fedda}.
\end{proof}

For ease to present the proof of Lemma~\ref{lem:local_dev_dfedda_2}, we use the following notation:
\begin{align*}
    &\rmG_i^\xi := \left[ 
    \ssum{k}{0}{K-1} \nabla F(\vw_{i,k}^1, \xi_{i,k}^1), \cdots, \ssum{k}{0}{K-1} \nabla F(\vw_{i,k}^M, \xi_{i,k}^M)
    \right],\\
    &\rmG_i^\zeta := \left[ \ssum{k}{0}{K-1}(\nabla F(\vw_{i,k}^1, \xi_{i,k}^1) - \nabla f_1(\vw_{i,k}^1)), \cdots,  \ssum{k}{0}{K-1}(\nabla F(\vw_{i,k}^M, \xi_{i,k}^M) - \nabla f_M(\vw_{i,k}^M))\right],\\
    &\rmG_i^\vartheta := \left[ 
    \ssum{k}{0}{K-1}( \nabla f_1(\vw_{i,k}^1) - \nabla f_1(\vw_{i,0}^1) ), \cdots, 
    \ssum{k}{0}{K-1}( \nabla f_M(\vw_{i,k}^M) - \nabla f_M(\vw_{i,0}^M) )
    \right],\\
    &\rmG_i := \left[ 
    K \nabla f_1 (\vw_{i,0}^1), \cdots, K \nabla f_M (\vw_{i,0}^M)
    \right],\\
    &\Bar{\rmG}_i := [K\nabla f_1(\Bar{\vw}_{i,0}), \cdots, K\nabla f_M(\Bar{\vw}_{i,0})],\\
    &\rmG^* := [K\nabla f_1(\vw^*), \cdots, K\nabla f_M(\vw^*)].
\end{align*}
And the following lemma is helpful for the computation of the deviation between different clients.

\begin{proof}[Proof of Lemma~\ref{lem:local_dev_dfedda_2}]
Recall the update rule. According to \eqref{eq:decentral_update_1} and \eqref{eq:decentral_update_2}, we can expand into the following equations,
\begin{align*}
    \Tilde{\rmZ}_{r,0} &= \Tilde{\rmZ}_{r-1,0}\rmU - \eta \rmG_{r-1}^\xi (\rmU - \rmJ_M)\\
    &= \cdots = \Tilde{\rmZ}_{r-\tau,0} \rmU^\tau - \eta \ssum{i}{r-\tau}{r-1} \rmG_i^\xi (\rmU^{r-i} - \rmJ_M).
\end{align*}
Note $\rmG_i^{\xi} = \rmG_i^\zeta + \rmG_i^\vartheta + (\rmG_i - \Bar{\rmG}_i) + (\Bar{\rmG}_i - \rmG^*) + \rmG^*$ and $\Tilde{\rmZ}_{r,0}\rmJ_M = 0$. Take $\tau = \frac{\log 4M}{2\log (1/\sigma_2(\rmU))}$, by the Cauchy-Schwartz inequality,
\begin{equation} \label{eq:T1_T5}
\begin{aligned}
    &\EB \norm{\Tilde{\rmZ}_{r,0}}_{p,2}^2 \le 2\eta^2 \EB \norm{\ssum{i}{r-\tau}{r-1} \left\{
    \rmG_i^\zeta + \rmG_i^\vartheta + (\rmG_i - \Bar{\rmG}_i) + (\Bar{\rmG}_i - \rmG^*) + \rmG^*
    \right\}(\rmU^{r-i} - \rmJ_M)  }_{p,2}^2\\
    &+ 2\EB \norm{\Tilde{\rmZ}_{r-\tau,0}\rmU^\tau}_{p,2}^2\\
    &\overset{(a)}{\le} 2\eta^2 \EB \norm{\ssum{i}{r-\tau}{r-1} \left\{
    \rmG_i^\zeta + \rmG_i^\vartheta + (\rmG_i - \Bar{\rmG}_i) + (\Bar{\rmG}_i - \rmG^*) + \rmG^*
    \right\}(\rmU^{r-i} - \rmJ_M) }_{p,2}^2\\
    &+ 2M \sigma_2(\rmU)^{2\tau}\EB \norm{\Tilde{\rmZ}_{r-\tau,0}}_{p,2}^2\\
    &\le \frac{1}{2}\EB \norm{\Tilde{\rmZ}_{r-\tau,0}}_{p,2}^2 + 10 \eta^2 \underbrace{\EB \norm{
    \ssum{i}{r-\tau}{r-1} \rmG_i^\zeta (\rmU^{r-i} - \rmJ_M)
    }_{p,2}^2}_{\gT_1} + 10\eta^2 \underbrace{\EB \norm{
    \ssum{i}{r-\tau}{r-1} \rmG_i^\vartheta(\rmU^{r-i} - \rmJ_M)
    }_{p,2}^2}_{\gT_2} \\
    &+ 10 \eta^2 \underbrace{\EB \norm{
    \ssum{i}{r-\tau}{r-1} (\rmG_i - \Bar{\rmG}_i)(\rmU^{r-i} - \rmJ_M)
    }_{p,2}^2}_{\gT_3} + 10\eta^2 \underbrace{\EB \norm{
    \ssum{i}{r-\tau}{r-1} (\Bar{\rmG}_i - \rmG^*)(\rmU^{r-i} - \rmJ_M)
    }_{p,2}^2}_{\gT_4}\\
    &+ 10\eta^2 \underbrace{\norm{\rmG^* \ssum{i}{r-\tau}{r-1}(\rmU^{r-i} - \rmJ_M)}_{p,2}^2}_{\gT_5}.
\end{aligned}
\end{equation}

Here (a) follows from Lemma~\ref{lem:mixing_rt_under_l_p2} and $\Tilde{\rmZ}_{r-\tau,0}\rmJ_M = 0$.
By using the combination of Cauchy-Schwartz inequality and Lemma~\ref{lem:db_stoc_under_l_p2}, we can bound $\gT_i,~ i = 2,3,4,5$ as follows.
\begin{align*}
    \gT_2 &\le \tau \ssum{i}{r-\tau}{r-1} \EB \norm{\rmG_i^\vartheta(\rmU^{r-i} - \rmJ_M)}_{p,2}^2
    \le 2\tau \ssum{i}{r-\tau}{r-1}(\EB \norm{\rmG_i^\vartheta \rmU^{r-i}}_{p,2}^2 + \EB \norm{\rmG_i^\vartheta \rmJ_M}_{p,2}^2)\\
    &\le 4\tau \ssum{i}{r-\tau}{r-1}\EB \norm{\rmG_i^\vartheta}_{p,2}^2 =
    4\tau \ssum{i}{r-\tau}{r-1}\ssum{m}{1}{M}\EB \norm{\ssum{k}{0}{K-1}(\nabla f_m(\vw_{i,k}^m) - \nabla f_m(\vw_{i,0}^m))}_p^2\\
    &\le 4d^{4/p}\tau L^2 K \ssum{i}{r-\tau}{r-1}\ssum{m}{1}{M}\ssum{k}{0}{K-1}\EB \norm{\vz_{i,k}^m - \vz_{i,0}^m}_p^2 = 4d^{4/p}\tau L^2 K \ssum{i}{r-\tau}{r-1}\gE_i.
\end{align*}

\begin{align*}
    \gT_3 &\le 2\tau \ssum{i}{r-\tau}{r-1} (\EB \norm{(\rmG_i - \Bar{\rmG}_i)\rmU^{r-i}}_{p,2}^2 + \EB \norm{(\rmG_i - \Bar{\rmG}_i)\rmJ_M}_{p,2}^2)\\
    &\le 4\tau \ssum{i}{r-\tau}{r-1} \EB \norm{\rmG_i - \Bar{\rmG}_i}_{p,2}^2 \le
    4d^{4/p} \tau L^2 K^2 \ssum{i}{r-\tau}{r-1}\ssum{m}{1}{M}\EB \norm{\vz_{i,0}^m - \Bar{\vz}_{i,0}^m}_p^2
    = 4d^{4/p}\tau L^2 K^2 \ssum{i}{r-\tau}{r-1}\gH_i.
\end{align*}

\begin{align*}
    \gT_4 &\le 4\tau K^2 \ssum{i}{r-\tau}{r-1}\left[
    \ssum{m}{1}{M}\EB \norm{\nabla f_m(\Bar{\vw}_{i,0}) - \nabla f_m(\vw^*)}_p^2
    \right]\\
    &\le 8d^{2/p}\tau L K^2 \ssum{i}{r-\tau}{r-1}\ssum{m}{1}{M} \EB\left\{
    f_m(\Bar{\vw}_{i,0}) - f_m(\vw^*) - \inner{\nabla f_m(\vw^*)}{\Bar{\vw}_{i,0} - \vw^*}
    \right\}\\
    &\le 8d^{2/p}\tau L M K^2 \ssum{i}{r-\tau}{r-1} \EB \left\{ f(\Bar{\vw}_{i,0}) - f(\vw^*) \right\}.
\end{align*}

\begin{align*}
    \gT_5 &\le 4\tau \ssum{i}{r-\tau}{r-1} \norm{\rmG^*}_{p,2}^2
    \le 4\tau^2 \norm{\rmG^*}_{p,2}^2 = 4\tau^2K^2 \ssum{m}{1}{M}\norm{\nabla f_m(\vw^*)}_p^2.
\end{align*}

Finally, we will use BDG-type inequality to control $\gT_1$.
\begin{align*}
    \gT_1 \le 2 \EB \norm{\ssum{t}{r-\tau}{r-1} \rmG_t^\zeta \rmU^{r-t}}_{p,2}^2 + 2\EB \norm{\ssum{i}{r-\tau}{r-1}\rmG_i^\zeta \rmJ_M}_{p,2}^2.
\end{align*}
We only present the analysis of the first term, while the second term can be bounded following the similar routine. For simplicity, we denote $\ssum{k}{0}{K-1}\zeta_{t,k}^i$ as $\zeta_{t}^i$, and denote $\ssum{i}{1}{M} u_{ij}^{(r-t)} [\zeta_t^i]_l$ as $[\Tilde{\zeta}_t^i]_l$.
\begin{align*}
    &\EB \norm{\ssum{t}{r-\tau}{r-1} \rmG_t^\zeta \rmU^{r-t}}_{p,2}^2 = \EB \ssum{j}{1}{M} \norm{\ssum{t}{r-\tau}{r-1} \ssum{i}{1}{M} u_{ij}^{(r-t)} \zeta_t^i}_p^2\\
    &\le \EB \ssum{j}{1}{M} \left( \ssum{l}{1}{d} \left| \ssum{t}{r-\tau}{r-1} \ssum{i}{1}{M} u_{ij}^{(r-t)} [\zeta_t^i]_l \right|^p \right)^{2/p}
    \le \ssum{j}{1}{M}\left( \ssum{l}{1}{d} \EB \left| \ssum{t}{r-\tau}{r-1} \ssum{i}{1}{M} u_{ij}^{(r-t)}[\zeta_t^i]_l\right|^p \right)^{2/p}\\
    &\le \ssum{j}{1}{M} \left\{ \ssum{l}{1}{d} C^p p^{3p/2} \EB \left( \ssum{t}{r-\tau}{r-1} \EB
    \left[([\Tilde{\zeta}_t^j]_l)^2| \gF_{t-1}\right]
    \right)^{p/2} \right\}^{2/p}\\
    &\le \ssum{j}{1}{M} \left\{
    \ssum{l}{1}{d} C^p p^{3p/2} \underbrace{\EB \left( 
    \ssum{t}{r-\tau}{r-1} \EB \left[
    \left(\ssum{i}{1}{M} u_{ij}^{(r-t)}\right) \left( \ssum{i}{1}{M} u_{ij}^{(r-t)} [\zeta_t^i]_l^2 \right) | \gF_{t-1}
    \right]
    \right)^{p/2}}_{\gX_l}
    \right\}^{2/p}.
\end{align*}
By Jensen's inequality,
\begin{align*}
    \gX_l &= \EB \left(
    \ssum{t}{r-\tau}{r-1} \ssum{i}{1}{M} u_{ij}^{(r-t)} \EB \left[[\zeta_t^i]_l^2 | \gF_{t-1}\right]
    \right)^{p/2} \le \tau^{\frac{p}{2} - 1} \ssum{t}{r-\tau}{r-1} \EB \left( 
    \ssum{i}{1}{M} u_{ij}^{(r-t)} \EB \left[[\zeta^i_t]_l^2 | \gF_{t-1}\right]
    \right)^{p/2}\\
    &\le \tau^{\frac{p}{2} - 1} \ssum{t}{r-\tau}{r-1} \ssum{i}{1}{M} u_{ij}^{(r-t)} \EB \left(
    \EB \left[ [\zeta_t^i]_l^2 | \gF_{t-1} \right]
    \right)^{p/2} \le \tau^{p/2} \mathop{\sup\limits}_{t\in[r-\tau, r-1]\atop i\in [M]} \EB \left(
    \EB \left[ [\zeta_t^i]_l^2 | \gF_{t-1} \right]
    \right)^{p/2}
\end{align*}
In addition,
\begin{align*}
    &\EB \left(
    \EB \left[ [\zeta_t^i]_l^2 | \gF_{t-1} \right]
    \right)^{p/2} = \EB \left( \EB \left[\left(\ssum{k}{0}{K-1} [\zeta_{t,k}^i]_l\right)^2| \gF_{t-1} \right] \right)^{p/2}\\
    \overset{(a)}{=}&
    \EB \left(
     \ssum{k}{0}{K-1} \EB [[\zeta_{t,k}^i]_l^2 | \gF_{t-1}]
    \right)^{p/2}\le
    K^{\frac{p}{2} - 1} \ssum{k}{0}{K-1} \EB \left|
    [\zeta_{t,k}^i]_l
    \right|^p \le
    K^{p/2}\frac{p!}{2^{p/2}(p/2)!}C^p\sigma^p.
\end{align*}
Here $(a)$ holds since that $[\zeta_{t,k}^i]_l$ is an $\gF_{t,k}$-martingale difference sequence given $\gF_{t-1}$. So we have
\[
\gX_l \le \frac{p!}{2^{p/2}(p/2)!}(C^2 \sigma^2\tau K )^{p/2}.
\]
Plugging this result into the bounds of $\gT_1$ yields
\begin{align*}
\gT_1 &\le 4 \ssum{j}{1}{M}\left\{ C^p p^{3p/2}d \frac{p!}{2^{p/2}(p/2)!}(C^2\sigma^2 \tau K)^{p/2}\right\}^{2/p} \le 4C^4p^{4}d^{2/p}\tau\sigma^2 MK.
\end{align*}
Finally, plugging the bounds on $\gT_1,\gT_2,\gT_3,\gT_4,\gT_5$ into \eqref{eq:T1_T5} concludes the proof.
\end{proof}

}

%% file: Proofs/proof_for_defedda-rt_conv.tex
\begin{proof}[Proof of Theorem \ref{thm:conv_rt_defedda_gt}]
With a little bit abuse of notation, we still use the notation and the definition in the proof of the original decentralized federated dual averaging scheme.

Under this symbolic system, the original algorithm can be re-expressed in the form of matrix calculations as follows.

\begin{enumerate}
    \item Local gradient update
    \[\rmZ_{r,k+1} = \rmZ_{r,k} - \eta_c (\rmG_{r,k}^\xi + \rmC_r)\]
    \item Update the gradient trackers on every local client
    \[\rmC_{r+1} = \rmC_r - \left(\frac{1}{K}\rmG_r^\xi + \rmC_r \right)(\rmI - \rmU) = \rmC_r \rmU - \frac{1}{K}{\rmG}_r^\xi (\rmI - \rmU).\]
    \item Communication according to the gossip matrix
    \[\rmZ_{r+1,0} = \rmZ_{r,K}\rmU\]
\end{enumerate}

Given Lemma~\ref{lem:local_dev_dfedda-gt_1} and Lemma~\ref{lem:local_dev_dfedda-gt_2}, set $2d^{4/p} L^2 \Tilde{\eta}^2 \le 1$. Then one can check
\begin{align*}
\label{eq:scaffolda gt Hr}
    \gH_r \le \frac{1}{2}\gH_{r-\tau} + 64d^{4/p}\tau L^2 \Tilde{\eta}^2 \ssum{i}{r-\tau}{r-1}\gH_i + 64\tau \Tilde{\eta}^2 \ssum{i}{r-\tau}{r-1} \gL_i + \frac{64 C^4 d^{2/p}p^4 \tau^2 \sigma^2 M\Tilde{\eta}^2}{K}.
\end{align*}

Similarly, the following recursive inequality on $\gL_r$ holds under the condition that $2ed^{4/p}L^2 \Tilde{\eta}^2 \leq 1$,
\begin{align}
    \gL_r &\le \frac{1}{4}\gL_{r-\tau} + 8d^{2/p} \tau ML^2 \ssum{i}{r-\tau}{r-1} \EB \norm{\Bar{\vw}_{i+1,0} - \Bar{\vw}_{i,0}}_1^2 + 64d^{4/p} \tau L^2 \ssum{i}{r-\tau}{r-1} \gH_i\\
    &+ 64ed^{4/p} \tau L^2 \Tilde{\eta}^2 \ssum{i}{r-\tau}{r-1} \gL_i + \frac{64C^4p^4 d^{2/p} \tau^2 \sigma^2 M}{K}. \label{eq:scaffolda gt Lr}
\end{align}

{
By combining \eqref{eq:scaffolda gt Hr} and \eqref{eq:scaffolda gt Lr}, and using $64d^{2/p}\tau L\Tilde{\eta}\leq 1$, we can derive
\begin{align*}
    &\gH_r + 2^{10} \tau^2 \Tilde{\eta}^2 \gL_r \le \frac{1}{2}\gH_{r-\tau} + \frac{2^{10}}{4}\tau^2 \Tilde{\eta}^2 \gL_{r-\tau} + 2^{17} d^{4/p}\tau^3 L^2 \Tilde{\eta}^2 \left(\ssum{i}{r-\tau}{r-1}\gH_i\right) + 128\tau \Tilde{\eta}^2 \left( \ssum{i}{r-\tau}{r-1} \gL_i \right)\\
    &+ 2^{17}\frac{C^4d^{2/p}p^4 \tau^4 \sigma^2 M\Tilde{\eta}^2}{K} + 2^{13}d^{2/p}\tau^3 ML^2 \Tilde{\eta}^2 \ssum{i}{r-\tau}{r-1}\EB \norm{\Bar{\vw}_{i+1,0} - \Bar{\vw}_{i,0}}_1^2.
\end{align*}
Summing up from $(s+1)\tau+1$ to $(s+1)\tau+\tau$,
\begin{align*}
    &\ssum{i}{1}{\tau} (\gH_{(s+1)\tau +i} + 2^{10} \tau^2 \Tilde{\eta}^2 \gL_{(s+1)\tau + i}) \le \frac{1}{2}\ssum{i}{1}{\tau} \left( \gH_{s\tau + i} + \frac{2^{10}}{2}\tau^2 \Tilde{\eta}^2 \gL_{s\tau + i} \right) + 2^{17} d^{4/p}\tau^3 L^2 \Tilde{\eta}^2 \ssum{i}{1}{\tau} \ssum{j}{s\tau + i}{(s+1)\tau + i -1}\gH_i \\
    &+ 128 \tau \Tilde{\eta}^2 \ssum{i}{1}{\tau} \ssum{j}{s\tau + i}{(s+1)\tau + i -1}\gL_i + 2^{17} \frac{C^4 d^{2/p} p^4 \tau^5 \sigma^2 M\Tilde{\eta}^2}{K} + 2^{13}d^{2/p}\tau^3 ML^2 \Tilde{\eta}^2 \ssum{i}{1}{\tau}\ssum{j}{s\tau + i}{(s+1)\tau + i -1}\EB \norm{\Bar{\vw}_{i+1,0} - \Bar{\vw}_{i,0}}_1^2\\
    &\le \frac{1}{2}\ssum{i}{1}{\tau} \left(\gH_{s\tau + i} + \frac{2^{10}}{2}\tau^2\Tilde{\eta}^2 \gL_{s\tau + i}\right) + 2^{17} d^{4/p}\tau^4 L^2 \Tilde{\eta}^2 \left\{\ssum{i}{1}{\tau - 1} \gH_{(s+1)\tau + i} + \ssum{i}{1}{\tau} \gH_{s\tau + i}\right\}\\
    &+ 128\tau^2 \Tilde{\eta}^2 \left\{ \ssum{i}{1}{\tau - 1}\gL_{(s+1)\tau + i} + \ssum{i}{1}{\tau} \gL_{s\tau + i} \right\} + \frac{2^{17} C^4 d^{2/p}p^4 \tau^5 \sigma^2 M\Tilde{\eta}^2 }{K}\\
    &+ 2^{13}d^{2/p}\tau^4 ML^2 \Tilde{\eta}^2 \left\{ \ssum{i}{1}{\tau - 1} \EB \norm{\Bar{\vw}_{(s+1)\tau + i + 1,0} - \Bar{\vw}_{(s+1)\tau + i,0}}_1^2 +  \ssum{i}{1}{\tau} \EB \norm{\Bar{\vw}_{s\tau + i + 1} - \Bar{\vw}_{s\tau + i}}_1^2\right\}.
\end{align*}
Let $\ssum{i}{1}{\tau} (\gH_{s\tau + i} + 2^{10}\tau^2 \Tilde{\eta}^2 \gL_{s\tau + i}) =: \gS_s^{1:\tau}$, $\ssum{i}{1}{\tau} \EB \norm{\Bar{\vw}_{s\tau + i + 1} - \Bar{\vw}_{s\tau + i}}_1^2 =: \Delta_s^{1:\tau}$.
Rearranging yields
\begin{align*}
    &(1 - 2^{17}d^{4/p}\tau^4 L^2 \Tilde{\eta}^2)\gS_{s+1}^{1:\tau} \le \left( \frac{1}{2} + 2^{17}d^{4/p}\tau^4 L^2 \Tilde{\eta}^2 \right)\gS_s^{1:\tau} + 2^{13} d^{2/p}\tau^4 ML^2 \Tilde{\eta}^2 \{\Delta_{s+1}^{1:\tau} + \Delta_s^{1:\tau}\}\\
    &+ \frac{2^{17}C^4d^{2/p}p^4 \tau^5 \sigma^2 M \Tilde{\eta}^2}{K}.
\end{align*}
If we let $2^{10}d^{2/p}\tau^2 L \Tilde{\eta} \le 1$, then
\begin{align*}
    \gS_{s+1}^{1:\tau} \le \frac{5}{7} \gS_{s}^{1:\tau} + 2^{14} d^{2/p}\tau^4 ML^2 \Tilde{\eta}^2 \{\Delta_{s+1}^{1:\tau} + \Delta_s^{1:\tau}\} + \frac{2^{18}C^4 d^{2/p} p^4 \tau^5 \sigma^2 M \Tilde{\eta}^2}{K}.
\end{align*}
}
By induction, we obtain
\begin{align*}
    \gS_s^{1:\tau} \leq \left(\frac{5}{7}\right)^s \gS_0^{1:\tau} + 2^{15} d^{2/p} \tau^4 ML^2 \Tilde{\eta}^2 \sum_{i=0}^s \left(\frac{5}{7}\right)^{s-1-i}\Delta_i^{1:\tau} + \frac{2^{20}C^4 d^{2/p} p^4 \tau^5 \sigma^2 M \Tilde{\eta}^2}{K}.
\end{align*}
If $R=b\tau$ then we have
\begin{align*}
    \sum_{s=0}^{b-1} \gS_s^{1:\tau} &\leq \sum_{s=0}^{b-1}\left(\frac{5}{7}\right)^s \gS_0^{1:\tau} + 2^{15} d^{2/p} \tau^4 ML^2 \Tilde{\eta}^2 \sum_{s=0}^{b-1} \sum_{i=0}^s \left(\frac{5}{7}\right)^{s-1-i}\Delta_i^{1:\tau} + \frac{2^{20}C^4 d^{2/p} p^4 \tau^4 \sigma^2 MR \Tilde{\eta}^2}{K} \\
    &\leq 4\gS_0^{1:\tau} + 2^{18} d^{2/p} \tau^4 ML^2 \Tilde{\eta}^2 \sum_{s=0}^{b-1} \Delta_s^{1:\tau} + \frac{2^{20}C^4 d^{2/p} p^4 \tau^4 \sigma^2 MR \Tilde{\eta}^2}{K},
\end{align*}
and consequently, using Lemma \ref{lem:local_dev_dfedda-gt_1},
\begin{align}
    \sum_{r=1}^{R} \gH_r &\leq \sum_{s=0}^{b-1} \gS_s^{1:\tau} \leq 4\gS_0^{1:\tau} + 2^{18} d^{2/p} \tau^4 ML^2 \Tilde{\eta}^2 \sum_{s=0}^{b-1} \Delta_s^{1:\tau} + \frac{2^{20}C^4 d^{2/p} p^4 \tau^4 \sigma^2 MR \Tilde{\eta}^2}{K}, \label{eq:descaffolda Hr}\\
    \sum_{r=1}^{R} \gL_r &\leq \frac{1}{2^{10}\tau^2\Tilde{\eta}^2}\sum_{s=0}^{b-1} \gS_s^{1:\tau} \leq \frac{\gS_0^{1:\tau}}{2^8\tau^2\Tilde{\eta}^2} + 2^{8} d^{2/p} \tau^2 ML^2 \sum_{s=0}^{b-1} \Delta_s^{1:\tau} + \frac{2^{10}C^4 d^{2/p} p^4 \tau^2 \sigma^2 MR}{K}, \label{eq:descaffolda Lr}\\
    \sum_{r=1}^{R} \gE_r &\leq 4eC^2 d^{2/p} p^4 MR\Tilde{\eta}^2 \sigma^2 + 4ed^{4/p}L^2 K \Tilde{\eta}^2 \sum_{r=1}^R\gH_r + 4eK \Tilde{\eta}^2 \sum_{r=1}^R\gL_r \\
    &\leq \frac{K}{8\tau^2}\gS_0^{1:\tau} + 2^{13}d^{2/p}\tau^2L^2MK\Tilde{\eta}^2 \sum_{s=0}^{b-1} \Delta_s^{1:\tau} + 2^{15}C^4d^{2/p}p^4\tau^2\sigma^2 MR\Tilde{\eta}^2. \label{eq:descaffolda Er}
\end{align}

Now, similar to the proof of Theorem \ref{thm:conv_rt_defedda}, we give an upper bound for $\ssum{m}{1}{M}(f_m(\bar{\vw}_{r,0}) - f_m(\vw^*))$. Define $\bar{\vc}_r:=\frac{1}{M}\ssum{m}{1}{M} \vc_r^m$. From the initialization and the update rule of $\vc_r^m$, we have $\bar{\vc}_0 = 0$ and $\bar{\vc}_{r+1}=\bar{\vc}_r$. Therefore, by replacing $\bar{\vg}_{r-1}^\xi$ with $\bar{\vg}_{r-1}^\xi + \bar{\vc}_{r-1}$ in the deduction of \eqref{eq:dfedda_1}-\eqref{eq:dfedda_last_term}, the resulting equation is still valid:
\begin{align*}
    &\frac{1}{MR}\ssum{r}{1}{R}\ssum{m}{1}{M}\EB [f_m(\Bar{\vw}_{r,0}) {-} f_m(\vw^*)] \leq \frac{h(\vw^*)}{R\Tilde{\eta}} {-} \frac{1}{8R\Tilde{\eta}}\ssum{r}{1}{R}\EB\norm{\Bar{\vw}_{r-1,0} {-} \Bar{\vw}_{r,0}}_q^2 {+} \frac{\Tilde{\eta}}{R}\ssum{r}{0}{R-1} \EB \norm{\Bar{\vg}_{r}^\xi - \Bar{\vg}_{r}}_p^2\\
    & + \frac{4L}{RMK}\ssum{r}{1}{R}\ssum{m}{1}{M}\ssum{k}{0}{K-1} \EB \norm{\Bar{\vw}_{r-1,0} - {\vw}_{r-1,k}^m}_1^2 \\
    &\leq \frac{h(\vw^*)}{R\Tilde{\eta}} - \frac{1}{8d^{2/p}R\Tilde{\eta}}\ssum{r}{1}{R}\EB\norm{\Bar{\vw}_{r-1,0} - \Bar{\vw}_{r,0}}_1^2 + \frac{C^2d^{2/p}p^4\sigma^2\Tilde{\eta}}{K} \\
    & + \frac{8d^{2/p}L}{RMK}\ssum{r}{1}{R}\ssum{m}{1}{M}\ssum{k}{0}{K-1} \left(\EB \norm{\vz_{r-1,k}^m - {\vz}_{r-1,0}^m}_p^2 + \EB \norm{\vz_{r-1,0}^m - \bar{\vz}_{r-1,0}}_p^2\right) \\
    &\leq \frac{h(\vw^*)}{R\Tilde{\eta}} - \frac{1}{8d^{2/p}R\Tilde{\eta}}\ssum{s}{0}{b-1}\Delta_s^{1:\tau} + \frac{C^2d^{2/p}p^4\sigma^2\Tilde{\eta}}{K} + \frac{8d^{2/p}L}{RMK}\ssum{r}{0}{R-1}(\gE_r+K\gH_r).
\end{align*}
Plugging \eqref{eq:descaffolda Hr}-\eqref{eq:descaffolda Er} into the above inequality yields
\begin{align*}
    &\frac{1}{MR}\ssum{r}{1}{R}\ssum{m}{1}{M}\EB [f_m(\Bar{\vw}_{r,0}) - f_m(\vw^*)] \\
    & \leq \frac{h(\vw^*)}{R\Tilde{\eta}} + \frac{33d^{2/p}L}{MR} + \left(\frac{2^{22}d^{4/p}\tau^4L^3\Tilde{\eta}^2}{R}-\frac{1}{8d^{2/p}R\Tilde{\eta}}\right)\sum_{s=0}^{b-1}\Delta_s^{1:\tau} + \frac{2^{24}C^4 d^{2/p}p^4\tau^2\sigma^2\Tilde{\eta}}{K} \\
    & \lesssim \frac{1}{R\Tilde{\eta}} + \frac{d^{2/p}p^4\tau^2\sigma^2\Tilde{\eta}}{K},
\end{align*}
provided that $d^{2/p}\tau^2L\Tilde{\eta} \lesssim 1$.

Finally, we obtain a convergence rate with respect to $\vw_{r,0}^m$ via similar arguments in the proof of Theorem \ref{thm:conv_rt_defedda}. Similar to \eqref{eq:dfedda_f-f2}, we have
\begin{align*}
    &\frac{1}{MR} \ssum{m}{1}{M} \ssum{r}{1}{R} \EB \{ f(\vw_{r,0}^m) - f(\Bar{\vw}_{r,0}) \} \\
    &\leq \frac{1}{MR} \ssum{m}{1}{M} \ssum{r}{1}{R} \left\{ \EB \{f(\Bar{\vw}_{r,0}) - f(\vw^*)\} + d^{2/p} L \EB \norm{\vw_{r,0}^m - \Bar{\vw}_{r,0}}_q^2 \right\} \\
    &\leq \frac{1}{R} \ssum{r}{1}{R} \EB \{f(\Bar{\vw}_{r,0}) - f(\vw^*)\} + \frac{d^{2/p}L}{MR} \ssum{r}{1}{R} \gH_r \\
    &\leq \frac{h(\vw^*)}{R\Tilde{\eta}} + \frac{33d^{2/p}L}{MR} + \left(\frac{2^{22}d^{4/p}\tau^4L^3\Tilde{\eta}^2}{R}-\frac{1}{8d^{2/p}R\Tilde{\eta}}\right)\sum_{s=0}^{b-1}\Delta_s^{1:\tau} + \frac{2^{24}C^4 d^{2/p}p^4\tau^2\sigma^2\Tilde{\eta}}{K} \\
    &\quad + \frac{d^{2/p}L}{MR} \left(4\gS_0^{1:\tau} + 2^{18} d^{2/p} \tau^4 ML^2 \Tilde{\eta}^2 \sum_{s=0}^{b-1} \Delta_s^{1:\tau} + \frac{2^{20}C^4 d^{2/p} p^4 \tau^4 \sigma^2 MR \Tilde{\eta}^2}{K}\right) \\
    &\lesssim \frac{1}{R\Tilde{\eta}} + \frac{d^{2/p}p^4\tau^2\sigma^2\Tilde{\eta}}{K}.
\end{align*}
If we choose $p=2\log d$ and $\Tilde{\eta} = \min \left\{ \frac{1}{\tau^2L}, \frac{K^{1/2}}{\tau R^{1/2}(\log d)^2 \sigma} \right\}$, then
\begin{align*}
    \frac{1}{M}\sum_{m=1}^{M} \E\{f(\hat{\vw}^m) - f(\vw^*)\} &\leq \frac{1}{MR} \ssum{m}{1}{M} \ssum{r}{1}{R} \EB \{ f(\vw_{r,0}^m) - f(\Bar{\vw}_{r,0}) \} \lesssim \frac{\tau^2L}{R} + \frac{\tau (\log d)^2 \sigma}{(RK)^{1/2}},
\end{align*}
which concludes the proof.
\end{proof}

\subsubsection{Proof for One-Step Deviation Lemmas for DFedDA-GT}\label{sec:prf_for_local_dev_dfed-gt}
\begin{proof}[Proof of Lemma~\ref{lem:local_dev_dfedda-gt_1}]
By using telescoping
\begin{align*}
    &\ssum{m}{1}{M}\EB \norm{\vz_{r,k}^m - \vz_{r,0}^m}_p^2 =
    \ssum{m}{1}{M}\eta^2\EB \norm{\ssum{i}{0}{k-1}(\zeta_{r,k}^m + \vartheta_{r,i}^m + \nabla f_m(\vw_{r,0}^m) + \vc_r^m)}_p^2 \\
    \le&
    4\eta^2 \ssum{m}{1}{M}\EB \norm{\ssum{i}{0}{k-1}\zeta_{r,i}^m}_p^2 + 4\eta^2 \ssum{m}{1}{M}\EB \norm{\ssum{i}{0}{k-1}\vartheta_{r,i}^m}_p^2 + 4k^2 \eta^2 \ssum{m}{1}{M} \EB\norm{\nabla f_m(\Bar{\vw}_{r,0}) + \vc_r^m}_p^2\\
    &+4k^2 \eta^2 \ssum{m}{1}{M}\norm{\nabla f_m(\vw_{r,0}^m) - \nabla f_m(\bar{\vw}_{r,0})}_p^2\\
    \le&
    4\eta^2 \ssum{m}{1}{M}\EB \norm{\ssum{i}{0}{k-1}\zeta_{r,i}^m}_p^2 + 4\eta^2 \ssum{m}{1}{M}\EB \norm{\ssum{i}{0}{k-1}\vartheta_{r,i}^m}_p^2 + 4d^{4/p}L^2 k^2 \eta^2 \gH_r
    + 4k^2 \eta^2 \gL_r.
\end{align*}
From the former results, we have
\[
\EB \norm{\ssum{i}{0}{k-1}\zeta_{r,i}^m}_p^2 \le C^2 d^{2/p}p^4 k\sigma^2,
\text{ and }
\EB \norm{\ssum{i}{0}{k-1}\vartheta_{r,i}^m}_p^2 \le d^{4/p}L^2 k \ssum{i}{0}{k-1} \EB \norm{\vz_{r,i}^m - \vz_{r,0}^m}_p^2.
\]

Thus, the following holds,
\begin{align*}
    &\ssum{i}{0}{k} \ssum{m}{1}{M} \EB \norm{\vz_{r,i}^m - \vz_{r,0}^m}_p^2 \le \ssum{i}{0}{k-1} \ssum{m}{1}{M} \EB \norm{\vz_{r,i}^m - \vz_{r,0}^m}_p^2 + 4C^2d^{2/p}p^4 M k\eta^2 \sigma^2\\
    &4d^{4/p}L^2 k \eta^2 \ssum{m}{1}{M}\ssum{i}{0}{k-1} \EB \norm{\vz_{r,i}^m - \vz_{r,0}^m}_p^2 +
    4d^{4/p}L^2 k^2 \eta^2 \gH_r + 4k^2 \eta^2 \gL_r\\
    &= (1 + 4d^{4/p}L^2 k \eta^2) \ssum{i}{0}{k-1} \ssum{m}{1}{M} \EB \norm{\vz_{r,i}^m - \vz_{r,0}^m}_p^2 + 4C^2d^{2/p}p^4 M k\eta^2 \sigma^2\\
    &+ 4d^{4/p}L^2 k^2 \eta^2 \gH_r + 4k^2 \eta^2 \gL_r.
\end{align*}
Making use of Lemma~\ref{lem:rcs_to_etr_fedda} can conclude the proof of the lemma.
\end{proof}

\begin{proof}[Proof of Lemma~\ref{lem:local_dev_dfedda-gt_2}]
    Let $\Tilde{\rmZ}_{r,0} = \rmZ_{r,0} - \Bar{\rmZ}_{r,0}$, then one has
    \begin{align*}
        \Tilde{\rmZ}_{r,0} = \Tilde{\rmZ}_{r-\tau,0}\rmU^\tau - \eta \ssum{i}{r-\tau}{r-1} (\rmG_i^\xi + K\rmC_i)(\rmU^{r-i} - \rmJ_M).
    \end{align*}
    So
    \begin{align*}
        &\EB \norm{\Tilde{\rmZ}_{r,0}}_{p,2}^2 \le 2\EB \norm{\Tilde{\rmZ}_{r-\tau,0}\rmU^i}_{p,2}^2\\
        &+ 2\eta^2 \EB \norm{
        \ssum{i}{r-\tau}{r-1} \left\{
    \rmG_i^\zeta + \rmG_i^\vartheta + (\rmG_i - \Bar{\rmG}_i) + (\Bar{\rmG}_i + K\rmC_i)
    \right\}(\rmU^{r-i} - \rmJ_M) 
        }_{p,2}^2\\
        &\overset{(a)}{\le} 2M\sigma_2(\rmU)^{2\tau} \EB \norm{\Tilde{\rmZ}_{r-\tau,0}}_{p,2}^2 + 8\eta^2 \EB \norm{
        \ssum{i}{r-\tau}{r-1} \rmG_i^\zeta(\rmU^{r-i} - \rmJ_M)}_{p,2}^2 + 8\eta^2 \EB \norm{\ssum{i}{r-\tau}{r-1}\rmG_i^\vartheta(\rmU^{r-i} - \rmJ_M)}_{p,2}^2\\
        &+ 8\eta^2 \EB \norm{\ssum{i}{r-\tau}{r-1} (\rmG_i - \Bar{\rmG}_i)(\rmU^{r-i} - \rmJ_M)}_{p,2}^2
        + 8\eta^2 \EB \norm{\ssum{i}{r-\tau}{r-1} (\Bar{\rmG}_i + K\rmC_i)(\rmU^{r-i} - \rmJ_M)}_{p,2}^2\\
        &\overset{(b)}{\le}
        2M\sigma_2(\rmU)^{2\tau} \EB \norm{\Tilde{\rmZ}_{r-\tau,0}}_{p,2}^2 + 32 C^4 d^{2/p} p^4 \tau \sigma^2 MK \eta^2  + 4d^{4/p}\tau L^2 K\eta^2 \ssum{i}{r-\tau}{r-1}\gE_i\\
        &+ 32 d^{4/p} \tau L^2 K^2 \eta^2 \ssum{i}{r-\tau}{r-1} \gH_i + 8\eta^2 \EB \norm{\ssum{i}{r-\tau}{r-1} (\Bar{\rmG}_i + K\rmC_i)(\rmU^{r-i} - \rmJ_M)}_{p,2}^2.
    \end{align*}
    Here $(a)$ holds for Lemma~\ref{lem:mixing_rt_under_l_p2}, and $(b)$ holds by performing the similar derivation in the proof of Lemma~\ref{lem:local_dev_dfedda_2}. Finally, as for the last term, we have
    \begin{align*}
    &\EB \norm{\ssum{i}{r-\tau}{r-1} (\Bar{\rmG}_i + K\rmC_i)(\rmU^{r-i} - \rmJ_M)}_{p,2}^2 \le 
    \tau \ssum{i}{r-\tau}{r-1}\EB \norm{(\Bar{\rmG}_i + K \rmC_i)(\rmU^{r-i} - \rmJ_M)}_{p,2}^2\\
    &\overset{(a)}{\le} 4\tau \ssum{i}{r-\tau}{r-1} \EB \norm{\Bar{\rmG}_i + K\rmC_i}_{p,2}^2
    = 4\tau K^2 \ssum{i}{r-\tau}{r-1} \gL_i,
    \end{align*}
    where $(a)$ can be showed for Lemma~\ref{lem:db_stoc_under_l_p2}. Given $\tau = \frac{\log 8M}{2 \log (1/\sigma_2(\rmU))}$, we obtain
    \begin{align*}
        \gH_r &\le \frac{1}{2} \gH_{r-\tau} + 4d^{4/p} \tau L^2 K\eta^2 \ssum{i}{r-\tau}{r-1} \gE_i 
        +32 d^{4/p}\tau L^2 K^2 \eta^2 \ssum{i}{r-\tau}{r-1}\gH_i\\
        &+ 32 \tau K^2 \eta^2 \ssum{i}{r-\tau}{r-1}\gL_i + 32C^4 d^{2/p} p^4 \tau \sigma^2 MK\eta^2.
    \end{align*}
\end{proof}

\begin{proof}[Proof of Lemma~\ref{lem:local_dev_dfedda-gt_3}]
    Based on the update rule of $\rmC_r$, we have
    \begin{align*}
        &\frac{1}{K}\Bar{\rmG}_{r} + \rmC_r = \left(\frac{1}{K}\Bar{\rmG}_{r-1} + \rmC_{r-1}\right) \rmU + \frac{1}{K}(\Bar{\rmG}_{r} - \Bar{\rmG}_{r-1})\\
        &- \frac{1}{K}\left\{ \rmG_{r-1}^\zeta + \rmG_{r-1}^\vartheta + (\rmG_{r-1} - \Bar{\rmG}_{r-1}) \right\}(\rmI - \rmU)\\
        &=
        \left( \frac{1}{K}\Bar{\rmG}_{r-\tau} + \rmC_{r-\tau} \right)\rmU^\tau + \frac{1}{K}\ssum{i}{r-\tau}{r-1}(\Bar{\rmG}_{i+1} - \Bar{\rmG}_i)\rmU^{r-i-1} - \frac{1}{K}\ssum{i}{r-\tau}{r-1}\rmG_{i}^\zeta (\rmI - \rmU)\rmU^{r-i-1}\\
        &- \frac{1}{K}\ssum{i}{r-\tau}{r-1} \rmG_{i}^\vartheta (\rmI - \rmU)\rmU^{r-i-1} -
        \frac{1}{K}\ssum{i}{r-\tau}{r-1} (\rmG_{i} - \Bar{\rmG}_i)(\rmI - \rmU)\rmU^{r-i-1}.
    \end{align*}
    Hence,
    \begin{align*}
        \gL_r &= \EB \norm{\frac{1}{K}\Bar{\rmG}_r + \rmC_r}_{p,2}^2 \le 
        2\EB\norm{\left(\frac{1}{K}\Bar{\rmG}_{r-\tau} + \rmC_{r-\tau}\right)\rmU^{\tau}}_{p,2}^2 + \frac{8}{K^2}\EB \norm{\ssum{i}{r-\tau}{r-1}(\Bar{\rmG}_{i+1} - \Bar{\rmG}_i)\rmU^{r-i-1}}_{p,2}^2\\ 
        &+ \frac{8}{K^2}\EB \norm{\ssum{i}{r-\tau}{r-1}\rmG_i^{\zeta}(\rmI - \rmU)\rmU^{r-i-1}}_{p,2}^2
        + \frac{8}{K^2} \EB \norm{\ssum{i}{r-\tau}{r-1} \rmG_i^\vartheta (\rmI - \rmU)\rmU^{r-i-1}}_{p,2}^2\\
        &+ \frac{8}{K^2}\EB \norm{\ssum{i}{r-\tau}{r-1}(\rmG_i - \Bar{\rmG}_i)(\rmI - \rmU)\rmU^{r-i-1}}_{p,2}^2
        =: \ssum{j}{1}{5} \gT_j.
    \end{align*}
    We bound the $\gT_j$ separately. First, by Lemma~\ref{lem:mixing_rt_under_l_p2}
    \begin{align*}
        \gT_1 = 2\EB \norm{\left( \frac{1}{K}\Bar{\rmG}_{r-\tau} + \rmC_{r-\tau} \right)\rmU^\tau}_{p,2}^2 \le 2M\sigma_2(\rmU)^{2\tau} \EB \norm{\frac{1}{K}\Bar{\rmG}_{r-\tau} + \rmC_{r-\tau}}_{p,2}^2 \le \frac{1}{4} \gL_{r-\tau}.
    \end{align*}
    And next, based on Lemma~\ref{lem:db_stoc_under_l_p2}, we have
    \begin{align*}
        &\gT_2 \le \frac{8\tau}{K^2}\ssum{i}{r-\tau}{r-1}\EB \norm{(\Bar{\rmG}_{i+1} - \Bar{\rmG}_i)\rmU^{r-i-1}}_{p,2}^2\le \frac{8\tau}{K^2}\ssum{i}{r-\tau}{r-1}\EB \norm{\Bar{\rmG}_{i+1} - \Bar{\rmG}_i}_{p,2}^2\\
        &= 8\tau \ssum{i}{r-\tau}{r-1}\ssum{m}{1}{M}\EB \norm{\nabla f_m(\Bar{\vw}_{i+1}) - \nabla f_m(\Bar{\vw}_i)}_p^2 \le 8d^{2/p}\tau M L^2 \EB \norm{\Bar{\vw}_{i+1} - \Bar{\vw}_i}_1^2;\\
        &\gT_4 \le \frac{8\tau}{K^2}\ssum{i}{r-\tau}{r-1}\EB \norm{\rmG_i^\vartheta(\rmI - \rmU)\rmU^{r-i-1}}_{p,2}^2\le \frac{16\tau}{K^2} \ssum{i}{r-\tau}{r-1} \EB \norm{\rmG_i^{\vartheta}}_{p,2}^2\\
        &= \frac{16\tau}{K^2}\ssum{i}{r-\tau}{r-1} \ssum{m}{1}{M} \EB \norm{\ssum{k}{0}{K-1}
        (\nabla f_m(\vw_{i,k}^m) - \nabla f_m(\vw_{i,0}^m))}_p^2\\
        &\le \frac{16\tau }{K}\ssum{i}{r-\tau}{r-1} \ssum{m}{1}{M}\ssum{k}{0}{K-1} \EB \norm{\nabla f_m(\vw_{i,k}^m) - \nabla f_m(\vw_{i,0}^m)}_p^2 \le
        \frac{16d^{4/p}\tau L^2}{K}\ssum{i}{r-\tau}{r-1} \gE_{i};\\
        &\gT_5 \le \frac{8\tau}{K^2} \ssum{i}{r-\tau}{r-1} \EB \norm{(\rmG_i - \Bar{\rmG}_i)(\rmI - \rmU)\rmU^{r-i-1}}_{p,2}^2 \le \frac{16\tau}{K^2}\ssum{i}{r-\tau}{r-1} \EB \norm{\rmG_i - \Bar{\rmG}_i}_{p,2}^2\\
        &= 16\tau \ssum{i}{r-\tau}{r-1} \ssum{m}{1}{M} \EB \norm{\nabla f_m(\vw_{i,0}^m) - \nabla f_m(\Bar{\vw}_{i,0})}_p^2 \le 16d^{4/p} \tau L^2 \ssum{i}{r-\tau}{r-1} \gH_i.
    \end{align*}
    As for $\gT_3$, like the proof of Lemma~\ref{lem:local_dev_dfedda_2},
    \begin{align*}
        \gT_3 \le \frac{16}{K^2} \left\{
        \EB \norm{\ssum{i}{r-\tau}{r-1}\rmG_i^\zeta\rmU^{r-i-1}}_{p,2}^2 + \EB \norm{\ssum{i}{r-\tau}{r-1} \rmG_i^\zeta \rmU^{r-i}}_{p,2}^2
        \right\}
        \le \frac{32C^4 p^4 d^{2/p} \tau \sigma^2 M}{K}.
    \end{align*}
    The proof is concluded by integrating these.
\end{proof}

%% file: theory/str_convrate.tex
\subsection{Convergence Rate for ReFedDA-GT and Multi-ReFedDA-GT}\label{sec:conv_rt & prf when str cvx}
\begin{theorem}[Convergence rate for ReFedDA-GT]\label{thm:conv_rt_recfda-gt}
    Suppose Assumption~\ref{asp:convex} - \ref{asp:grad_subG} hold. Let $p=2\log d$, $\eta_{server} = 1$ and $\eta_{client} \le \frac{1}{48LK}$, and denote $K\eta_{server}\eta_{client}$ as $\Tilde{\eta}$. Then the average iterate of Algorithm~\ref{alg:recfdagtr}, $\Bar{\vw}_R:=\frac{1}{R}\ssum{r}{0}{R-1} \vw_{r,0}$ satisfies that, with probability at least $1-\delta$,
    \begin{align}
        f(\bar{\vw}_R) - f(\vw^*) \le \frac{6Q\log d}{R\Tilde{\eta}} + \frac{3\Tilde{\eta}\sigma^2}{K}\left\{(\log d)^4 + \left(\frac{\log (2d/\delta)}{RM}\right)\vee \sqrt{\frac{1}{RM}\log \frac{2d}{\delta}}\right\}.
    \end{align}
    Furthermore, if we take $\eta_{client} = \min\left\{ \frac{1}{48LK}, \sqrt{\frac{2Q}{RK \sigma^2 \log d }}, \sqrt{\frac{2QM\log d}{K\sigma^2 \log(2d/\delta)}} \right\}$, then
    \begin{align}
    f(\bar{\vw}_R) - f(\vw^*) \lesssim \frac{LQ\log d}{R} + \sqrt{Q\sigma^2\log d}\left( 
\frac{(\log d)^2}{\sqrt{RK}} + \frac{\sqrt{\log(2d/\delta)}}{R\sqrt{MK}} \right).
    \end{align}
\end{theorem}


\begin{theorem}[Convergence rate for Multi-ReFedDA-GT]\label{thm:conv_rt_multi_recfda}
    Suppose Assumption~\ref{asp:convex} - \ref{asp:lsc} hold. For any given small positive real number $0< \delta \le 1$ and $\varepsilon > 0$, we design the hyperparameters of Algorithm~\ref{alg:multirecfda} as follows: 
    \begin{itemize}
        \item $N = \lceil \log_2(Q_0/\varepsilon)\rceil$;
        \item $R_n = \lceil 64s\kappa \log d \rceil,\quad \forall n \in [N]$;
        \item $K_n = \left\lceil\frac{2^{n+14} s^2 \sigma^2}{\mu(Q_0)^2 Q_0 R_n}\left\{(\log d)^4 + \frac{2\log (2nd/\delta)}{R_nM} \right\}\right\rceil,\quad \forall n\in [N]$.
    \end{itemize}
    Here $\kappa := \frac{L}{\mu(Q_0)}$ and $\norm{\vw_0 - \vw^*}_1^2 \le Q_0$. Then with probability at least $1-\delta$, the output of Algorithm~\ref{alg:multirecfda} $\hat{\vw}_N$ satisfies $\norm{\hat{\vw}_N - \vw^*}_1^2 \le \varepsilon$.
    As a corollary, at this time the communication complexity is $\gO\left( s\kappa \log d \log \frac{Q_0}{\varepsilon} \right)$ and the sample complexity is $\gO\left( \frac{s^2\sigma^2(\log d)^4}{\mu(Q_0)^2\varepsilon} + \frac{s\sigma^2 \log\left( \frac{d}{\delta}\log\frac{Q_0}{\varepsilon} \right)}{\mu(Q_0)LM\varepsilon\log d} \right)$.
\end{theorem}

\subsection{Proof for Restricted Centralized FedDA-GT}
\input{Proofs/proof_for_recfda_gt}

%% file: Proofs/proof_for_recfda_gt.tex
    Using a similar analysis, we can derive the following by performing Algorithm \ref{alg:recfdagtr}
    \begin{align*}
        &\frac{1}{MR}\ssum{m}{1}{M}\ssum{r}{1}{R}(f_m({\vw}_{r+1,0}) - f_m(\vw^*)) \le \frac{1}{R\Tilde{\eta}} \left\{
        h(\vw^* - \vw_0) - \frac{1}{2}\ssum{r}{1}{R}\norm{{\vw}_{r-1,0} - {\vw}_{r,0}}_q^2
        \right\}\\
        &+ \frac{2L}{RMK}\ssum{r}{1}{R} \ssum{m}{1}{M}\ssum{k}{0}{K-1} \norm{{\vw}_{r,0} - \vw_{r-1,k}^m}_1^2 + \frac{1}{R}\ssum{r}{1}{R}\inner{\Bar{\vg}_{r-1} - \Bar{\vg}_{r-1}^\xi}{ {\vw}_{r,0} - \vw^*}
    \end{align*}
with $\Bar{\vg}_r = \frac{1}{MK}\ssum{m}{1}{M}\ssum{k}{0}{K-1} \nabla f_m(\vw_{r,k}^m)$ and $\Bar{\vg}_r^\xi = \frac{1}{MK}\ssum{m}{1}{M}\ssum{k}{0}{K-1}\nabla F(\vw_{r,k}^m, \xi_{r,k}^m)$. And we do some preparation to bound $\frac{1}{R}\ssum{r}{1}{R}\inner{\Bar{\vg}_{r-1} - \Bar{\vg}_{r-1}^\xi}{{\vw}_{r,0} - \vw^*}$ and $\frac{4L}{RMK}\ssum{r}{1}{R}\ssum{m}{1}{M}\ssum{k}{0}{K-1}\norm{{\vw}_{r-1,k}^m - {\vw}_{r-1,0}}_1^2$ separately.

First, note 
\[\frac{1}{R}\ssum{r}{1}{R}\inner{\Bar{\vg}_{r-1} {-} \Bar{\vg}_{r-1}^\xi}{\vw_{r,0} - \vw^*} {=} \frac{1}{R}\ssum{r}{1}{R}\inner{\Bar{\vg}_{r-1} {-} \Bar{\vg}_{r-1}^\xi}{\vw_{r,0} {-} \vw_{r-1,0}} {+} \frac{1}{R}\ssum{r}{1}{R}\inner{\Bar{\vg}_{r-1} {-} \Bar{\vg}_{r-1}^\xi}{\vw_{r-1,0} {-} \vw^*}\]
and for the initial term,
\begin{align}\label{eq:innr_unmeasurable}
    &\frac{1}{R}\ssum{r}{1}{R}\inner{\Bar{\vg}_{r-1} - \Bar{\vg}_{r-1}^\xi}{\vw_{r,0} - \vw_{r-1,0}}
    \le \frac{1}{R}\ssum{r}{0}{R-1}\left\{\Tilde{\eta}\norm{\Bar{\vg}_{r} - \Bar{\vg}_{r}^\xi}_p^2 + \frac{1}{4\Tilde{\eta}}\norm{\vw_{r+1,0} - \vw_{r,0}}_q^2\right\}\\
    &\le \frac{\eta}{RMK}\ssum{r}{0}{R-1}\ssum{m}{1}{M}\norm{\ssum{k}{0}{K-1}\zeta_{r,k}^m}_p^2 + 
    \frac{1}{4R\Tilde{\eta}}\ssum{r}{0}{R-1}\norm{\vw_{r+1,0} - \vw_{r,0}}_q^2.
\end{align}
On the other hand, for the second term, we have
\begin{lemma}[Lemma 4 in \citep{duchi2012randomized}]\label{lem:conc_noise_2}
    Under Assumption \ref{asp:grad_subG}, and under the fact $\norm{\vw_{r,0} - \vw^*}_q^2 \le 4Q$, with probability at least $1-\delta$, we have
    \begin{align*}
        \frac{1}{RMK}\ssum{r}{0}{R-1}\ssum{m}{1}{M}\ssum{k}{0}{K-1}\inner{-\zeta_{r,k}^m}{\vw_{r,0} - \vw^*} \le 2\sqrt{Q}\sigma\sqrt{\frac{\log 1/\delta}{RMK}}.
    \end{align*}
\end{lemma}

For every $\norm{\vw_{r,k}^m - \vw_{r,0}}_1^2$, by using the strong convexity of the function $h$, we have $\norm{\vw_{r,k}^m - \vw_{r,0}}_1^2 \precsim \norm{\vz_{r,k}^m - \vz_{r,0}}_p^2$ and then
\begin{align*}
    &\norm{\vz_{r,k}^m - \vz_{r,0}}_p^2 = \eta^2 \norm{
    \ssum{i}{0}{k-1}\zeta_{r,i}^m + \ssum{i}{0}{k-1}\vartheta_{r,i}^m - \frac{k}{K}\ssum{l}{0}{K-1}\zeta_{r-1,l}^m - \frac{k}{K}\ssum{l}{0}{K-1}\vartheta_{r-1,l}^m + k \vartheta_{r-1,K} + k\vc_r
    }_p^2 \\
    &\le 6\eta^2 \left\{ \norm{\ssum{i}{0}{k-1}\zeta_{r,i}^m}_p^2 + \norm{\ssum{i}{0}{K-1}\zeta_{r-1,i}^m}_p^2 \right\} + 6\eta^2 \left\{ \norm{\ssum{i}{0}{k-1}\vartheta_{r,i}^m}_p^2 + \norm{\ssum{l}{0}{K-1}\vartheta_{r-1,l}^m}_p^2 \right\}\\
    &+ 6k^2 \eta^2 \norm{\vartheta_{r-1,K}^m}_p^2 + 6k^2 \eta^2 \norm{\vc_r}_p^2.
\end{align*}
For any $r, m, k$, we have $\norm{\ssum{i}{0}{k-1}\vartheta_{r,i}^m}_p^2 \le k \ssum{i}{0}{k-1} \norm{\vartheta_{r,i}^m}_p^2 =: \gS_{r,k}^m$. On the other hand,
\begin{align*}
    &\norm{\vc_r}_p^2 = \norm{\frac{1}{MK}\ssum{m}{1}{M}\ssum{k}{0}{K-1}\nabla F(\vw_{r-1,k}^m, \xi_{r-1,k}^m)}_p^2 = \frac{1}{M^2K^2}\norm{\ssum{m}{1}{M}\ssum{k}{0}{K-1}\left\{\zeta_{r-1,k}^m + \vartheta_{r-1,k}^m + \nabla f_m(\vw_{r-1,0})\right\}}_p^2\\
    &\le \frac{3}{M^2K^2}\norm{\ssum{m}{1}{M}\ssum{k}{0}{K-1}\zeta_{r-1,k}^m}_p^2 + \frac{3}{M^2 K^2}\norm{\ssum{m}{1}{M}\ssum{k}{0}{K-1}\vartheta_{r-1,k}^m}_p^2 + 3\norm{\nabla f(\vw_{r-1,0})}_p^2.
\end{align*}
So
\begin{align*}
    &\frac{1}{M}\ssum{m}{1}{M}\norm{\vz_{r,k}^m - \Bar{\vz}_r}_p^2 \le 6k\eta^2 \gS_{r,k} + 6K\eta^2 \gS_{r-1,K} + \frac{6\eta^2}{M} \ssum{m}{1}{M}\left\{ \norm{\ssum{i}{0}{k-1}\zeta_{r,i}^m}_p^2 + \norm{\ssum{i}{0}{K-1}\zeta_{r-1,i}^m}_p^2 \right\}\\ 
    &+  6k^2\eta^2\norm{\vartheta_{r-1,K}}_p^2 + \frac{18\eta^2}{M}\ssum{m}{1}{M}\norm{\ssum{k}{0}{K-1}\zeta_{r-1,k}^m}_p^2 + 18K\eta^2 \gS_{r-1,k} + 18k^2\eta^2 \norm{\nabla f(\vw_{r-1,0})}_p^2\\
    &\le 6k\eta^2 \gS_{r,k} + 24K\eta^2 \gS_{r-1,K} + \frac{24\eta^2}{M} \ssum{m}{1}{M}\left\{\norm{\ssum{i}{0}{k-1}\zeta_{r,i}^m}_p^2 + \norm{\ssum{i}{0}{K-1}\zeta_{r-1,i}^m}_p^2\right\}\\
    &+ 6k^2 \eta^2\norm{\vartheta_{r-1,K}}_p^2 + 18k^2 \eta^2 \norm{\nabla f(\vw_{r-1,0})}_p^2.
\end{align*}
Omitting some abuse of notation, denote $\frac{1}{MK}\ssum{m}{1}{M}\ssum{k}{0}{K-1}\norm{\vz_{r,k}^m - \Bar{\vz}_r}_p^2$ as $\gE_{r}$. By applying Lemma~\ref{lem:rcs_to_etr_fedda}, we have
\begin{align*}
    \gE_{r} &\le 24L^2 K^2 \eta^2 \gE_{r-1} + \frac{24\eta^2}{M}\ssum{m}{1}{M}\left\{\sup\limits_{0\le k \le K-1}\norm{\ssum{i}{0}{k-1}\zeta_{r,i}^m}_p^2 + \norm{\ssum{i}{0}{K-1}\zeta_{r-1,i}^m}_p^2\right\}\\
    &+ 6K^2 \eta^2 \norm{\vartheta_{r-1,K}}_p^2 + 18K^2 \eta^2 \norm{\nabla f(\vw_{r-1,0})}_p^2.
\end{align*}
If we let $48L^2K^2 \eta^2 \le 1$, then
\[
\gE_r \le \frac{1}{2}\gE_{r-1} {+} \frac{24\eta^2}{M}\ssum{m}{1}{M}\left\{\sup\limits_{0\le k \le K-1}\norm{\ssum{i}{0}{k-1}\zeta_{r,i}^m}_p^2 {+} \norm{\ssum{i}{0}{K-1}\zeta_{r-1,i}^m}_p^2\right\} {+} \norm{\vw_{r,0} {-} \vw_{r-1,0}}_q^2 {+} \frac{1}{8L^2}\norm{\nabla f(\vw_{r-1,0})}_p^2.
\]
Rearranging this formula, then we get
\begin{align}\label{eq:1_step_bd_recfda_gt}
    &\frac{4L}{R}\ssum{r}{1}{R}\gE_{r} \le \frac{8L}{R}\gE_0 + \frac{CL\eta^2}{RM}\ssum{r}{0}{R}\ssum{m}{1}{M}\sup\limits_{0\le k \le K-1}\norm{\ssum{i}{0}{k-1}\zeta_{r,i}^m}_p^2\\
    &+ \frac{4L}{R}\ssum{r}{0}{R-1}\norm{\vw_{r} - \vw_{r-1}}_q^2 + \frac{1}{2R}\ssum{r}{0}{R-1}\left\{f(\vw_{r,0}) - f(\vw^*)\right\}
\end{align}

We then use the following lemma to control the norm of the summation of the martingale difference sequences.
\begin{lemma}[Concentration of local cumulative noise]\label{lem:conc_noise_1}
    Under Assumption~\ref{asp:grad_subG}, with probability at least $1-\delta$, we have
    \begin{align*}
    \ssum{r}{0}{R}\ssum{m}{1}{M}\left\{\sup\limits_{0\le k \le K-1}\norm{\ssum{i}{0}{k-1}\zeta_{r,i}^m}_p^2\right\} \le \frac{CK\sigma^2}{d^{2/p}} \left\{d^{4/p}p^4RM + \left(\log \frac{d}{\delta}\right)\vee \sqrt{RM\log \frac{d}{\delta}}\right\}.
    \end{align*}
\end{lemma}
\begin{proof}[Proof of Lemma~\ref{lem:conc_noise_1}]
    We first show that for every $r,m \in [R]\times [M]$, $\sup\limits_{0\le k \le K-1}\norm{\ssum{i}{0}{k-1}\zeta_{r,i}^m}_p^2$ is sub-exponential. And in this step of proof, we abbreviate $\zeta_{r,i}^m$ as $\zeta_i$ for simplicity.

    Actually, given any $t > 0$, using Doob's inequality (note $\norm{\ssum{i}{0}{k-1}\zeta_i}_p^2$ is a sub-martingale and $e^{\lambda x}$ is a convex and increase function of $x$ so long as $\lambda > 0$),
    \begin{align*}
        &\PB\left( \sup\limits_{0\le k \le K-1}\norm{\ssum{i}{0}{k-1}\zeta_{i}}_p^2 \ge t \right) = 
        \PB\left( \sup\limits_{0\le k \le K-1}\exp\left\{\lambda\norm{\ssum{i}{0}{k-1}\zeta_{i}}_p^2\right\} \ge e^{\lambda t} \right)\\
        &\overset{(a)}{\le} e^{-\lambda t} \EB \exp \left\{
        \lambda \norm{\ssum{i}{0}{K-1}\zeta_i}_p^2
        \right\}\le e^{-\lambda t} \EB \sup\limits_{l\in [d]}\exp \left\{
        d^{2/p}\lambda \left(\ssum{i}{0}{K-1}[\zeta_i]_l\right)^2
        \right\}\\
        &\le e^{-\lambda t}\ssum{l}{1}{d}\EB \exp\left\{d^{2/p}\lambda \left(\ssum{i}{0}{K-1} [\zeta_i]_l \right)^2 \right\},
    \end{align*}
    
    where $(a)$ holds for all $\lambda$ such that $\gK := \EB \exp\left\{d^{2/p}\lambda \left(\ssum{i}{0}{K-1}[\zeta_i]_l\right)^2\right\}$ exist. Owing to Assumption~\ref{asp:grad_subG}, for any $\lambda \ge 0$, we have $\EB \exp \left\{\lambda\ssum{i}{0}{K-1} [\zeta_i]_l\right\} \le \exp\left\{\lambda K\sigma^2\right\}$. This implies $\ssum{i}{0}{K-1}[\zeta_i]_l$ is $\sqrt{K}\sigma$ sub-Gaussian. Further, $\left(\ssum{i}{0}{K-1}[\zeta_i]_l\right)^2$ is $K\sigma^2$ sub-exponential. So $\gK$ exists for all $\lambda \le \frac{d^{2/p}}{K\sigma^2}$. Plugging this into the above probability inequality yields
    \[
    \PB\left( \sup\limits_{0\le k \le K-1}\norm{\ssum{i}{0}{k-1}\zeta_{i}}_p^2 \ge t \right) \le d\min\limits_{\lambda \in \left[0, \frac{d^{2/p}}{K\sigma^2}\right]} e^{-\lambda t}\exp\left\{d^{2/p}\lambda K\sigma^2\right\} \le d \exp \left\{- \frac{d^{2/p}t}{K\sigma^2}\right\}.
    \]
    Now we obtain the sub-exponential nature of $\gJ_{r,m}:= \sup\limits_{0\le k\le K-1}\norm{\ssum{i}{0}{k-1}\zeta_{r,i}^m}_p^2$. Denote $\Tilde{\gJ}_{r,m}$ as $\gJ_{r,m} - \EB [\gJ_{r,m}| \gF_{r,0}]$. Then by leveraging Bernstein's inequality, for any $x \ge 0$,
    \[
    \PB\left(
    \ssum{r}{0}{R}\ssum{m}{1}{M} \Tilde{\gJ}_{r,m} \ge x
    \right) \le d \exp \left\{
    -\left(\frac{d^{4/p}x^2}{32e\sigma^4K^2RM}\right)\wedge 
    \left( \frac{d^{2/p}x}{4K\sigma^2} \right)
    \right\}.
    \]
    Simple algebraic computation can generate the following equivalent claim: With probability at least $1-\delta$, we have
    \begin{align}\label{eq:bd_sum_tilde_j_rm}
    \ssum{r}{0}{R}\ssum{m}{1}{M}\gJ_{r,m} \le
    \ssum{r}{0}{R}\ssum{m}{1}{M}\EB[\gJ_{r,m}| \gF_{r,0}] +
    \frac{9K\sigma^2}{d^{2/p}}\left\{\left(\log \frac{d}{\delta}\right)\vee \sqrt{RM\log \frac{d}{\delta}}\right\}.
    \end{align}
    As for the conditional expectations $\EB [\gJ_{r,m}| \gF_{r,0}]$, applying Lemma~\ref{lem:bd_of_p_norm_2nd_mmt} yields the following inequality holds almost surely for any $r,m$ pair,
    \begin{align}\label{eq:bd_sum_j_rm}
        \EB [\gJ_{r,m}| \gF_{r,0}] = \EB \left[\sup\limits_{0\le k\le K-1}\norm{\ssum{i}{0}{k-1}\zeta_{r,i}^m}_p^2| \gF_{r,0}\right] \le C^2 d^{2/p} p^4 K \sigma^2.
    \end{align}
    The proof is concluded by combining the above two inequalities \eqref{eq:bd_sum_tilde_j_rm} and \eqref{eq:bd_sum_j_rm}.

\end{proof}

    Integrating Lemma~\ref{lem:conc_noise_2},\ref{lem:conc_noise_1} and formula~\eqref{eq:innr_unmeasurable}, \eqref{eq:1_step_bd_recfda_gt} finally producing
    \begin{align*}
        &\frac{1}{R}\ssum{r}{1}{R}\left\{f(\vw_{r,0}) - f(\vw^*)\right\} \le \frac{h(\vw^* - \vw_0)}{R\Tilde{\eta}} + \frac{8L}{R}\gE_0 - \left(\frac{1}{4R\Tilde{\eta}} - \frac{4L}{R}\right)\ssum{r}{1}{R}\norm{\vw_{r-1,0} - \vw_{r,0}}_q^2\\
        &+ \frac{(1+ CL\Tilde{\eta})\Tilde{\eta}}{RMK^2}\ssum{r}{1}{R}\ssum{m}{1}{M}\sup\limits_{0\le k \le K-1}\norm{\ssum{i}{0}{k-1}\zeta_{r,i}^m}_p^2 + \frac{1}{2R}\ssum{r}{1}{R}\left\{f(\vw_{r,0}) - f(\vw^*)\right\} + 2\sqrt{\frac{\sigma^2 Q \log 2\delta}{RMK}}\\
        &\overset{(a)}{\le} \frac{2Q\log d }{R\Tilde{\eta}} + \frac{2\Tilde{\eta}\sigma^2}{d^{2/p}K}\left\{d^{4/p}p^4 + \left(\frac{\log (2d/\delta)}{RM}\right)\vee \sqrt{\frac{1}{RM}\log \frac{2d}{\delta}}\right\}\\
        &+ \frac{1}{2R}\ssum{r}{1}{R}\left\{f(\vw_{r,0}) - f(\vw^*)\right\} + 2\sqrt{\frac{\sigma^2 Q \log 2\delta}{RMK}}.
    \end{align*}
    Rearranging and note $\hat{\vw}_R = \frac{1}{R}\ssum{r}{1}{R}\vw_{r,0}$
    \begin{align*}
        f(\hat{\vw}_R) - f(\vw^*) \le \frac{6Q\log d}{R\Tilde{\eta}} + \frac{3\Tilde{\eta}\sigma^2}{K}\left\{d^{2/p}p^4 + \left(\frac{\log (2d/\delta)}{RM}\right)\vee \sqrt{\frac{1}{RM}\log \frac{2d}{\delta}}\right\}.
    \end{align*}

%% file: apendix_experiments.tex
\subsection{Details and Proofs for Experiments Design}
\subsubsection{Environments}
We use JAX framework to implement the algorithms in our work, and the
codes are run on Nvidia RTX Titan GPUs and Intel Xeon Gold 6132 CPUs with 252GB memory.

\subsubsection{Implementation Details.}
Since Fast-FedDA, C-FedDA, and MReFedDA-GT cannot be easily generalized to the decentralized case (a naive generalized version cannot converge), we will still use their original centralized versions in our experiments. Additionally, we will increase the batch size and communication rounds for evaluation by a factor of $M$.
We do not consider the MC-FedDA in \citep{bao2022fast} which cannot converge in our setting empirically.
We do not truncate the covariate and use standard Gaussian distribution instead, we have observed that this has minimal impact on the experimental results.

The code of our algorithms and decentralized optimization problems has been released on https://github.com/pengyang7881187/DFedDA-GT.
\subsubsection{Optimal Solution of the Examples}
\paragraph{Decentralized Sparse Linear Regression}
\begin{lemma}\label{lem:linear_optimal_solution}
    In the setting of decentralized linear regression, the global optimal solution $\vw^*$ satisfies $\vw^*=\frac{1}{M}\ssum{m}{1}{M}\tilde{\vw}^m$.
\end{lemma}
\begin{proof}[Proof of Lemma~\ref{lem:linear_optimal_solution}]
    It suffices to show that $\frac{1}{M}\ssum{m}{1}{M}\tilde{\vw}^m$ satisfies first order condition since the global objective function \ref{eq:linear_global_obj} is strongly convex.
    \begin{equation}\label{eq:linear_global_obj}
    \begin{aligned}
            &\frac{1}{M}\ssum{m}{1}{M}\EB_{\xi_m \sim \gD_m} F(\vw; \xi_m)=\frac{1}{M}\ssum{m}{1}{M}f_m(\vw)\\
            &=\frac{1}{2M}\EB_{\vx \sim \left(\mu,\Sigma\right),e\sim N(0, \sigma_2^2)}\ssum{m}{1}{M}\left(e+\inner{\vx}{\tilde{\vw}^m-\vw} \right)^2 \\
            &=\frac{1}{2}\left\{\sigma_2^2+\frac{1}{M}\ssum{m}{1}{M}\left(\tilde{\vw}^m-\vw\right)^\top \EB_{\vx \sim \left(\mu,\Sigma\right)}\left(\vx\vx^\top\right)\left(\tilde{\vw}^m-\vw\right)\right\}\\
            &=\frac{1}{2}\left\{\sigma_2^2+\frac{1}{M}\ssum{m}{1}{M}\left[\left([\tilde{\vw}^m]_1-[\vw]_1\right)^2+\sigma_1^2\norm{\tilde{\vw}^m_{-1}-\vw_{-1}}^2_2\right]\right\},
    \end{aligned}
    \end{equation}
    where $\vw_{-1}=([\vw]_2,\cdots,[\vw]_d)^\top$, and so is $\tilde{\vw}^m_{-1}$. Then the first order condition says $\vw^*=\frac{1}{M}\ssum{m}{1}{M}\tilde{\vw}^m$.
\end{proof}

\paragraph{Decentralized Sparse Logistic Regression}
In this part, we assume $\vx$ is sampled from an isotropic Gaussian for simplicity, while the proof remains valid for the isotropic distribution we use in practice.
We start with the case of $M=2$ for simplicity: In Lemma~\ref{lem:logit_two_optimal_solution}, we design non-sparse $\tilde{\vw}^1$ and $\tilde{\vw}^2$ such that $\vw^*$ is sparse.
\begin{lemma}\label{lem:logit_general_optimal_solution}
    In the setting of decentralized logistic regression, let $M=2I$, and for all $i=1,\cdots,I$, $\tilde{\vw}^i=\left({\bm{\alpha}^i}^\top,{\bm{\beta}^i}^\top\right)^\top,\tilde{\vw}^{-i}=\left({\bm{\alpha}^i}^\top,-{\bm{\beta}^i}^\top\right)^\top\in\RB^{s}\times\RB^{d-s}$. Then there exists a $\bm{\alpha}\in\text{span}\left\{\bm{\alpha}^1,\cdots,\bm{\alpha}^I\right\}$ such that the global optimal solution $\vw^*$ satisfies $\vw^*=(\bm{\alpha}^\top,\mathbf{0}_{d-s}^\top)^\top$.
\end{lemma}
\begin{lemma}\label{lem:logit_two_optimal_solution}
    In the setting of decentralized logistic regression, let $M=2$, $\tilde{\vw}^1=(\bm{\alpha}^\top,\bm{\beta}^\top)^\top\in\RB^{s}\times\RB^{d-s}$ and $\tilde{\vw}^2=(\bm{\alpha}^\top,-\bm{\beta}^\top)^\top\in\RB^{s}\times\RB^{d-s}$. Then there exists a $\gamma>0$ such that the global optimal solution $\vw^*$ satisfies $\vw^*=(\gamma\bm{\alpha}^\top,\mathbf{0}_{d-s}^\top)^\top$.
\end{lemma}
\begin{proof}[Proof of Lemma~\ref{lem:logit_two_optimal_solution}]
    It suffices to show that there exists a $\gamma>0$ such that $(\gamma\bm{\alpha}^\top,\mathbf{0}_{d-s}^\top)^\top$ satisfies first order condition since the global objective function \ref{eq:logit_global_obj} is strictly convex. 
    
    We state the global objective function and first order condition for general $M>0$ and $\left\{\tilde{\vw}^m\right\}_{m=1}^M$ first.
    \begin{equation}\label{eq:logit_global_obj}
    \begin{aligned}
            &\frac{1}{M}\ssum{m}{1}{M}\EB_{\xi_m \sim \gD_m} F(\vw; \xi_m)=\frac{1}{M}\ssum{m}{1}{M}f_m(\vw)\\
            &=-\frac{1}{M}\EB_{\vx \sim N\left(\mathbf{0}_{d},\sigma_1^2 \mathbf{I}_{d}\right)}\ssum{m}{1}{M}\EB_{y\sim\text{Ber}\left(\sigma(\inner{\vx}{\tilde{\vw}^m} \right)}\left\{y\log\sigma(\inner{\vx}{\vw})+(1-y)\log\sigma(-\inner{\vx}{\vw}) \right\}\\
            &=-\frac{1}{M}\EB_{\vx \sim N\left(\mathbf{0}_{d},\sigma_1^2 \mathbf{I}_{d}\right)}\ssum{m}{1}{M}\left\{\sigma(\inner{\vx}{\tilde{\vw}^m})\log\sigma(\inner{\vx}{\vw})+\sigma(-\inner{\vx}{\tilde{\vw}^m})\log\sigma(-\inner{\vx}{\vw}) \right\},
    \end{aligned}
    \end{equation}
the first order condition is
\begin{equation}\label{eq:logit_foc}
    \begin{aligned}
            0&=\EB_{\vx \sim N\left(\mathbf{0}_{d},\sigma_1^2 \mathbf{I}_{d}\right)}\ssum{m}{1}{M}\nabla_{\vw}\left\{\sigma(\inner{\vx}{\tilde{\vw}^m})\log\sigma(\inner{\vx}{\vw})+\sigma(-\inner{\vx}{\tilde{\vw}^m})\log\sigma(-\inner{\vx}{\vw}) \right\}\\
            &=\EB_{\vx \sim N\left(\mathbf{0}_{d},\sigma_1^2 \mathbf{I}_{d}\right)}\ssum{m}{1}{M}\left\{\sigma(\inner{\vx}{\tilde{\vw}^m})\frac{\sigma^\prime(\inner{\vx}{\vw})}{\sigma(\inner{\vx}{\vw})}-\sigma(-\inner{\vx}{\tilde{\vw}^m})\frac{\sigma^{\prime}(-\inner{\vx}{\vw})}{\sigma(-\inner{\vx}{\vw})} \right\}\vx\\
            &=\EB_{\vx \sim N\left(\mathbf{0}_{d},\sigma_1^2 \mathbf{I}_{d}\right)}\ssum{m}{1}{M}\left\{\sigma(\inner{\vx}{\tilde{\vw}^m})\sigma(-\inner{\vx}{\vw})-\sigma(-\inner{\vx}{\tilde{\vw}^m})\sigma(\inner{\vx}{\vw}) \right\}\vx\\
            &=\EB_{\vx \sim N\left(\mathbf{0}_{d},\sigma_1^2 \mathbf{I}_{d}\right)}\ssum{m}{1}{M}\left\{\sigma(\inner{\vx}{\tilde{\vw}^m})-\sigma(\inner{\vx}{\vw}) \right\}\vx\\
            &=\EB_{\vx \sim N\left(\mathbf{0}_{d},\sigma_1^2 \mathbf{I}_{d}\right)}\ssum{m}{1}{M}\left\{\tilde{\sigma}(\inner{\vx}{\tilde{\vw}^m})-\tilde{\sigma}(\inner{\vx}{\vw}) \right\}\vx,
    \end{aligned}
    \end{equation}
where we have used $\sigma^{\prime}(z)=\sigma(z)\sigma(-z)$, $\sigma(-z)=1-\sigma(z)$, and $\tilde{\sigma}(z):=\sigma(z)-\frac{1}{2}$ is the centralized sigmoid function, which is an odd function.

Back to the $M=2$ special case, we write $\vx=(\vx_{a}^\top,\vx_{b}^\top)^\top\in\RB^s\times\RB^{d-s}$ and the first order condition reads 
\begin{equation*}
    \begin{aligned}
            0&=\EB_{\vx \sim N\left(\mathbf{0}_{d},\sigma_1^2 \mathbf{I}_{d}\right)}\left\{\tilde{\sigma}\left(\inner{\vx}{\tilde{\vw}^1}\right)+\tilde{\sigma}\left(\inner{\vx}{\tilde{\vw}^2}\right)-2\tilde{\sigma}(\inner{\vx}{\vw}) \right\}\vx,
    \end{aligned}
\end{equation*}
substitute $\vx$ with $\vz=(\vx_a^\top,-\vx_b^\top)^\top$ in the term involving $\tilde{\vw}^2$,
\begin{equation*}
    \begin{aligned}
            \EB_{\vx \sim N\left(\mathbf{0}_{d},\sigma_1^2 \mathbf{I}_{d}\right)}\left\{\tilde{\sigma}\left(\inner{\vx}{\tilde{\vw}^2}\right)\vx\right\}&=\EB_{\vz \sim N\left(\mathbf{0}_{d},\sigma_1^2 \mathbf{I}_{d}\right)}\left\{\tilde{\sigma}(\inner{\vz_a}{\bm{\alpha}}+\inner{\vz_b}{\bm{\beta}})    \begin{pmatrix}
        \vz_a \\
        -\vz_b \\
    \end{pmatrix}  \right\}\\
    &=\EB_{\vx \sim N\left(\mathbf{0}_{d},\sigma_1^2 \mathbf{I}_{d}\right)}\left\{\tilde{\sigma}\left(\inner{\vx}{\tilde{\vw}^1}\right)\begin{pmatrix}
        \vx_a \\
        -\vx_b \\
    \end{pmatrix}   \right\},
    \end{aligned}
\end{equation*}
substituting back into the first-order conditions, we obtain
\begin{equation*}
    \begin{aligned}
            \EB_{\vx \sim N\left(\mathbf{0}_{d},\sigma_1^2 \mathbf{I}_{d}\right)}\left\{\tilde{\sigma}(\inner{\vx}{\vw}) \vx\right\}=\EB_{\vx \sim N\left(\mathbf{0}_{d},\sigma_1^2 \mathbf{I}_{d}\right)}\left\{\tilde{\sigma}\left(\inner{\vx}{\tilde{\vw}^1}\right)\begin{pmatrix}
        \vx_a \\
        0 \\
    \end{pmatrix}  \right\},
    \end{aligned}
\end{equation*}
or equivalently,
\begin{equation*}
    \begin{aligned}
            \begin{pmatrix}
        \EB_{\vx \sim N\left(\mathbf{0}_{d},\sigma_1^2 \mathbf{I}_{d}\right)}\left\{\tilde{\sigma}(\inner{\vx}{\vw}) \vx_a\right\} \\
        \EB_{\vx \sim N\left(\mathbf{0}_{d},\sigma_1^2 \mathbf{I}_{d}\right)}\left\{\tilde{\sigma}(\inner{\vx}{\vw}) \vx_b\right\} \\
    \end{pmatrix}=\begin{pmatrix}
        \EB_{\vx \sim N\left(\mathbf{0}_{d},\sigma_1^2 \mathbf{I}_{d}\right)}\left\{\tilde{\sigma}\left(\inner{\vx}{\tilde{\vw}^1}\right) \vx_a\right\} \\
        \mathbf{0}_{d-s} \\
    \end{pmatrix}.
    \end{aligned}
\end{equation*}
Note that for all $\vw_a\in\RB^s$, 
\begin{equation*}
    \begin{aligned}
    &\EB_{\vx \sim N\left(\mathbf{0}_{d},\sigma_1^2 \mathbf{I}_{d}\right)}\left\{\tilde{\sigma}\left(\inner{\vx}{(\vw_a^\top,\mathbf{0}_{d-s}^\top)^\top}\right) \vx_b\right\}\\
            &=\EB_{\vx \sim N\left(\mathbf{0}_{d},\sigma_1^2 \mathbf{I}_{d}\right)}\left\{\tilde{\sigma}(\inner{\vx_a}{\vw_a}) \vx_b\right\}\\
            &=\EB_{\vx_a \sim N\left(\mathbf{0}_{s},\sigma_1^2 \mathbf{I}_{s}\right)}\left\{\tilde{\sigma}(\inner{\vx_a}{\vw_a})\right\}\EB_{\vx_b \sim N\left(\mathbf{0}_{d-s},\sigma_1^2 \mathbf{I}_{d-s}\right)}\left\{\vx_b\right\}\\
            &=\mathbf{0}_{d-s},
    \end{aligned}
\end{equation*}
i.e. when $\vw_b=\mathbf{0}_{d-s}$, the second system of equations always holds true. 
Hence, we only need to verify there exists a $\gamma>0$ such that $\vw_a=\gamma\bm{\alpha}$ is a solution of the following equation:
\begin{equation*}
    \begin{aligned}
\EB_{\vx_a \sim N\left(\mathbf{0}_{s},\sigma_1^2 \mathbf{I}_{s}\right)}\left\{\tilde{\sigma}(\inner{\vx_a}{\vw_a}) \vx_a\right\}=\EB_{\vx \sim N\left(\mathbf{0}_{d},\sigma_1^2 \mathbf{I}_{d}\right)}\left\{\tilde{\sigma}\left(\inner{\vx}{\tilde{\vw}^1}\right) \vx_a\right\}.
    \end{aligned}
\end{equation*}
We define a map $G\colon\RB^s\to\RB^s$ to deal with LHS, $G(\vw_a):=\EB_{\vx_a \sim N\left(\mathbf{0}_{s},\sigma_1^2 \mathbf{I}_{s}\right)}\left\{\tilde{\sigma}(\inner{\vx_a}{\vw_a}) \vx_a\right\}$ for all $\vw_a\in\RB^s$. 
We will show in Lemma~\ref{lem:G_form} that $G$ can be written in the form:
\begin{equation*}
    \begin{aligned}
G(\vw_a)&=\frac{2}{\sigma_1}\left[\int_0^{+\infty}\tilde{\sigma}\left(t\norm{\vw_a}_2\right)\phi\left(\frac{t}{\sigma_1}\right)t~dt\right]\ve_{\vw_a}\\
&=2\sigma_1 \left[\int_0^{+\infty}\tilde{\sigma}\left(t\sigma_1\norm{\vw_a}_2\right)\phi\left(t\right)t~dt\right]\ve_{\vw_a},
    \end{aligned}
\end{equation*}
where $\ve_{\vw_a}:=\frac{\vw_a}{\norm{\vw_a}_2}$ is the unit vector of $\vw_a$ ($\ve_{\mathbf{0}_{s}}:=\mathbf{0}_{s}$), and $\phi(t)=\frac{1}{\sqrt{2\pi}}\exp\left\{-\frac{t^2}{2}\right\}$ is the probability density function of the standard normal distribution. Furthermore, the image of $G$ is the open ball $G\left(\RB^s\right)=B\left(\mathbf{0}_{s}, \sigma\sqrt{\frac{1}{2\pi}}\right)$.

We define a function $h\colon\RB\to\RB$ to deal with RHS, $h(z):=\EB_{\vx_b \sim N\left(\mathbf{0}_{d-s},\sigma_1^2 \mathbf{I}_{d-s}\right)}\left[\tilde{\sigma}\left(z+\inner{\vx_b}{\bm{\beta}}\right)\right]$, it is easy to verify that $h$ is odd, increasing and $h(+\infty)=\frac{1}{2}$, we can utilize Lemma~\ref{lem:G_form} again, 
\begin{equation*}
    \begin{aligned}
&\EB_{\vx \sim N\left(\mathbf{0}_{d},\sigma_1^2 \mathbf{I}_{d}\right)}\left\{\tilde{\sigma}\left(\inner{\vx}{\tilde{\vw}^1}\right) \vx_a\right\}\\
&=\EB_{\vx_a \sim N\left(\mathbf{0}_{s},\sigma_1^2 \mathbf{I}_{s}\right)}\left\{\vx_a\EB_{\vx_b \sim N\left(\mathbf{0}_{d-s},\sigma_1^2 \mathbf{I}_{d-s}\right)}\left[\tilde{\sigma}\left(\inner{\vx_a}{\bm{\alpha}}+\inner{\vx_b}{\bm{\beta}}\right)\right] \right\}\\
&=\EB_{\vx_a \sim N\left(\mathbf{0}_{s},\sigma_1^2 \mathbf{I}_{s}\right)}\left\{h(\inner{\vx_a}{\bm{\alpha}}) \vx_a \right\}\\
&=2\sigma_1\left[\int_0^{+\infty}h\left(t\sigma_1\norm{\bm{\alpha}}_2\right)\phi\left(t\right)t~dt\right]\ve_{\bm{\alpha}}\\
&\in B\left(\mathbf{0}_{s}, \sigma\sqrt{\frac{1}{2\pi}}\right)=G\left(\RB^s\right).
    \end{aligned}
\end{equation*}
By comparing the forms of the LHS and RHS, we can conclude that there exists a $\gamma$ > 0 such that $\vw_a = \gamma\bm{\alpha}$ is a solution to the first system of equations. 
To summarize, $\vw^*=(\gamma\bm{\alpha}^\top,\mathbf{0}_{d-s}^\top)^\top$ is the global optimal solution.



\end{proof}

\begin{proof}[Proof of Lemma~\ref{lem:logit_general_optimal_solution}]
Recall the first order condition we have derived,
\begin{equation*}
    \begin{aligned}
0&=\EB_{\vx \sim N\left(\mathbf{0}_{d},\sigma_1^2 \mathbf{I}_{d}\right)}\ssum{m}{1}{M}\left\{\tilde{\sigma}(\inner{\vx}{\tilde{\vw}^m})-\tilde{\sigma}(\inner{\vx}{\vw}) \right\}\vx\\
&=\EB_{\vx \sim N\left(\mathbf{0}_{d},\sigma_1^2 \mathbf{I}_{d}\right)}\ssum{i}{1}{I}\left\{\tilde{\sigma}\left(\inner{\vx}{\tilde{\vw}^i}\right)+\tilde{\sigma}\left(\inner{\vx}{\tilde{\vw}^{-i}}\right)-2\tilde{\sigma}(\inner{\vx}{\vw}) \right\}\vx,
    \end{aligned}
\end{equation*}
substitute $\vx$ with $\vz=(\vx_a^\top,-\vx_b^\top)^\top$ in the terms involving $\left\{\tilde{\vw}^{-i}\right\}_{i=1}^I$, we get the following two systems of equations,
\begin{equation*}
    \begin{aligned}
            \begin{pmatrix}
        \EB_{\vx \sim N\left(\mathbf{0}_{d},\sigma_1^2 \mathbf{I}_{d}\right)}\left\{\tilde{\sigma}(\inner{\vx}{\vw}) \vx_a\right\} \\
        \EB_{\vx \sim N\left(\mathbf{0}_{d},\sigma_1^2 \mathbf{I}_{d}\right)}\left\{\tilde{\sigma}(\inner{\vx}{\vw}) \vx_b\right\} \\
    \end{pmatrix}=\begin{pmatrix}
        \EB_{\vx \sim N\left(\mathbf{0}_{d},\sigma_1^2 \mathbf{I}_{d}\right)}\left\{\frac{1}{I}\ssum{i}{1}{I}\tilde{\sigma}\left(\inner{\vx}{\tilde{\vw}^i}\right) \vx_a\right\} \\
        \mathbf{0}_{d-s} \\
    \end{pmatrix}.
    \end{aligned}
\end{equation*}
Similarly, when $\vw_b=\mathbf{0}_{d-s}$, the second system of equations always holds true, and the first system of equations can be written in the following form:
\begin{equation*}
    \begin{aligned}
        G(\vw_a)=\frac{1}{I}\ssum{i}{1}{I}2\sigma_1\left[\int_0^{+\infty}h_i\left(t\sigma_1\norm{\bm{\alpha}^i}_2\right)\phi\left(t\right)t~dt\right]\ve_{\bm{\alpha}^i}\in B\left(\mathbf{0}_{s}, \sigma\sqrt{\frac{1}{2\pi}}\right)=G\left(\RB^s\right),
    \end{aligned}
\end{equation*}
where $h_i(z):=\EB_{\vx_b \sim N\left(\mathbf{0}_{d-s},\sigma_1^2 \mathbf{I}_{d-s}\right)}\left[\tilde{\sigma}\left(z+\inner{\vx_b}{\bm{\beta}^i}\right)\right]$.
To summarize, there exists a $\bm{\alpha}\in\RB^s$ such that $\vw^*=(\bm{\alpha}^\top,\mathbf{0}_{d-s}^\top)^\top$ is the global optimal solution. Furthermore, there exists a $\gamma>0$ such that $\bm{\alpha}=\gamma\ssum{i}{1}{I}\left[\int_0^{+\infty}h_i\left(t\sigma_1\norm{\bm{\alpha}^i}_2\right)\phi\left(t\right)t~dt\right]\ve_{\bm{\alpha}^i}\in\text{span}\left\{\bm{\alpha}^1,\cdots,\bm{\alpha}^I\right\}$.
\end{proof}

\begin{lemma}\label{lem:G_form}
Let $G\colon\RB^k\to\RB^k,~G(\vw):=\EB_{\vx \sim N\left(\mathbf{0}_{k},\sigma^2 \mathbf{I}_{k}\right)}\left\{g(\inner{\vx}{\vw}) \vx\right\}$ for all $\vw\in\RB^k$, where $k\in\mathbb{Z}_+,\sigma>0$ and $g\colon \RB\to\RB$ is an odd function. Then $G$ can be written in the form:
\begin{equation*}
    \begin{aligned}
G(\vw)&=\frac{2}{\sigma}\left[\int_0^{+\infty}g\left(t\norm{\vw}_2\right)\phi\left(\frac{t}{\sigma}\right)t~dt\right]\ve_{\vw}\\
&=2\sigma \left[\int_0^{+\infty}g\left(t\sigma\norm{\vw}_2\right)\phi\left(t\right)t~dt\right]\ve_{\vw},
    \end{aligned}
\end{equation*}
where $\ve_{\vw}:=\frac{\vw}{\norm{\vw}_2}$ ($\ve_{\mathbf{0}_{k}}:=\mathbf{0}_{k}$) is the unit vector of $\vw$, and $\phi(t)=\frac{1}{\sqrt{2\pi}}\exp\left\{-\frac{t^2}{2}\right\}$ is the probability density function of the standard normal distribution.

Furthermore, if $g$ is increasing and $G(+\infty)=D\geq 0$, then the image of $G$ is exactly the open ball $G\left(\RB^k\right)=B\left(\mathbf{0}_{k}, K\sigma\sqrt{\frac{2}{\pi}}\right)$ .
\end{lemma}
\begin{proof}[Proof of Lemma~\ref{lem:G_form}]
It is easy to verify that $G(\mathbf{0}_{k})=\mathbf{0}_{k}$, hence we will assume $\vw\neq\mathbf{0}_{k}$.

For $k=1$, $G(w)=\EB_{x \sim N\left(0,\sigma^2 \right)}\left\{g(xw) x\right\}=\frac{1}{\sigma}\left[\int_{-\infty}^{+\infty}g\left(xw\right)\phi\left(\frac{x}{\sigma}\right)x~dx\right]$, note that the integrand and $g$ are odd functions, we have $G(w)=\frac{2}{\sigma}\left[\int_{0}^{+\infty}g\left(t\abs{w}\right)\phi\left(\frac{t}{\sigma}\right)t~dt\right]\text{sgn}(w)$.

For $k>1$, we denote $p_{\sigma,l}(\cdot)$ as the probability density function of $N\left(\mathbf{0}_{l},\sigma^2 \mathbf{I}_{l}\right)$ for $l\in\ZB_+$.
We still utilize the property that $g$ is an odd function,
\begin{equation*}
    \begin{aligned}
&\int_{\vx\in\RB^k\colon\inner{\vx}{\vw}<0}p_{\sigma,k}(\vx)g(\inner{\vx}{\vw}) \vx~d\vx\\
&=\int_{\vz\in\RB^k\colon\inner{\vz}{\vw}>0}p_{\sigma,k}(-\vz)g(\inner{-\vz}{\vw}) (-\vz)~d\vz~(\vz=-\vx)\\
&=\int_{\vx\in\RB^k\colon\inner{\vx}{\vw}>0}p_{\sigma,k}(\vx)g(\inner{\vx}{\vw}) \vx~d\vx,
    \end{aligned}
\end{equation*}
note that $\int_{\vx\in\RB^k\colon\inner{\vx}{\vw}=0}p_{\sigma,k}(\vx)g(\inner{\vx}{\vw}) \vx~d\vx=0$, we have
\begin{equation*}
    \begin{aligned}
G(\vw)&=\int_{\RB^k}p_{\sigma,k}(\vx)g(\inner{\vx}{\vw}) \vx~d\vx\\
&=2\int_{\vx\in\RB^k\colon\inner{\vx}{\vw}>0}p_{\sigma,k}(\vx)g(\inner{\vx}{\vw}) \vx~d\vx.
    \end{aligned}
\end{equation*}
Let's take an orthonormal basis that includes $\ve_\vw$, denoted as $\mathbf{Q}=(\ve_\vw, \vq_2, ..., \vq_k)$, and substitute $\vz=\mathbf{Q}^\top \vx$. We will also express $\vz$ in the form of orthogonal decomposition as $\vz=[\vz]_1\ve_1+\vz_{-1}$, where $\ve_1=(1,0,\cdots,0)^\top$, $\vz_{-1}=(0,[\vz]_2,\cdots,[\vz]_k)^\top$, and $d\vz_{-1}:=d[\vz]_{2}\cdots ~d[\vz]_k$.
\begin{equation*}
    \begin{aligned}
G(\vw)&=2\int_{\vz\in\RB^k\colon[\vz]_1>0}p_{\sigma,k}\left(\mathbf{Q}\vz\right)g\left([\vz]_1\norm{\vw}_2\right) \mathbf{Q}\vz~d\vz \\
&=2\mathbf{Q}\int_0^{+\infty}p_{\sigma,1}([\vz]_1)g([\vz]_1\norm{\vw}_2) \int_{\RB^{k-1}} p_{\sigma,k-1}(\vz_{-1})\left([\vz]_1\ve_1+\vz_{-1}\right)~d\vz_{-1} ~d[\vz]_1\\
&=2\mathbf{Q}\ve_1\int_0^{+\infty}p_{\sigma,1}([\vz]_1)g([\vz]_1\norm{\vw}_2)[\vz]_1  ~d[\vz]_1\\
&=\frac{2}{\sigma}\left[\int_0^{+\infty}g\left(t\norm{\vw}_2\right)\phi\left(\frac{t}{\sigma}\right)t~dt\right]\ve_{\vw}\\
&=2\sigma \left[\int_0^{+\infty}g\left(t\sigma\norm{\vw}_2\right)\phi\left(t\right)t~dt\right]\ve_{\vw}.
    \end{aligned}
\end{equation*}
Furthermore, if $g$ is increasing and $g(+\infty)=D\geq 0$.
Let $\tilde{g}(z):=2\sigma \left[\int_0^{+\infty}g\left(t\sigma z\right)\phi\left(t\right)t~dt\right]$, it is easy to verify $\tilde{g}(0)=0$ and $\tilde{g}$ is continuous. 
By dominated convergence theorem, $\tilde{g}(+\infty)=2D\sigma\int_0^\infty \phi(t)t~dt=D\sigma\sqrt{\frac{2}{\pi}}$.
Hence, $G\left(\RB^k\right)=B\left(\mathbf{0}_{k}, D\sigma\sqrt{\frac{2}{\pi}}\right)$.
\end{proof}

\subsubsection{Assumptions are Satisfied in the Examples}\label{appendix:assumptions}
It is obvious that Assumption~\ref{asp:convex} (Convexity) is satisfied for both examples.  
\paragraph{Decentralized Sparse Linear Regression}
\begin{itemize}
    \item Assumption~\ref{asp:smooth} (Smoothness):$\quad\forall \vw,\vv\in\RB^d$,
\begin{equation*}
    \begin{aligned}
\norm{\nabla f_m(\vw) - \nabla f_m(\vv)}_\infty &= \max\left\{ \abs{[\vw]_1-[\vv]_1},\sigma_1^2\abs{[\vw]_2-[\vv]_2},\cdots,\sigma_1^2\abs{[\vw]_d-[\vv]_d}
 \right\}\\
 &\leq \tilde{\sigma}_1^2\norm{\vw - \vv}_1,
    \end{aligned}
\end{equation*}
where $\tilde{\sigma}_1:=\max\{\sigma_1,1\}$.
    \item Assumption~\ref{asp:grad_subG} (Sub-Gaussian): 
\begin{equation*}
    \begin{aligned}
\abs{\zeta_i^m(\vw)}&=\abs{[\nabla F(\vw, \xi_m)]_i - [\nabla f_m(\vw)]_i}\\
&\leq\abs{\inner{\vw-\tilde{\vw}^m}{\vx}-e}\abs{[\vx]_i}+ \max\{\sigma_1^2,1\}\abs{[\vw]_i-[\tilde{\vw}^m]_i}\\
&\leq \abs{\inner{\vw-\tilde{\vw}^m}{\vx}}\abs{[\vx]_i}+\abs{e}\abs{[\vx]_i}+ \max\{\sigma_1^2,1\}\abs{[\vw]_i-[\tilde{\vw}^m]_i}\\
&\leq \norm{\vw-\tilde{\vw}^m}_1\left(\norm{\vx}_\infty^2+\tilde{\sigma}^2_1 \right)+\norm{\vx}_\infty\abs{e}\\
&\leq \tilde{C}\left(C^2+\tilde{\sigma}^2_1\right)+C\abs{e}, 
    \end{aligned}
\end{equation*}
    hence $\zeta_i^m(\vw)$ is sub-Gaussian, here we assume $\norm{\vw-\tilde{\vw}^m}_1\leq\tilde{C}$. 
    \item Assumption~\ref{asp:lsc} (Locally Strong Convexity):
    
    $f_m$ is in fact strongly convex, $\nabla^2 f_m(\cdot)\succeq \bar{\sigma}_1^2\mathbf{I}_d$, where $\bar{\sigma}_1:=\min\{\sigma_1, 1\}$.
\end{itemize}
\paragraph{Decentralized Sparse Logistic Regression}
\begin{itemize}
    \item Assumption~\ref{asp:smooth} (Smoothness):$\quad\forall \vw,\vv\in\RB^d$,
\begin{equation*}
    \begin{aligned}
\norm{\nabla f_m(\vw) - \nabla f_m(\vv)}_\infty &= \norm{\EB\left\{\tilde{\sigma}\left(\inner{\vx}{\vw}\right)-\tilde{\sigma}\left(\inner{\vx}{\vv}\right)\right\}\vx}_\infty\\
 &\leq \EB \abs{\tilde{\sigma}\left(\inner{\vx}{\vw}\right)-\tilde{\sigma}\left(\inner{\vx}{\vv}\right)}\norm{\vx}_\infty\\
 &\leq \frac{1}{4}\EB\abs{\inner{\vx}{\vw-\vv}}\norm{\vx}_\infty\\
 &\leq\frac{1}{4}\norm{\vw-\vv}_1\EB\norm{\vx}_\infty^2\\
 &\leq \frac{C^2}{4}\norm{\vw-\vv}_1.
    \end{aligned}
\end{equation*}
    \item Assumption~\ref{asp:grad_subG} (Sub-Gaussian): 
\begin{equation*}
    \begin{aligned}
\abs{[\nabla F(\vw, \xi_m)]_i}&=\abs{y\sigma\left(-\inner{\vx}{\vw}\right)-(1-y)\sigma\left(\inner{\vx}{\vw}\right)}\abs{[\vx]_i}\\
&\leq \norm{\vx}_\infty\\
&\leq C,
    \end{aligned}
\end{equation*}
which is bounded, hence $\zeta_i^m(\vw)$ is sub-Gaussian.
    \item Assumption~\ref{asp:lsc} (Locally Strong Convexity):
    
    When $\norm{\vw}_1^2\leq Q$, $\abs{\inner{\vx}{\vw}}\leq \norm{\vx}_\infty\norm{\vw}_1\leq C\sqrt{Q}$, hence
\begin{equation*}
    \begin{aligned}
\nabla^2 f_m(\vw)&=\EB\left\{\sigma\left(\inner{\vx}{\vw}\right)\sigma\left(-\inner{\vx}{\vw}\right)\vx\vx^\top\right\}\\
&\succeq\sigma\left(C\sqrt{Q}\right)\sigma\left(-C\sqrt{Q}\right)\EB\left\{\vx\vx^\top\right\}\\
&=\sigma_1^2\sigma\left(C\sqrt{Q}\right)\sigma\left(-C\sqrt{Q}\right)\mathbf{I}_d.
    \end{aligned}
\end{equation*}
\end{itemize}